\documentclass[11pt]{amsart}
\usepackage{amsthm,amsmath,amssymb}
\usepackage{stmaryrd,mathrsfs}
\usepackage{enumerate}
\usepackage{enumitem}
\usepackage[T1]{fontenc}
\usepackage{esint}
\usepackage{tikz}
\usepackage{hyperref}
\usepackage{stackengine}
\usepackage{verbatim}
\usepackage[greek, ngerman, english]{babel}
\usepackage{csquotes}
\stackMath
\allowdisplaybreaks
\usepackage{geometry}
\usepackage{hyperref}
	\hypersetup{colorlinks, breaklinks,
            	linkcolor = blue,
				urlcolor = blue,
				anchorcolor = blue,
				citecolor = blue}

\usepackage[doi=true,isbn=true,sorting=nty,giveninits=true,maxbibnames=99,backend=biber]{biblatex}
\newcommand{\abs}[1]{|#1|}

\newcommand{\Norm}[2]{\|#1\|_{#2}}

\newcommand{\ave}[1]{\langle #1\rangle}
\newcommand{\BMO}[0]{\operatorname{BMO}}
\newcommand{\CMO}[0]{\operatorname{CMO}}

\newcommand{\supp}[0]{\operatorname{supp}}
\newcommand{\loc}[0]{\operatorname{loc}}
\DeclareMathOperator*{\esssup}{ess\;sup}
\DeclareMathOperator*{\essinf}{ess\;inf}

\newcommand{\R}{\mathbb{R}}
\newcommand{\C}{\mathbb{C}}
\newcommand{\N}{\mathbb{N}}
\newcommand{\Z}{\mathbb{Z}}

\newcommand{\eps}[0]{\varepsilon}

\newcommand{\dd}{\,\mathrm{d}}

\swapnumbers \numberwithin{equation}{section}

\theoremstyle{plain}
\newtheorem{theorem}{Theorem}[section]
\newtheorem{proposition}[theorem]{Proposition}

\newtheorem{lemma}[theorem]{Lemma}

\theoremstyle{definition}
\newtheorem{definition}[theorem]{Definition}

\newtheorem{remark}[theorem]{Remark}

\makeatletter
\@namedef{subjclassname@2020}{
  \textup{2020} Mathematics Subject Classification}
 \def\@textbottom{\vskip \z@ \@plus 1pt}
 \let\@texttop\relax
\makeatother

\makeatletter
\newcommand*{\mint}[1]{%
  \mint@l{#1}{}%
}
\newcommand*{\mint@l}[2]{%
  \@ifnextchar\limits{%
    \mint@l{#1}%
  }{%
    \@ifnextchar\nolimits{%
      \mint@l{#1}%
    }{%
      \@ifnextchar\displaylimits{%
        \mint@l{#1}%
      }{%
        \mint@s{#2}{#1}%
      }%
    }%
  }%
}
\newcommand*{\mint@s}[2]{%
  \@ifnextchar_{%
    \mint@sub{#1}{#2}%
  }{%
    \@ifnextchar^{%
      \mint@sup{#1}{#2}%
    }{%
      \mint@{#1}{#2}{}{}%
    }%
  }%
}
\def\mint@sub#1#2_#3{%
  \@ifnextchar^{%
    \mint@sub@sup{#1}{#2}{#3}%
  }{%
    \mint@{#1}{#2}{#3}{}%
  }%
}
\def\mint@sup#1#2^#3{%
  \@ifnextchar_{%
    \mint@sup@sub{#1}{#2}{#3}%
  }{%
    \mint@{#1}{#2}{}{#3}%
  }%
}
\def\mint@sub@sup#1#2#3^#4{%
  \mint@{#1}{#2}{#3}{#4}%
}
\def\mint@sup@sub#1#2#3_#4{%
  \mint@{#1}{#2}{#4}{#3}%
}
\newcommand*{\mint@}[4]{%
  \mathop{}%
  \mkern-\thinmuskip
  \mathchoice{%
    \mint@@{#1}{#2}{#3}{#4}%
        \displaystyle\textstyle\scriptstyle
  }{%
    \mint@@{#1}{#2}{#3}{#4}%
        \textstyle\scriptstyle\scriptstyle
  }{%
    \mint@@{#1}{#2}{#3}{#4}%
        \scriptstyle\scriptscriptstyle\scriptscriptstyle
  }{%
    \mint@@{#1}{#2}{#3}{#4}%
        \scriptscriptstyle\scriptscriptstyle\scriptscriptstyle
  }%
  \mkern-\thinmuskip
  \int#1%
  \ifx\\#3\\\else_{#3}\fi
  \ifx\\#4\\\else^{#4}\fi  
}
\newcommand*{\mint@@}[7]{%
  \begingroup
    \sbox0{$#5\int\m@th$}%
    \sbox2{$#5\int_{}\m@th$}%
    \dimen2=\wd0 %
    \let\mint@limits=#1\relax
    \ifx\mint@limits\relax
      \sbox4{$#5\int_{\kern1sp}^{\kern1sp}\m@th$}%
      \ifdim\wd4>\wd2 %
        \let\mint@limits=\nolimits
      \else
        \let\mint@limits=\limits
      \fi
    \fi
    \ifx\mint@limits\displaylimits
      \ifx#5\displaystyle
        \let\mint@limits=\limits
      \fi
    \fi
    \ifx\mint@limits\limits
      \sbox0{$#7#3\m@th$}%
      \sbox2{$#7#4\m@th$}%
      \ifdim\wd0>\dimen2 %
        \dimen2=\wd0 %
      \fi
      \ifdim\wd2>\dimen2 %
        \dimen2=\wd2 %
      \fi
    \fi
    \rlap{%
      $#5%
        \vcenter{%
          \hbox to\dimen2{%
            \hss
            $#6{#2}\m@th$%
            \hss
          }%
        }%
      $%
    }%
  \endgroup
}

\begin{document}

\title[Riesz--Kolmogorov and extrapolation of compactness]{Compact sets in quasi-Banach function spaces and multilinear extrapolation of compactness on $L^{p(\cdot)}(w)$}

\author{Spyridon Kakaroumpas}
\address{Institut f{\"u}r Mathematik, Julius-Maximilians-Universit{\"a}t
W{\"u}rzburg, Emil-Fischer-Str.~41, 97074 W{\"u}rzburg, Germany}
\email{spyridon.kakaroumpas@uni-wuerzburg.de}

\author{Stefanos Lappas}
\address{Mathematisches Institut, Rheinische Friedrich-Wilhelms-Universit{\"a}t Bonn, Endenicher Allee 60, 53115 Bonn, Germany}
\email{lappas@math.uni-bonn.de; vlappas@hotmail.com}

\thanks{}




\begin{abstract}
This paper addresses a novel Riesz--Kolmogorov theorem on quasi-Banach function spaces under mild conditions and the extrapolation of multilinear compact operators in the context of weighted variable Lebesgue spaces. We establish the latter result via our Riesz--Kolmogorov theorem which yields a weighted interpolation theorem for multilinear compact operators in the variable Lebesgue setting. In proving this, we also show a weighted interpolation theorem in mixed-norm variable Lebesgue spaces. By means of our extrapolation result, we obtain new weighted compactness estimates for the commutators of multilinear $\omega$-Calder\'{o}n--Zygmund operators, multilinear fractional integrals and multilinear Fourier multipliers on weighted variable Lebesgue spaces. Our work generalizes several recent ones, including but not limited to those of Cao, Olivo and Yabuta in the setting of multilinear operators acting on the classical weighted Lebesgue spaces as well as the previous result by the authors in the setting of bilinear operators and variable Lebesgue spaces.
\end{abstract}

\maketitle


\section{Introduction}
This paper develops two new compactness results that resolve significant gaps in the existing literature: the Riesz--Kolmogorov characterization of compactness on quasi-Banach function spaces and a extrapolation theorem for multilinear compact operators in the variable exponent setting. Let us begin by first recalling some historical facts related to the Riesz--Kolmogorov compactness criteria.

The theory of variable Lebesgue spaces has seen substantial progress over the past few years, with systematic treatments provided in \cite{DHHR2011, CF2013}. One important aspect of this theory concerns the characterization of totally bounded sets. In the classical constant $L^p$ Lebesgue spaces, relative compactness is characterized by the Riesz--Kolmogorov theorem, see, for instance, \cite{Kolmogorov1931, Tamarkin1932, Riesz1933, Tulajkov1933, Sudakov1957}. We remark that this criterion was further extended (when $0<p<1$) in \cite{Tsuji1951}. A weighted $L^p(v)$ version of Riesz--Kolmogorov theorem for $p>1$ was studied in \cite{ClopCruz2013}, and was further extended to the case $0<p<1$ (see \cite{XYY2021, CaoOlivoYabuta2022}). In recent years, Riesz--Kolmogorov type theorems have also been generalized to variable exponent $L^{p(\cdot)}$ Lebesgue spaces (see \cite{Rafeiro2009} in the Euclidean setting and \cite{GorkaMacios2015, BG2018} in the setting of metric measure spaces), to variable exponent $L^{p(\cdot),\lambda(\cdot)}$ Morrey spaces on a bounded doubling metric measure space (see \cite{BGGS2021}), to Banach/quasi-Banach function spaces (see \cite{GorkaRafeiro2016} and \cite{CGO2016, GuoZhao2020}) and to the setting of matrix weights (see \cite{LYZ2023}). 

Our first main result is the following new version of the Riesz--Kolmogorov criterion in the quasi-Banach function setting (see Section \ref{sec: preliminaries} for a detailed description of the definitions and notations used).

\begin{theorem}\label{thm: general RK criterion}
Let $X$ be a power quasi-BFS over $\R^n$ that has an absolutely continuous quasi-norm such that $\chi_B\in X$ for all balls $B\subset\R^n$. Moreover, let $x_0\in\R^n$, $b\in(0,\infty)$ and assume that $M_b:X\rightarrow X$ is bounded, where $M_{b}(f)(x):=\sup\limits_{r>0}\left(\displaystyle \mint{-}_{B(x,r)}|f(y)|^{b}\; \dd y\right)^{1/b}$ is the $b$-Hardy--Littlewood maximal operator. 
Then, the family $\mathcal{F}\subset X$ is totally bounded in $X$ (or equivalently precompact/relatively compact) if and only if the following conditions are satisfied:
\begin{enumerate}[label=\textup{(\roman*)}]
\item $\mathcal{F}$ is uniformly bounded, i.e.,
  \begin{equation*}
   \sup_{f\in\mathcal{F}}\|f\|_{X}<\infty;     
  \end{equation*}\label{item: (i)}
\item $\mathcal{F}$ is $b$-equicontinuous, i.e.,
  \begin{equation*}
  \lim_{r\rightarrow 0}\sup_{f\in\mathcal{F}}\left\|\left(\displaystyle \mint{-}_{B(\cdot,r)}|f(\cdot)-f(y)|^{b}\dd y\right)^{1/b}\right\|_{X}=0 ;\end{equation*}\label{item: (ii)}
\item $\mathcal{F}$ uniformly vanishes at infinity, i.e.,
  \begin{equation*}
  \lim_{R\rightarrow\infty}\sup_{f\in\mathcal{F}}\|f\chi_{\R^n\setminus B(x_0,R)}\|_{X}=0. 
  \end{equation*}\label{item: (iii)}
\end{enumerate}
\end{theorem}

As an immediate corollary of Theorem \ref{thm: general RK criterion}, we obtain the weighted variable $L^{p(\cdot)}(w)$  Riesz--Kolmogorov theorem. To state it we introduce the following definition of variable Muckenhoupt $\mathcal{A}_{p(\cdot)}$ class of weights (see also Section \ref{sec: preliminaries} for the notations).
\begin{definition}[\cite{CruzUribe2011}, (1.3) in page 365]
Given an exponent $p(\cdot)\in\mathscr{P}$ and weight $w$, we say that $w\in\mathcal{A}_{p(\cdot)}$ if     
\begin{equation*}
  [w]_{\mathcal{A}_{p(\cdot)}}:=\sup_{Q} |Q|^{-1}\Vert w\chi_{Q}\Vert_{p(\cdot)}\Vert w^{-1}\chi_{Q}\Vert_{p'(\cdot)}<\infty,   
\end{equation*}
where the supremum is taken over all cubes $Q\subset\R^n$.
\end{definition}
Notice that $w\in\mathcal{A}_{p(\cdot)}$ if and only if $u=w^{p(\cdot)}\in A_{p(\cdot)}$, with $[w]_{\mathcal{A}_{p(\cdot)}}=[u]_{A_{p(\cdot)}}$. Here, we define $[u]_{A_{p(\cdot)}}:=\sup\limits_{Q} |Q|^{-1}\Vert u^{\frac{1}{p(\cdot)}}\chi_{Q}\Vert_{p(\cdot)}\Vert u^{-\frac{1}{p(\cdot)}}\chi_{Q}\Vert_{p'(\cdot)}<\infty$ and this is the non-symmetric definition of the variable Muckenhoupt class (see \cite[page 364]{CruzUribe2011}). By choosing $p(\cdot)=p=\text{constant}$ in the latter definition, one recovers the classical Muckenhoupt $A_p$ class of weights introduced in \cite{Muckenhoupt1972}. 
 
\begin{theorem}\label{thm: RK criterion}
Let $p(\cdot)\in\mathscr{P}_0\cap\mathrm{LH}$, $\widetilde{q}\in(0,p_{-})$,  $x_0\in\R^n$ and $w^{\widetilde{q}}\in\mathcal{A}_{\frac{p(\cdot)}{\widetilde{q}}}$. Then, the family $\mathcal{F}\subset L^{p(\cdot)}(w)$ is totally bounded (or equivalently precompact/relatively compact) if and only if the following conditions are satisfied:
\begin{enumerate}[label=\textup{(\roman*)}]
\item $\mathcal{F}$ is uniformly bounded, i.e.,
  \begin{equation*}
   \sup_{f\in\mathcal{F}}\|f\|_{L^{p(\cdot)}(w)}<\infty     
  \end{equation*}
  or equivalently, there exists $M>0$ such that for all $f\in\mathcal{F}$
  \begin{equation*}
  \sup_{f\in\mathcal{F}}\int_{\R^n}|f(x)|^{p(x)}\;w(x)^{p(x)} \dd x\leq M;
  \end{equation*}\label{itm: (i)}
\item $\mathcal{F}$ is $\widetilde{q}$-equicontinuous, i.e.,
  \begin{equation*}
  \lim_{r\rightarrow 0}\sup_{f\in\mathcal{F}}\left\|\left(\displaystyle \mint{-}_{B(\cdot,r)}|f(\cdot)-f(y)|^{\widetilde{q}}\dd y\right)^{1/\widetilde{q}}\right\|_{L^{p(\cdot)}(w)}=0     
  \end{equation*}
  or equivalently,
  \begin{equation*}
  \lim_{r\rightarrow 0}\sup_{f\in\mathcal{F}}\int_{\R^n}\left(\displaystyle \mint{-}_{B(x,r)}|f(x)-f(y)|^{\widetilde{q}}\dd y\right)^{p(x)/\widetilde{q}}w(x)^{p(x)}\dd x=0;  
  \end{equation*}\label{itm: (ii)}
\item $\mathcal{F}$ uniformly vanishes at infinity, i.e.,
  \begin{equation*}
  \lim_{R\rightarrow\infty}\sup_{f\in\mathcal{F}}\|f\chi_{\R^n\setminus B(x_0,R)}\|_{L^{p(\cdot)}(w)}=0 
  \end{equation*}
  or equivalently,
  \begin{equation*}  \lim_{R\rightarrow\infty}\sup_{f\in\mathcal{F}}\int_{\R^n\setminus B(x_0,R)}|f(x)|^{p(x)}\;w(x)^{p(x)}\dd x=0.
  \end{equation*}\label{itm: (iii)}
\end{enumerate}
\end{theorem}

\begin{remark}\label{rmk: main 1}
\leavevmode
\begin{enumerate}[label=\rm{(\arabic*)}]
  \item By choosing $\widetilde{q}=1$ in Theorem \ref{thm: RK criterion}, one notices that this result is very similar to the one stated in \cite[Theorem 3.2]{BG2018}.
  \item Theorem \ref{thm: RK criterion} extends the weighted Fr\'{e}chet--Kolmogorov compactness theorem of \cite[Theorem 2.10]{CaoOlivoYabuta2022} to the setting of weighted variable exponent Lebesgue spaces. Indeed, by taking $p(\cdot)=p=\text{constant}\in(0,\infty)$, $p_0:=\frac{p}{\widetilde{q}}\in(1,\infty)$ and $x_0=0$ in Theorem \ref{thm: RK criterion}, one sees that $w^{\widetilde{q}}\in\mathcal{A}_{p_0}$ if and only if $u:=(w^{\widetilde{q}})^{p_0}=w^p\in A_{p_0}$, thereby recovering exactly the statement of \cite[Theorem 2.10]{CaoOlivoYabuta2022}. Additionally, setting $p(\cdot)=p=\text{constant}\in(1,\infty)$, $\widetilde{q}=1$ and $x_0=0$ in Theorem \ref{thm: RK criterion}, and with similar reasoning as before, one recovers \cite[Theorem 2.9]{CaoOlivoYabuta2022}.
\end{enumerate}
\end{remark}

Let us briefly outline the core mechanism behind the proof of Theorem \ref{thm: general RK criterion}, which will be presented in Section \ref{sec: char. cmp.}. In particular, it is inspired by \cite[Theorems 2.4 and 3.1]{BG2018} and relies on two fundamental ingredients: Lemma \ref{lem: X-Krotov's result} (that is, any almost uniformly bounded and almost equicontinuous subset of $L^0(B(x_0,R))$ is totally bounded) and Lemma \ref{lem: X-LV cmp. thm} (Lebesgue--Vitali compactness criterion). Compared to \cite[Theorems 2.4 and 3.1]{BG2018}, we are forced to develop new ideas. We highlight the particular case of Theorem \ref{thm: general RK criterion}, namely, Theorem \ref{thm: RK criterion} which will imply an interpolation of compactness result (see Theorem \ref{thm:interpolation_compactness}) that will be essential in order to conclude the proof of our second main result below.

Having stated our first main result, we now proceed with the main extrapolation result for multilinear compact operators on weighted variable Lebesgue spaces. We first provide a brief historical overview of this topic.

One of the fundamental tools for proving weighted $L^p$ boundedness results in harmonic analysis is the extrapolation theorem of Rubio de Francia \cite{Rubio_de_Francia_1984}. A systematic study of this theory for linear and multilinear operators on weighted Lebesgue spaces is presented in \cite{Book_Extrapolation} and \cite{GrafakosMartell2004, CM2018, Nieraeth_2019, LMO2020, LMMOV2020}, respectively. There is now a well-developed extrapolation theory for variable exponent spaces in both the linear and multilinear settings; see \cite{CW2017, CaoMarinMartell2022, Nieareth2023}. Nevertheless, it remains an open problem whether one can establish the precise linear and multilinear counterparts of Rubio de Francia’s extrapolation of boundedness on weighted variable Lebesgue spaces (see \cite[Remark 1.7]{LO2025}).

In the past few years, the study of extending extrapolation results to compact operators has been the focus of considerable research activity. In the linear setting, the first work addressing this problem was obtained by Hyt\"onen and the second named author in \cite{HL2023}. Subsequently, several extensions have been developed in different settings, including weighted Morrey spaces \cite{Lappas2022}, two-weight inequalities \cite{Liu2022}, Banach function spaces \cite{Lorist2024} (see also \cite{LO2025}, where an alternative proof for the compactness of linear operators is provided in the framework of weighted variable Lebesgue spaces), and one-sided situations \cite{MartinReyes2025}. Furthermore, the first extrapolation results for multilinear compact operators in the context of classical weighted Lebesgue spaces were obtained by Cao, Olivo, and Yabuta in \cite{CaoOlivoYabuta2022} (see also \cite{HL2022} by Hyt\"onen and the second named author). The latter was further extended to bilinear operators acting on weighted variable Lebesgue spaces by us \cite{extr_comp_var_bil}. Motivated by the aforementioned works, our second main result is the following abstract extrapolation of compactness principle for multilinear operators on weighted variable Lebesgue spaces.

\begin{theorem}
\label{thm:main_result_extr}
Let $\Theta$ be a collection of pairs $(\vec{Y},Y)$ with $\vec{Y}=(Y_1,\ldots,Y_m)$ a $m$-tuple of Banach spaces and $Y$ a quasi-Banach space. Let $T$ be a $m$-linear operator defined and
\begin{equation*}
  \text{bounded}\quad T:Y_1\times\ldots\times Y_m\to Y\quad \text{for \textbf{all}} \quad (\vec{Y},Y)\in\Theta        
\end{equation*}
as well as 
\begin{equation*}
  \text{compact}\quad T:X_1\times\ldots\times X_m\to X \quad \text{for \textbf{some}} \quad (\vec{X},X)\in\Theta.   
\end{equation*}
Then $T$ is
\begin{equation*}
     \text{compact}\quad T:Z_1\times\ldots\times Z_m\to Z \quad \text{for \textbf{all}} \quad (\vec{Z},Z)\in\Theta,
\end{equation*}
in each of the following cases:
\begin{enumerate}[label=\textup{(\arabic*)}]
    \item There exist fixed constants $t\in(0,\infty)$ and $\gamma\in\left[0,\infty\right)$ such that $\Theta$ consists of all pairs $(\vec{Y},Y)$ of the form
    \begin{align*}
        &Y_j = L^{p_j(\cdot)}(w_j),\quad j=1,\ldots,m,\\
        &Y = L^{q(\cdot)}(\nu_{\vec{w}}),
    \end{align*}
    where $(\vec{p}(\cdot),q(\cdot),\vec{1},\infty)$ is a proper $m$-admissible quadruple with
    \begin{equation*}
        t<\min\{(p_{j})_{-}:~j=1,\ldots,m\}
    \end{equation*}
    and
    \begin{equation*}
        \frac{1}{p(\cdot)}-\frac{1}{q(\cdot)}=\gamma,
    \end{equation*}
    and $\vec{w} = (w_1,\ldots,w_m)$ is a $m$-tuple of weights such that $\vec{w}\in\mathcal{A}_{\vec{p}(\cdot),q(\cdot)}$ and $\vec{w}^{t}\in\mathcal{A}_{\frac{\vec{p}(\cdot)}{t},\frac{q(\cdot)}{t}}$.\label{eq: main thm 1}

    \item There exist a fixed constant $\gamma\in\left[0,\infty\right)$, a fixed $m$-tuple $\vec{r}\in[1,\infty)^{m}$ and a fixed constant $s\in(0,\infty]$ such that $\Theta$ consists of all pairs $(\vec{Y},Y)$ of the form
    \begin{align*}
        &Y_j = L^{p_j(\cdot)}(w_j),\quad j=1,\ldots,m,\\
        &Y = L^{q(\cdot)}(\nu_{\vec{w}}),
    \end{align*}
    where $(\vec{p}(\cdot),q(\cdot),\vec{r},s)$ is a proper $m$-admissible quadruple with
    \begin{equation*}
        \frac{1}{p(\cdot)}-\frac{1}{q(\cdot)}=\gamma,
    \end{equation*}
    and $\vec{w} = (w_1,\ldots,w_m)$ is a $m$-tuple of weights such that $\vec{w}\in\mathcal{A}_{(\vec{p}(\cdot),q(\cdot)),(\vec{r},s)}$.\label{eq: main thm 2}
\end{enumerate}
\end{theorem}

\begin{remark}
In contrast to \cite{extr_comp_var_bil}, we do not restrict ourselves to bilinear operators, and moreover, we explicitly allow that our operators map into quasi-Banach spaces rather than just Banach ones. Furthermore, setting $t=1$ and taking all variable exponents to be constant in Theorem \ref{thm:main_result_extr}, allows us to recover and extend several results of \cite{CaoOlivoYabuta2022, HL2022, Lorist2024}.
\end{remark}

The main ingredients needed for the proof of Theorem~\ref{thm:main_result_extr} are three-fold. First, we rely on our weighted Riesz--Kolmogorov type Theorem \ref{thm: RK criterion} for variable Lebesgue spaces. Second, we need to understand the interpolation of compactness for multilinear operators on weighted variable Lebesgue spaces, see Theorem \ref{thm:interpolation_compactness}. This is achieved through delicate technical work in Section~\ref{sec: interpolation}, see, especially, the weighted interpolation of boundedness Theorems \ref{thm:interpolation_boundedness} and \ref{thm:interpolation_boundedness_mixed} in weighted variable Lebesgue spaces and mixed-norm weighted variable Lebesgue spaces, respectively. The fact that we explicitly allow quasi-Banach spaces in the range of our operators considerably complicates matters. In particular, we cannot apply one of the classical abstract interpolation schemes that can be found in standard references on the subject such as \cite{BeLo1976}. To overcome this difficulty, we adapt the strategy from  \cite{CaoOlivoYabuta2022}, which concerned a similar setting to ours but with constant exponents. However, the presence of variable exponents necessitates much more involved arguments and new ideas. Finally, we use the factorization results for multilinear variable Muckenhoupt classes from the paper \cite{extr_comp_var_bil} of the authors, which themselves relied in turn on the factorization results for linear variable Muckenhoupt classes proved by the second named author and Oikari in \cite{LO2025}.

As applications of Theorem \ref{thm:main_result_extr}, we establish new compactness results for the multilinear commutators of $\omega$-Calder\'{o}n--Zygmund operators (see Subsection \ref{subsec:CZ}), fractional integrals (see Subsection \ref{subsec:fractional}) and Fourier multipliers (see Subsection \ref{subsec:fourier}).

The subsequent sections of this paper are structured as follows. In Section \ref{sec: preliminaries}, we recall the notion of multilinear compact operators and introduce several definitions related to power quasi-Banach function spaces, unweighted/weighted variable Lebesgue spaces and variable Muckenhoupt weights. In Section \ref{sec: char. cmp.}, we prove our first main result, that is, the Riesz–Kolmogorov Theorem \ref{thm: general RK criterion} and its corollary Theorem \ref{thm: RK criterion}. This will be used in Section \ref{sec: interpolation} to establish a weighted interpolation theorem for multilinear compact operators in the setting of variable Lebesgue spaces (see Theorem \ref{thm:interpolation_compactness}). With this at our disposal, and together with the factorization results from \cite{extr_comp_var_bil} (see Lemmata \ref{lem:KeyLemma} and \ref{lem:KeyLemma_with_t}), we complete the proof of our second main result, namely, Theorem \ref{thm:main_result_extr}, in Section \ref{sec: pf. abstract results}. Lastly, in Section \ref{sec: applic.}, we turn to applications of Theorem \ref{thm:main_result_extr}.

\subsection*{Notation}
We use the notation $x\lesssim y$ to indicate that there exists a constant $C>0$ such that $x\leq Cy$. The constant $C$ is either an absolute constant or may depend on parameters that are specified in each instance or are clear from the context. Moreover, we write $x\sim y$ if $x\lesssim y$ and $y\lesssim x$. For $x\in\R^n$ and $r>0$, we denote by $B(x,r)$ the Euclidean open ball centered at $x$ with radius $r$. Next, for a set $U\subset\R^n$, we denote by $|U|$ its $n$-dimensional Lebesgue measure. Finally, we define the integral average $\langle g\rangle_U$ of the function $g$ over any measurable set $U$ of finite positive measure by
\begin{equation*}
  \langle g\rangle_U:=\displaystyle \mint{-}_U\;g(x)\dd x=\frac{1}{|U|}\int_U\;g(x)\dd x.
\end{equation*}

\section{Preliminaries}\label{sec: preliminaries}
As this paper primarily concerns multilinear compact operators, we briefly review their definition here for the convenience of the reader. Let $A_1,\dots A_m$ be normed spaces and $B$ be a quasi normed space. We say that a multilinear operator   
$T:A_1\times\cdots\times A_m\rightarrow B$ is compact if the set $\{T(a_1,\dots,a_m):\;\|a_i\|_{A_i}\leq1,\;1\leq i\leq m\}$ is precompact in $B$. It is easy to check that \cite[Proposition 1]{BenTor2013} also holds for multilinear operators and thus it provides several equivalent ways of formulating the compactness of $T$. 

\subsection{Power quasi-Banach function spaces}
The goal of this subsection is to define the power quasi-Banach function spaces. To this end, we first recall the following notion of a quasi-Banach function space from \cite{LoNie2024}.

\begin{definition}[quasi-BFS]\label{Def: quasi-BFS}
Let $(A,\mu)$ be a $\sigma$-finite measure space and denote by $L^0(A)$ the space of a.e. finite complex valued measurable functions on $A$. We say that a quasi-normed vector space $X\subseteq L^0(A)$ is a \emph{quasi-Banach function space} over $A$ provided that it satisfies the following properties:

\begin{itemize}
  \item \textbf{\emph{Ideal property}}: If $f\in X$ and $g\in L^0(A)$ with $|g|\leq|f|$, then $g\in X$ with $\|g\|_X\leq\|f\|_X$.
  \item \textbf{\emph{Saturation}}: For every measurable $C\subseteq A$ of positive measure, there exists a measurable $D\subseteq A$ of positive measure with $\chi_D\in X$.
  \item \textbf{\emph{Fatou property}}: If for any pointwise increasing sequence $(f_k)_{k=1}^{\infty}$ of nonnegative functions in $X$ with $f_k\to f$ pointwise a.e.~as $k\to\infty$ for some $f\in L^{0}(\Omega)$ and $\sup\limits_{k=1,2,\ldots}\Vert f_k\Vert_{X}<\infty$, then $f\in X$ and $\Vert f\Vert_{X}=\sup\limits_{k=1,2,\ldots}\Vert f_k\Vert_{X}$. 
\end{itemize}
\end{definition}

\begin{remark}
  \item[(1)] If $\|\cdot\|_X$ is a norm, then X is called \emph{Banach function space} (BFS) over $A$.
  \item[(2)] Since $\|\cdot\|_X$ is a quasi-norm, this means that there exists a constant $K\ge 1$ such that 
  \begin{equation*}
    \|f+g\|_X\leq K(\|f\|_X+\|g\|_X),
  \end{equation*}
  for all $f,g\in X$.  
  \item[(3)] Some works in the literature assume that the underlying measure space $(A,\mu)$ is complete. As explained in \cite[Remark 2.1 (iii)]{Lorist2024}, however, this assumption is not needed. 
  \item[(4)] By \cite[Remark 2.1 (ii)]{Lorist2024}, one sees that the Fatou property implies that $X$ is complete. 
  \item[(5)] Although our Definition \ref{Def: quasi-BFS} is formulated for a general $\sigma$-finite measure space $(A,\mu)$ of positive measure, some of our results will be stated in the particular setting $A=\R^n$ or subsets thereof, equipped with the Lebesgue measure $|\cdot|$. 
\end{remark}

We define the K{\"o}the dual or associate space $X'$ of a quasi-BFS $X$ by
\begin{equation*}
  X':=\left\{g\in L^0(A):fg\in L^1(A)\;\text{for all}\;f\in X\right\}.
\end{equation*}
Furthermore, for $g\in X'$ we define
\begin{equation}\label{eq: X' norm}
  \|g\|_{X'}:=\sup_{\|f\|_X=1}\|fg\|_{L^1(A)}
\end{equation}
which is a finite norm as explained in \cite[page 256]{LoNie2024}. Notice that, by applying \eqref{eq: X' norm} to the normalized function $h=f/\|f\|_X$ (for $f\neq 0$) yields a generalization of H\"older's inequality for a quasi-BFS $X$; the case $\|f\|_X=0$ follows trivially. In particular, if $f\in X$ and $g\in X'$, then $fg\in L^1(A)$ and 
\begin{equation}\label{eq: X-Holders}
  \int_A|f(x)g(x)|\dd x\leq\|f\|_X\|g\|_{X'}. 
\end{equation}
Further properties of $X'$ can be found in \cite{LoNie2024}.

We are now in a position to introduce the notion of a power quasi-BFS. For this terminology, we refer to \cite[Definition 3.4]{GuoZhao2020}.

\begin{definition}(Power quasi-BFS)\label{def: PQBS}
A quasi-BFS $X$ is called a power quasi-BFS if there exists a constant $b\in(0,\infty)$ such that $X^b:=\left\{f\in L^0(\mu):|f|^{1/b}\in X\right\}$ equipped with the quasi-norm $\|f\|_{X^b}:=\left\||f|^{1/b}\right\|_X^b$ is a BFS.
\end{definition}

\subsubsection{Absolute continuity and related results}

We recall the following property of the absolutely continuous quasi-norm which will be important for our purposes later. We again adopt the formulation given in \cite{LoNie2024}.

\begin{definition}
Let $X$ be a quasi-BFS over a $\sigma$-finite measure space $(A,\mu)$. We say that $X$ has an \emph{absolutely continuous quasi-norm}, if for all $f\in X$ and for all decreasing sequences $(A_k)_{k=1}^{\infty}$ of measurable sets with $\mu\left(\displaystyle\bigcap_{k=1}^{\infty}A_k\right)=0$, there holds $\lim\limits_{k\to\infty}\Vert f\chi_{A_k}\Vert_{X}=0$.
\end{definition}

The following proposition follows from \cite[page 260 and Proposition 3.12]{LoNie2024}.

\begin{proposition}[\cite{LoNie2024}]\label{prop: X-DCT}
Let $X$ be a quasi-BFS over a $\sigma$-finite measure space $(A,\mu)$. Then $X$ has an absolutely continuous quasi-norm if and only if an $X$-version of the dominated convergence theorem holds, i.e.~given any sequence $(f_k)_{k=1}^{\infty}$ in $X$ such that $f_k\rightarrow f$ pointwise a.e.~and there exists $g\in X$ such that $|f_k|\leq g$ for all $k\ge1$, then $\lim\limits_{k\to\infty}\|f_k-f\|_X=0$.    
\end{proposition}

The proof of the following proposition is the same as the one in \cite[Chapter 1, Section 3, Lemma 3.4]{Interpolation_operators}. 

\begin{proposition}[\cite{Interpolation_operators}]\label{prop: X-ABC}
Let $X$ be a quasi-BFS over a $\sigma$-finite measure space $(A,\mu)$ that has an absolutely continuous quasi-norm and let $f\in X$. Then for every $\varepsilon>0$ there is a $\delta>0$ such that given any $\mu$-measurable set $D\subseteq A$ satisfying $\mu(D)<\delta$ it holds that $\|f\chi_D\|_X<\varepsilon$.       
\end{proposition}

\subsection{Lebesgue spaces in the variable exponent setting} We begin by gathering some basic results about variable Lebesgue spaces. We point out that the references \cite{CF2013} and \cite{DHHR2011} contain most of the properties for these spaces.

Throughout the paper, all sets denoted by notations such as $\Omega,\Omega_i,\widetilde{\Omega},$ or similar will be assumed to be measurable subsets of some $\R^{n}$ of positive measure, unless explicitly otherwise stated. Notice that the implicit $n$ might be different from set to set, unless explicitly otherwise stated.

\subsubsection{Notation for functions with variable exponents}
For any measurable function $p(\cdot):\Omega\to[0,\infty]$ and any set $U\subset \Omega$, we define
$$
  p_{-}(U)=\essinf_{x\in U}p(x),\qquad p_{+}(U)=\esssup_{x\in U}p(x),
$$
and if $U=\Omega$, then we will abbreviate $p_{\pm}(\Omega)$ to $p_{\pm}$.

In the sequel, we will focus on certain notable classes of measurable functions $p(\cdot):\Omega\to[0,\infty]$. These are defined as
\begin{equation*}
  \mathscr{P}_0:=\{p(\cdot):~0<p_{-} \leq p_{+}<\infty\}
\end{equation*}
and
\begin{equation*}
  \mathscr{P}:=\{p(\cdot):~1<p_{-}\leq p_{+}<\infty\}.
\end{equation*}
 
Given $p(\cdot)\in\mathscr{P}$, its dual exponent $p'(\cdot)\in\mathscr{P}$ is defined by the relation $1/p(\cdot)+ 1/p'(\cdot)=1.$

\subsubsection{Luxemburg norm of $L^{p(\cdot)}$ spaces} Given $p(\cdot)\in\mathscr{P}_0$, define the modular associated with $p(\cdot)$ by 
\begin{equation}\label{eq:modular}
  \rho_{p(\cdot)}(f):=\int_{\Omega}|f(x)|^{p(x)}\dd x.
\end{equation} 
The variable Lebesgue space $L^{p(\cdot)}(\Omega)$, or for simplicity just $L^{p(\cdot)}$, consists of all measurable functions $f:\Omega\rightarrow \R$ such that 
\begin{equation*}
  \|f\|_{L^{p(\cdot)}}= \|f\|_{{p(\cdot)}}:=\inf\{\lambda\in(0,\infty):\rho_{p(\cdot)}(f/\lambda)\leq 1\}<\infty.
\end{equation*}
It follows from \cite[Proposition 2.12]{CF2013} that $f\in L^{p(\cdot)}(\Omega)$ if and only if $\rho_{p(\cdot)}(f)<\infty$.

\subsubsection{Weighted Lebesgue spaces in the variable exponent setting} A weight is a measurable function $w$ on $\Omega$ such that $0<w(x)<\infty$ for almost every $x\in\Omega$. Given a weight $w$, we define $L^{p(\cdot)}(w)$ to be the weighted variable Lebesgue space with norm
\begin{equation*}
  \|f\|_{L^{p(\cdot)}(w)}:= \|fw\|_{p(\cdot)}< \infty.
\end{equation*}
  
\subsubsection{Log-H\"{o}lder continuity conditions} For technical purposes, the variable exponent functions we will be considering on the whole of $\R^n$ will possess both local and asymptotic H\"{o}lder continuity. The relevant definitions are provided below.

\begin{definition}\label{logHolder}
Let $p(\cdot):\R^n\rightarrow\R$ be a function. We say that $p(\cdot)$
\begin{enumerate}
\item\label{cond1} is locally log-H{\"o}lder continuous, denoted by $p(\cdot)\in \mathrm{LH}_0,$ if there exists $C_0>0$ such that
\begin{equation*}
  |p(x)-p(y)|\leq\frac{C_0}{-\log(|x-y|)}\;\;\text{for all}\;\; x,y\in\R^n,\;\; \text{whenever}\;\; |x-y|<\tfrac{1}{2};
\end{equation*}
\item\label{cond2} is log-H{\"o}lder continuous at infinity, denoted by $p(\cdot)\in \mathrm{LH}_{\infty},$ if there exist $p_{\infty}\in\R$ and $C_\infty>0$ such that
\begin{equation*}
  |p(x)-p_{\infty}|\leq\frac{C_{\infty}}{\log(e+|x|)} \;\;\text{for all}\;\; x\in\R^n.
\end{equation*}
\end{enumerate}
We then define the class of globally log-H{\"o}lder continuous functions as $\mathrm{LH} := \mathrm{LH}_0\cap \mathrm{LH}_{\infty}$. The corresponding log-H{\"o}lder constant is defined as $C_{\log}(p):=\max\{C_0,C_{\infty}\}$.
\end{definition}

\subsubsection{Auxiliary results}

The following results were originally established for exponents in $\mathscr{P}$. In case of Lemma~\ref{lem:homog}, the same proof remains valid for exponents in $\mathscr{P}_0$. For the rest, the same proofs remain valid for exponents in $\mathscr{P}_0$, or one can use Lemma~\ref{lem:homog} to reduce back to exponents in $\mathscr{P}$.

\begin{lemma}[\cite{CF2013}, Proposition 2.18]\label{lem:homog}
Let $p(\cdot)\in\mathscr{P}_0$. Then, it holds that  $\||f|^s\|_{p(\cdot)}=\|f\|^s_{sp(\cdot)},$ for all constants $s\in(0,\infty)$. 
\end{lemma}

\begin{lemma}[\cite{CF2013}, Corollary 2.23]\label{lem: min/max}
Let $p(\cdot)\in \mathscr{P}_0$ and $f\in L^{p(\cdot)}$. If $\Vert f\Vert_{p(\cdot)}>1$, then
\begin{equation*}
  \rho_{p(\cdot)}(f)^{\frac{1}{p_{+}}} \leq \Vert f\Vert_{p(\cdot)} \leq \rho_{p(\cdot)}(f)^{\frac{1}{p_{-}}}.
\end{equation*}
If $0<\Vert f\Vert_{p(\cdot)}\leq1$, then
\begin{equation*}
  \rho_{p(\cdot)}(f)^{\frac{1}{p_{-}}} \leq \Vert f\Vert_{p(\cdot)} \leq \rho_{p(\cdot)}(f)^{\frac{1}{p_{+}}}.
\end{equation*}
Summarizing we have 
\begin{equation*}
  \min\left\{\rho_{p(\cdot)}(f)^{\frac{1}{p_-}},\rho_{p(\cdot)}(f)^{\frac{1}{p_{+}}}
  \right\}\leq\|f\|_{p(\cdot)}\leq\max\left\{\rho_{p(\cdot)}(f)^{\frac{1}{p_-}},\rho_{p(\cdot)}(f)^{\frac{1}{p_{+}}}\right\}.
\end{equation*}
\end{lemma}

\begin{lemma}[\cite{TWX2025}, Lemma 6.2]\label{lem:Hoelder}
Let $p(\cdot),p_1(\cdot),\dots,p_m(\cdot)\in\mathscr{P}_{0}$ satisfy
\begin{equation*}
  \frac{1}{p(\cdot)}=\sum_{j=1}^m\frac{1}{p_j(\cdot)},
\end{equation*}
where $j=1,\dots,m$. Then, for any $f_j\in L^{p_j(\cdot)}$, $j=1,\ldots,m$ we have $f_1\cdots f_m \in L^{p(\cdot)}$, and in fact
\begin{equation*}
\|f_1\cdots f_m\|_{L^{p(\cdot)}}\lesssim\|f_1\|_{L^{p_1(\cdot)}}\cdots\|f_m\|_{L^{p_m(\cdot)}},
\end{equation*}
where the implicit constant depends only on $m$.
\end{lemma}

\subsubsection{Convergence theorems}
The following monotone and dominated convergence theorems for variable Lebesgue spaces were originally established for exponents in $\mathscr{P}$. We state them directly for exponents in $\mathscr{P}_0$, since the same proofs remain valid in this more general setting. Alternatively, one can use Lemma~\ref{lem:homog} to reduce back exponents in $\mathscr{P}$.

\begin{proposition}[\cite{CF2013}, Theorem 2.59]\label{prop:monotone_convergence}
Let $p(\cdot)\in\mathscr{P}_0$. Let $(f_k)_{k=1}^{\infty}$ be an increasing sequence of nonnegative functions in $L^{p(\cdot)}$ converging pointwise a.e.~to a function $f$. Then, we have $\lim\limits_{k\to\infty}\Vert f_k\Vert_{L^{p(\cdot)}} = \Vert f\Vert_{L^{p(\cdot)}}$.    
\end{proposition}

\begin{proposition}[\cite{CF2013}, Theorem 2.62]\label{prop:dominated_convergence}
Let $p(\cdot)\in\mathscr{P}_0$. Let $(f_k)_{k=1}^{\infty}$ be a sequence of functions in $L^{p(\cdot)}$ converging pointwise a.e.~to a function $f$. Assume that there exists a nonnegative function $g\in L^{p(\cdot)}$ such that $|f_{k}|\leq g$ a.e., for all $k=1,2,\ldots$. Then, we have $f\in L^{p(\cdot)}$ and $f_k\to f$ as $k\to\infty$ in $L^{p(\cdot)}$.    
\end{proposition}

Recall that a sequence of measurable functions $f_k:\Omega\to\C$, $k=1,2,\ldots$ \emph{converges in measure} to a measurable function $f:\Omega\to\C$ if
\begin{equation*}
  \lim_{k\to\infty}|\{|f_k-f|\geq\eps\}| = 0,\quad\forall \eps>0.
\end{equation*}
Recall also that if $(f_k)_{k=1}^{\infty}$ converges in measure to $f$, then one can find  a subsequence $(f_{m_k})^{\infty}_{k=1}$ converging pointwise a.e.~to $f$.

\begin{proposition}[\cite{CF2013}, Theorem 2.66 and remarks after its proof]\label{prop:norm_convergence_to_subs_pointwise_convergence}
Let $p(\cdot)\in\mathscr{P}_0$. Let $(f_k)_{k=1}^{\infty}$ be a sequence of functions in $L^{p(\cdot)}$ such that $f_k\to f$ as $k\to\infty$ in $L^{p(\cdot)}$. Then, $(f_k)_{k=1}^{\infty}$ converges in measure to $f$. In particular, there exists a subsequence $(f_{m_k})^{\infty}_{k=1}$ converging pointwise a.e.~to $f$.
\end{proposition}

Proposition~\ref{prop:norm_convergence_to_subs_pointwise_convergence} will be useful later to prove a.e.~equality of functions that arise as limits of the same sequence of functions in two different modes of convergence. For future purposes, we will also need a result of this type that concerns \emph{mixed} norms.

\begin{proposition}\label{prop:almost_everywhere_equality_mixed}
Let $p_1(\cdot),p_2(\cdot)\in\mathscr{P}_0(\Omega)$ and $q_1,q_2\in(0,\infty)$. Let $f_k:\Omega\times\widetilde{\Omega}\to\C$, $k=1,2,\ldots$ be a sequence of measurable functions and let $f,g:\Omega\times\widetilde{\Omega}\to\C$ be measurable functions such that
\begin{equation*}
  \lim_{k\to\infty}\Vert \Vert f_{k}(x,y)-f(x,y)\Vert_{L^{q_1}_{y}}\Vert_{L^{p_1(\cdot)}_{x}}=0
\end{equation*}
and
\begin{equation*}
  \lim_{k\to\infty}\Vert \Vert f_{k}(x,y)-g(x,y)\Vert_{L^{q_2}_{y}}\Vert_{L^{p_2(\cdot)}_{x}}=0.
\end{equation*}
Then, it holds $f = g$ a.e.~on $\Omega\times\widetilde{\Omega}$.
\end{proposition}

\begin{proof}
Consider the measurable set
\begin{equation*}
E:=\{f \neq g\}\subset\Omega\times\widetilde{\Omega}
\end{equation*}
and define
\begin{equation*}
  E_x:= \{y\in\widetilde{\Omega}:~(x,y)\in E\},\quad x\in\Omega.
\end{equation*}
By the theorem of Tonelli--Fubini we have
\begin{equation*}
  |E|=\int_{\Omega}|E_x|\dd x.
\end{equation*}
Therefore, in order to show that $|E|=0$, it suffices to show that $|E_x|=0$ for a.e.~$x\in\Omega$.

Since
\begin{equation*}
  \lim_{k\to\infty}\Vert \Vert f_{k}(x,y) - f(x,y)\Vert_{L^{q_1}_{y}}\Vert_{L^{p_1(\cdot)}_{x}} = 0,
\end{equation*}
by Proposition~\ref{prop:norm_convergence_to_subs_pointwise_convergence} we deduce that there exist a subsequence $(f_{m_k})_{k=1}^{\infty}$ and a measurable set $A\subset\Omega$ with $|\Omega\setminus A|=0$, such that
\begin{equation*}
  \lim_{k\to\infty}\Vert f_{m_k}(x,y) - f(x,y)\Vert_{L^{q_1}_{y}} = 0,\quad\forall x\in A.
\end{equation*}
Then, since
\begin{equation*}
  \lim_{k\to\infty}\Vert \Vert f_{m_k}(x,y) - g(x,y)\Vert_{L^{q_2}_{y}}\Vert_{L^{p_2(\cdot)}_{x}} = 0,
\end{equation*}
by Proposition~\ref{prop:norm_convergence_to_subs_pointwise_convergence} we deduce that there exist a further subsequence $(f_{m_{\ell_k}})_{k=1}^{\infty}$ and a measurable set $B\subset\Omega$ with $|\Omega\setminus B|=0$, such that
\begin{equation*}
  \lim_{k\to\infty}\Vert f_{m_{\ell_k}}(x,y) - g(x,y)\Vert_{L^{q_2}_{y}} = 0,\quad\forall x\in B.
\end{equation*}
    
Set $C:= A\cap B$. Then, $C\subset\Omega$ is a measurable set with $|\Omega\setminus C|=0$. We will show that $|E_x|=0$, for every $x\in C$. Indeed, for every $x\in C$ we have
\begin{equation*}
  f_{m_{\ell_k}}(x,y) \to f(x,y)\quad\text{ as }k\to\infty\quad\text{in }L^{q_1}_{y} 
\end{equation*}
and
\begin{equation*}
  f_{m_{\ell_k}}(x,y) \to g(x,y)\quad\text{ as }k\to\infty\quad\text{in }L^{q_2}_{y}, 
\end{equation*}
which yields that
\begin{equation*}
  f(x,y)=g(x,y)\quad\text{for a.e. }y\in\widetilde{\Omega},
\end{equation*}
that is $|E_x|=0$. This concludes the proof.
\end{proof}

\subsection{Multilinear variable Muckenhoupt conditions}

In this subsection we describe the classes of the weights that feature in the results of this paper. The definitions and facts that follow have already appeared in \cite{extr_comp_var_bil}, where we refer the reader for more details.

\begin{definition}
\noindent\begin{enumerate}
    \item[(1)] 
We say that a quadruple $(\vec{p}(\cdot), q(\cdot), \vec{r}, s)$ is \emph{$m$-admissible} if all of the following conditions are satisfied:
\begin{itemize}
    \item $\vec{r}= (r_1,\ldots,r_m)\in(0,\infty)^m$ is a $m$-tuple of constants.

    \item $s\in(0,\infty]$ is a constant.

    \item $\vec{p}(\cdot) = (p_1(\cdot),\ldots,p_m(\cdot))\in\mathscr{P}_0^m$ is a $m$-tuple of variable exponents with
    \begin{equation*}
    r_j < (p_{j})_{-},\quad\forall j=1,\ldots,m.
    \end{equation*}

    \item $q(\cdot)\in\mathscr{P}_{0}$ is a variable exponent with $q_{+} < s$.

    \item There is a constant $\gamma\in[0,\infty)$ such that
    \begin{equation*}
    \frac{1}{p(\cdot)} - \frac{1}{q(\cdot)} = \gamma,
    \end{equation*}
    where the variable exponent $p(\cdot)\in\mathscr{P}_0$ is defined via
    \begin{equation*}
    \frac{1}{p(\cdot)} := \sum_{j=1}^{m}\frac{1}{p_j(\cdot)}.
    \end{equation*}
\end{itemize}
In this case, we define the constant $r\in(0,\infty)$ via
    \begin{equation*}
    \frac{1}{r} := \sum_{j=1}^{m}\frac{1}{r_j}.
    \end{equation*}

    \item[(2)] We say that a $m$-admissible quadruple $(\vec{p}(\cdot), q(\cdot), \vec{r}, s)$ is \emph{proper} if
    \begin{equation*}
    p_j(\cdot)\in \mathrm{LH},\quad\forall j=1,\ldots,m
    \end{equation*}
    and
    \begin{equation*}
    q(\cdot)\in \mathrm{LH}.
    \end{equation*}
   Notice that in this case, $p(\cdot)\in \mathrm{LH}$.

   \item[(3)] A quadruple $(\vec{p}(\cdot), q(\cdot), \vec{r}, s)$ is said to be \emph{generalized $m$-admissible}, if the conditions of item (1) are satisfied with the following changes:
    \begin{itemize}
        \item $s \in (-\infty, 0)\cup(0,\infty]$, and

        \item $\frac{1}{q_{+}} > \frac{1}{s}$.
    \end{itemize}
    A \emph{proper generalized $m$-admissible quadruple} is analogously defined.
\end{enumerate}
\end{definition}

\begin{definition}
Let $(\vec{p}(\cdot),q(\cdot),\vec{r},s)$ be a generalized $m$-admissible quadruple with
\begin{equation*}
    \frac{1}{p(\cdot)} - \frac{1}{q(\cdot)} = \gamma\in[0,\infty).
\end{equation*}
A vector of weights $\vec{w}=(w_1,\dots,w_m)$ is said to satisfy the $m$-linear $\mathcal{A}_{({\vec{p}(\cdot),q(\cdot)}),(\vec{r},s)}$ condition (or $\vec{w}\in \mathcal{A}_{({\vec{p}(\cdot),q(\cdot)}),(\vec{r},s)}$) if

\begin{equation}\label{eq:Apqrs constant}
    [\vec{w}]_{\mathcal{A}_{(\vec{p}(\cdot),q(\cdot)),(\vec{r},s)}}:=\sup_{Q}|Q|^{\gamma - \left(\frac{1}{r}-\frac{1}{s}\right)}
    \Vert \nu_{\vec{w}}\chi_{Q}\Vert_{\frac{1}{\frac{1}{q(\cdot)}-\frac{1}{s}}}\prod_{j=1}^{m}\Vert w_j^{-1}\chi_{Q}\Vert_{\frac{1}{\frac{1}{r_j}-\frac{1}{p_j(\cdot)}}}<\infty.
\end{equation}
Here, the supremum is taken over all cubes $Q\subset\R^n$ and $\nu_{\vec{w}} := \prod_{j=1}^{m}w_j$.

If $\gamma=0$, we denote
\begin{equation*}
    \mathcal{A}_{\vec{p}(\cdot),(\vec{r},s)} := \mathcal{A}_{({\vec{p}(\cdot),q(\cdot)}),(\vec{r},s)}.
\end{equation*}
Moreover, if $r_j=1$ and $s=\infty$, then we denote
\begin{equation*}
    \mathcal{A}_{\vec{p}(\cdot),q(\cdot)} := \mathcal{A}_{({\vec{p}(\cdot),q(\cdot)}),(\vec{r},s)}.
\end{equation*}
\end{definition}

An equivalent characterization for these multilinear variable exponent Muckenhoupt classes in terms of linear ones was proved in \cite[Lemma 3.5]{extr_comp_var_bil}. Since this principle will be important for us, we state it below and refer the reader to \cite[Lemma 3.5]{extr_comp_var_bil} for its proof.

\begin{lemma}[\cite{extr_comp_var_bil}, Lemma 3.5]\label{lem:Apqrs char.}
Let $(\vec{p}(\cdot),q(\cdot),\vec{r},s)$ be a $m$-admissible quadruple. Then $\vec{w}\in\mathcal{A}_{({\vec{p}(\cdot),q(\cdot)}),(\vec{r},s)}$ if and only if    
\begin{equation}\label{equiv. mult. weights 1}
\begin{cases}
  w_j\in\mathcal{A}_{p_j(\cdot),(r_j,\sigma_j)}, &\text{with}\quad\frac{1}{\sigma_j}=\frac{1}{r_j}-\big(\frac{1}{r}-\frac{1}{s}\big)\quad\text{for all}\quad j=1,\dots,m, \\
  \nu_{\vec{w}}\in\mathcal{A}_{(p(\cdot),q(\cdot)),(r,s)}.
\end{cases}
\end{equation}
\end{lemma}

\section{Characterization of total boundedness in the power quasi-bfs and variable exponent settings}\label{sec: char. cmp.}

In this section, we establish the Riesz--Kolmogorov criteria for total boundedness in the settings of a power quasi-BFS over $\R^n$ under mild conditions and as a corollary weighted variable Lebesgue spaces $L^{p(\cdot)}(w)$ for $p(\cdot)\in\mathscr{P}_0\cap\mathrm{LH}$. 

\subsection{The setting of quasi-BFS}

We begin by gathering several key results that will be needed later. Let $B(x_0,R)$ be an open ball that is contained in $\R^n$.
Then, we denote by $L^0(B(x_0,R))$ the space of measurable functions on $B(x_0,R)$ equipped with the metric 
\begin{equation*}
  d_0(f,g):=\int_{B(x_0,R)}\phi(f(x)-g(x))\dd x,  
\end{equation*}
where 
\begin{equation*}
  \phi(t):=\frac{|t|}{1+|t|},\quad t\in\C.
\end{equation*}
It is a well-known fact that this space is a complete metric space with respect to this metric and the convergence in this metric is equivalent to the convergence in measure. The following result is a special case of \cite[Theorem 1.3]{BG2018} (see also \cite{Krotov2012}):

\begin{lemma}[\cite{BG2018}, Theorem 1.3]\label{lem: X-Krotov's result}
Any almost uniformly bounded and almost equicontinuous subset of $L^0(B(x_0,R))$ is totally bounded. This means that a subset $\mathcal{F}$ of $L^0(B(x_0,R))$ is totally bounded if for any $\varepsilon>0$, there exist $\delta>0$ and $\lambda>0$ such that, for any function $f\in\mathcal{F}$, there exists a measurable subset $E(f)\subset B(x_0,R)$ satisfying the following properties
\begin{enumerate}[label=\textup{(\arabic*)}]
  \item 
  \begin{equation*}
  |E(f)|<\varepsilon;       
  \end{equation*}
  \item 
  \begin{equation*}
  |f(x)-f(y)|<\varepsilon,\quad\text{for any}\;\;x,y\in B(x_0,R)\setminus E(f)\;\;\text{satisfying}\;\;|x-y|<\delta;
  \end{equation*}
  \item 
  \begin{equation*}
  |f(x)|\leq\lambda\quad\text{for}\;\;x\in B(x_0,R)\setminus E(f).
  \end{equation*}
\end{enumerate}
\end{lemma}

The following lemma is a version of the Lebesgue--Vitali compactness-type theorem. Our proof is motivated by \cite[Theorem 2.1]{BG2018}.  

\begin{lemma}\label{lem: X-LV cmp. thm}
Let $X$ be a power quasi-BFS over the ball $B(x_0,R)\subset\R^n$ that has an absolutely continuous quasi-norm. Assume, in addition, that $\chi_{B(x_0,R)}\in X^b\cap (X^b)'$, where $b\in(0,\infty)$ and $(X^b)'$ denotes the K{\"o}the dual of $X^b$. Then, a subset $\mathcal{F}$ of $X$ is totally bounded in $X$ if and only if the following conditions are satisfied:
\begin{enumerate}[label=\textup{(\arabic*)}]
  \item $\mathcal{F}$ is totally bounded in $L^{0}(B(x_0,R))$;\label{first item: (i)}
  \item the family $\mathcal{F}$ is $X$-equi-integrable, i.e., for every $\varepsilon>0$, there exists $\delta>0$, such that
  \begin{equation*}
  \sup_{f\in\mathcal{F}}\|f\chi_E\|_{X}<\varepsilon 
\end{equation*}
for all $E\subset B(x_0,R)$ satisfying $|E|<\delta$.\label{second item: (ii)}      
\end{enumerate}
\end{lemma}

\begin{proof}
Assume that \ref{first item: (i)} and \ref{second item: (ii)} hold. Then, by \ref{second item: (ii)}, for a fixed $\varepsilon>0$, there exists $\delta>0$ such that, for each set $E\subset B(x_0,R)$ satisfying $|E|<\delta$, we have
\begin{equation}\label{eq: X-equi. int.}
  \sup_{f\in\mathcal{F}}\|f\chi_E\|_{X}<\varepsilon.    
\end{equation}
Now, let us fix $\alpha\in(0,1)$ such that $\alpha\|\chi_{B(x_0,R)}\|_X<\varepsilon$. Here, notice the fact that our assumption $\chi_{B(x_0,R)}\in X^b$ is equivalent to $\chi_{B(x_0,R)}\in X$. Therefore, $\|\chi_{B(x_0,R)}\|_{X}<\infty$. By \ref{first item: (i)} we can suppose that $\{f_k\}_{k=1}^N\subset\mathcal{F}$ is a $\delta\phi(\alpha)$-net in $L^0(B(x_0,R))$, that is, there exists a finite number of functions $\{f_k\}_{k=1}^N\subset\mathcal{F}$ such that for any $f\in\mathcal{F}$  
\begin{equation}\label{eq: d0}
  d_0(f,f_k)<\delta\phi(\alpha).    
\end{equation} 
Hence, by the quasi-triangle inequality for the quasi-norm on $X$ and the ideal property of $X$, we get
\begin{align}\label{eq: F tol. bdd.}
  \|f-f_k\|_X&\leq C\left(\|(f-f_k)\chi_{\left\{|f-f_k|\leq\alpha\right\}}\|_X+\|(f-f_k)\chi_{\left\{|f-f_k|>\alpha\right\}}\|_X\right)\\
  &\leq C\alpha\|\chi_{B(x_0,R)}\|_{X}+C\|(f-f_k)\chi_{\left\{|f-f_k|>\alpha\right\}}\|_{X}\notag\\
  &<C\varepsilon+C\|(f-f_k)\chi_{\left\{|f-f_k|>\alpha\right\}}\|_{X},\notag
\end{align}
where $C\ge 1$.

Moreover, by the monotonicity of $\phi$, Markov's inequality and \eqref{eq: d0}, we get
\begin{align*}
  |\left\{x\in X:|f(x)-f_k(x)|>\alpha\right\}|&\leq|\left\{x\in X:\phi(|f(x)-f_k(x)|)>\phi(\alpha)\right\}|\\
  &\leq\frac{1}{\phi(\alpha)}\int_{B(x_0,R)}\phi(|f(x)-f_k(x)|)\dd x=\frac{d_0(f,f_k)}{\phi(\alpha)}<\delta. 
\end{align*}
Thus, by combining this with \eqref{eq: X-equi. int.} and \eqref{eq: F tol. bdd.} we obtain   
\begin{equation*}
  \|f-f_k\|_X<2C\varepsilon,
\end{equation*}
where $C\ge1$. This shows that the set $\mathcal{F}$ is totally bounded in $X$.

Now, we assume that the family $\mathcal{F}$ is totally bounded $X$. We will prove that conditions \ref{first item: (i)} and \ref{second item: (ii)} are satisfied. 

Let us fix $\varepsilon>0$ and take $\alpha\in(0,1)$ such that $\phi(\alpha)|B(x_0,R)|<\varepsilon$. Furthermore, since the family $\mathcal{F}$ is totally bounded in $X$, there exists a finite $\frac{\varepsilon^{1/b}\alpha}{\|\chi_{B(x_0,R)}\|_{(X^b)'}^{1/b}}$-net $\{f_k\}_{k=1}^N$ in $\mathcal{F}$, that is, there exists a finite number of functions $\{f_k\}_{k=1}^N\subset\mathcal{F}$ such that for any $f\in\mathcal{F}$
\begin{equation}\label{eq: tol. bdd in X}
  \|f-f_k\|_{X}<\frac{\varepsilon^{1/b}\alpha}{\|\chi_{B(x_0,R)}\|_{(X^b)'}^{1/b}}.    
\end{equation}
Here, notice that our assumption $\chi_{B(x_0,R)}\in (X^b)'$ implies that $\|\chi_{B(x_0,R)}\|_{(X^b)'}<\infty$.

Thus, by the monotonicity of $\phi$, the fact $\phi(t)<1$ for every $t\in\C$,  H\"older's inequality (see \eqref{eq: X-Holders}) and \eqref{eq: tol. bdd in X}, we obtain
\begin{align*}
  d_0(f-f_k)&=\int_{B(x_0,R)}\phi(f(x)-f_k(x))\chi_{\left\{x\in X: |f(x)-f_k(x)|\leq\alpha\right\}}(x)\dd x\\
  &\qquad+\int_{B(x_0,R)}\phi(f(x)-f_k(x))\chi_{\left\{x\in X: |f(x)-f_k(x)|>\alpha\right\}}(x)\dd x\\
  &\leq\phi(\alpha)|B(x_0,R)|+\int_{B(x_0,R)}\chi_{\left\{x\in X: |f(x)-f_k(x)|>\alpha\right\}}(x)\dd x\\
  &<\varepsilon+\int_{B(x_0,R)}\frac{|f(x)-f_k(x)|^b}{\alpha^b}\dd x\\
  &\leq\varepsilon+\alpha^{-b}\||f-f_k|^b\|_{X^b}\|\chi_{B(x_0,R)}\|_{(X^b)'}\\
  &=\varepsilon+\alpha^{-b}\|f-f_k\|_{X}^b\|\chi_{B(x_0,R)}\|_{(X^b)'}\\
  &=2\varepsilon.
\end{align*}
This proves condition \ref{first item: (i)}.

We fix $\varepsilon>0$ and assume that $\{f_k\}_{k=1}^N$ is a $\varepsilon$-net in $X$, that is, there exists a finite number of functions $\{f_k\}_{k=1}^N\subset\mathcal{F}$ such that for any $f\in\mathcal{F}$
\begin{equation*}
  \|f-f_k\|_{X}<\varepsilon.
\end{equation*} 
Since $f_k\in X$, then, by Proposition \ref{prop: X-ABC}, there exist $\delta_k>0$, $k=1,\ldots,N$, for which the following property holds: for all $E\subset B(x_0,R)$ satisfying 
\begin{equation*}
  |E|<\delta_k\implies\|f_k\chi_E\|_X<\varepsilon.    
\end{equation*}
Then, for $\delta:=\min\left\{\delta_k:k=1,\ldots,N\right\}$ and $E\subset B(x_0,R)$ satisfying $|E|<\delta$, we have
\begin{equation*}
  \|f_k\chi_E\|_X<\varepsilon.    
\end{equation*}
By this, the quasi-triangle inequality for the quasi-norm on $X$ and the fact that $\{f_k\}_{k=1}^N$ is a $\varepsilon$-net in $X$ we obtain
\begin{equation*}
  \|f\chi_E\|_X\leq C\left(\|(f-f_k)\chi_E\|_X+\|f_k\chi_E\|_X\right)<2C\varepsilon.
\end{equation*}
This shows that condition \ref{second item: (ii)} holds.
\end{proof}

Now, we provide the proof of the Riesz--Kolmogorov Theorem \ref{thm: general RK criterion} in the power quasi-BFS setting.

\begin{proof}
Let us first prove the necessity. We fix $0<\varepsilon<1$ such that $\widetilde{\varepsilon}:=\frac{\varepsilon}{C}<1$, where $C:=\|M_b\|_{X\to X}>0$. Since $\mathcal{F}$ is totally bounded in $X$, there exists a finite $\widetilde{\varepsilon}$-net $\{f_k\}_{k=1}^N$ in $\mathcal{F}$, that is, there exists a finite number of functions $\{f_k\}_{k=1}^N\subset\mathcal{F}$ such that 
\begin{equation*}
  \mathcal{F}\subset\bigcup_{k=1}^N B(f_k,\widetilde{\varepsilon}).    
\end{equation*}
Let $f\in\mathcal{F}$ be an arbitrary function. Then, we can fix $k\in\{1,\dots,N\}$ satisfying
\begin{equation}\label{eq: fk-f}
  \|f-f_k\|_{X}<\widetilde{\varepsilon}.
\end{equation}
Set $M_1:=\sup\limits_{k=1,\dots,N}\|f_k\|_{X}<\infty$.
Then, condition \ref{item: (i)} is satisfied since by the quasi-triangle inequality for the quasi-norm on $X$ and \eqref{eq: fk-f} we have 
\begin{equation*}
  \|f\|_{X}\leq\widetilde{C}\left(\|f_k\|_{X}+\|f-f_k\|_{X}\right)\leq \widetilde{C}\left(M_1+\widetilde{\varepsilon}\right)=M,
\end{equation*}
where $\widetilde{C}\ge1$ and $M:=\widetilde{C}\left(M_1+\widetilde{\varepsilon}\right)>0$.
\newline By the elementary inequality
\begin{equation}\label{eq: kn. ineq.}
  (\beta+\gamma)^{\alpha}\leq 2^{\alpha}(\beta^{\alpha}+\gamma^{\alpha}),
\end{equation}
where $\beta,\gamma\ge0$ and $\alpha>0$, the ideal property of $X$ and the quasi-triangle inequality for the quasi-norm on $X$,  we deduce the following estimates
\begin{align}\label{eq: chain of ineq.}
  &\left\|\left(\displaystyle \mint{-}_{B(x,r)}|f(x)-f(y)|^{b}\dd y\right)^{1/b}\right\|_X\\ 
  &\leq D 2^{1+1/b}\left(\left\|\left(\displaystyle \mint{-}_{B(x,r)}|f(x)-f_k(x)|^{b}\dd y\right)^{1/b}\right\|_X+\left\|\left(\displaystyle \mint{-}_{B(x,r)}|f_k(x)-f(y)|^{b}\dd y\right)^{1/b}\right\|_X\right)\notag\\
  &\leq D2^{1+1/b}\left\|f-f_k\right\|_X\notag+\widetilde{D}D4^{1+1/b}\left\|\left(\displaystyle \mint{-}_{B(x,r)}|f_k(x)-f_k(y)|^{b}\dd y\right)^{1/b}\right\|_X\notag\\
  &\qquad+\widetilde{D}D4^{1+1/b}\left\|\left(\displaystyle \mint{-}_{B(x,r)}|f_k(y)-f(y)|^{b}\dd y\right)^{1/b}\right\|_X\notag\\
  &=:D2^{1+1/b}\mathcal{I}_1+\widetilde{D}D4^{1+1/b}\mathcal{I}_2+\widetilde{D}D4^{1+1/b}\mathcal{I}_3,\notag
\end{align}
where $\widetilde{D},D\ge1$. Next, we provide the bounds for $\mathcal{I}_1$, $\mathcal{I}_2$ and $\mathcal{I}_3$. By \eqref{eq: fk-f} it is immediate that $\mathcal{I}_1<\widetilde{\varepsilon}$.  
\newline Now, by the triangle inequality, \eqref{eq: kn. ineq.} and the assumption that $M_b:X\rightarrow X$ is bounded, we have 
\begin{align*}
  &\left(\displaystyle \mint{-}_{B(x,r)}|f_k(x)-f_k(y)|^{b}\dd y\right)^{1/b}\\
  &\leq 2\left(|f_k(x)|^{b}+\displaystyle \mint{-}_{B(x,r)}|f_k(y)|^{b}\dd y\right)^{1/b}\\
  &\leq 2^{1+1/b}\left(|f_k(x)|+\left(\displaystyle \mint{-}_{B(x,r)}|f_k(y)|^{b}\dd y\right)^{1/b}\right)\\ 
  &\leq 2^{1+1/b}\left(|f_k(x)|+M_{b} f_k(x)\right)\in X.
\end{align*}
Notice that since $X$ is a power quasi-BFS over $\R^n$ (see Definition \ref{def: PQBS}) then $M_b:X\rightarrow X$ is bounded if and only if the classical Hardy--Littlewood maximal operator $M:X^b\rightarrow X^b$ is bounded. Therefore, this and the fact that $f_k\in X$ imply $|f_k|^b\in X^b\subset L^1_{\loc}(\R^n)$. Hence, by the Lebesgue differentiation theorem for classical Lebesgue spaces we obtain that
\begin{equation*}
  \displaystyle \mint{-}_{B(x,r)}|f_k(x)-f_k(y)|^{b}\dd y\rightarrow0\;\text{as}\;r\rightarrow 0^{+},\;\text{a.e.}\;x\in\R^n.
\end{equation*}
Thus, since $X$ also has an absolutely continuous quasi-norm it follows from the $X$-version of the dominated convergence theorem (see Proposition \ref{prop: X-DCT}) that $\mathcal{I}_2<\varepsilon$ for sufficiently small $r>0$.
\newline Lastly, to deal with $\mathcal{I}_3$, we have 
\begin{equation*}
  \|M_{b}(|f_k-f|)\|_{X}\leq C\|f_k-f\|_{X}< C\widetilde{\varepsilon}=\varepsilon,
\end{equation*}
where we used the fact that $M_b:X\rightarrow X$ is bounded and \eqref{eq: fk-f}. Therefore, this and the ideal property of $X$ imply that
\begin{equation*}
  \mathcal{I}_3=\left\|\left(\displaystyle \mint{-}_{B(x,r)}|f_k(y)-f(y)|^{b}\dd y\right)^{1/b}\right\|_X\leq\left\|M_{b}(|f_k-f|)\right\|_{X}<\varepsilon.\\
\end{equation*}
Combining \eqref{eq: chain of ineq.} with the preceding estimates, we conclude that condition \ref{item: (ii)} holds. 
\newline Fix a point $x_0\in \R^n$. Since $f_k\in X$ and $X$ is a power quasi-BFS over $\R^n$ that has an absolutely continuous quasi-norm, it follows from the $X$-version of the dominated convergence theorem (see Proposition \ref{prop: X-DCT}) that we can fix sufficiently large $R_k>0$ such that
\begin{equation}\label{eq: tail}
  \|f_k\chi_{\R^n\setminus B(x_0,R_k)}\|_{X}<\varepsilon,\; k=1,\dots,N.  
\end{equation}
Then, for $R:=\max\{R_k:\;k=1,\dots,N\}$, by the ideal property of $X$ and the quasi-triangle inequality for the quasi-norm on $X$, \eqref{eq: fk-f} and \eqref{eq: tail} we have 
\begin{equation*}
\|f\chi_{\R^n\setminus B(x_0,R)}\|_{X}\leq\widetilde{\widetilde{D}}\left(\|f-f_k\|_X+\|f_k\chi_{\R^n\setminus B(x_0,R_k)}\|_{X}\right)<\widetilde{\widetilde{D}}2\varepsilon,
\end{equation*}
where $\widetilde{\widetilde{D}}\ge1$. This shows that condition  \ref{item: (iii)} holds. Therefore, we have shown that conditions \ref{item: (i)}–\ref{item: (iii)} are satisfied, which completes the proof of necessity.   

Let us next prove the sufficiency. Suppose that conditions \ref{item: (i)}, \ref{item: (ii)} and \ref{item: (iii)} hold. Then by condition \ref{item: (iii)}, for any fixed $0<\kappa<1$, there exists $R>>1$ such that 
\begin{equation}\label{eq: van. at inf.}
  \sup_{f\in\mathcal F}\|f\chi_{\R^n\setminus B(x_0,R)}\|_{X}<\kappa.
\end{equation}
At this point, we introduce the following Lipschitz cut-off function with the Lipschitz constant $\frac{1}{R}$:
\begin{equation*}
  \varphi_R(x):=
  \begin{cases}
  2-\frac{|x-x_0|}{R}, & \text{if}\quad x\in B(x_0,2R)\setminus B(x_0,R),\\
  1, & \text{if}\quad x\in B(x_0,R),\\
  0, & \text{if}\quad x\in\R^n\setminus B(x_0,2R).
  \end{cases}
\end{equation*}
Note that $0\leq\varphi_{R}\leq1$.
In addition, we define the following set $\mathcal{F}_R$ by
\begin{equation*}
  \mathcal{F_R}:=\{f\varphi_R:\;f\in\mathcal{F}\}
\end{equation*}
and $X(B(x_0,2R)):=\left\{f\in X: f\;\text{vanishes outside of}\;B(x_0,2R)\right\}$ equipped with the quasi-norm $\|f\|_{X(B(x_0,2R))}:=\|f\|_X$ to be the power quasi-BFS over the ball $B(x_0,2R)\subset\R^n$ that has an absolutely continuous quasi-norm.  

Our goal is to show that $\mathcal{F_R}$ is totally bounded in $X(B(x_0,2R))$. The proof first relies on Lemmata \ref{lem: first prop. of F_R}, \ref{lem: p-equiv.-int. of F_R}, and \ref{lem: total bdd. of F_R}, which will be stated and proved below. Observe that Lemma \ref{lem: first prop. of F_R} is used only in the proofs of Lemmata \ref{lem: p-equiv.-int. of F_R} and \ref{lem: total bdd. of F_R}. Then, the latter two lemmata guarantee that conditions \ref{first item: (i)} and \ref{second item: (ii)} of Lemma \ref{lem: X-LV cmp. thm} are satisfied which in turn implies that $\mathcal{F}_R$ is totally bounded in $X(B(x_0,2R))$. 

Thus, we can choose $\{f_k\}_{k=1}^N$ in $\mathcal{F}$ such that $\{f_k\varphi_R\}_{k=1}^N$ is an $\varepsilon$-net in $\mathcal{F}_R$, that is, there exists a finite number of functions $\{f_k\varphi_R\}_{k=1}^N\subset\mathcal{F}_R$  such that 
\begin{equation*}
  \mathcal{F}_R\subset\bigcup_{k=1}^N B(f_k\varphi_R,\varepsilon).    
\end{equation*}
Let $g\in\mathcal{F}_R$ be an arbitrary function. Then, $g$ can be written in the form $g=f\varphi_R $ with $f\in\mathcal{F}$ and we can fix $k\in\{1,\dots,N\}$ satisfying
\begin{equation}\label{eq: fk-f in F_R}
  \|g-g_k\|_{X}<\varepsilon,
\end{equation}
where $g_k=f_k\varphi_R$. We will show that $\{f_k\}_{k=1}^N\subset\mathcal{F}$ is an $\widetilde{\varepsilon}$-net set in $\mathcal{F}$, where 
\begin{equation*}
  0<\widetilde{\varepsilon}:=C\left(2\kappa+\varepsilon\right)<1.
\end{equation*}
For any $f\in\mathcal{F}$, by invoking the quasi-triangle inequality for the quasi-norm on $X$, together with inequalities \eqref{eq: van. at inf.} and \eqref{eq: fk-f in F_R}, and applying the properties of the cut-off function $\varphi_R$, we conclude that 
\begin{equation*}
    \|f-f_k\|_{X}\leq C(\|f\chi_{\R^n\setminus B(x_0,R)}\|_{X}+\|f_k\chi_{\R^n\setminus B(x_0,R)}\|_{X}
  +\|(f-f_k)\varphi_R\chi_{B(x_0,R)}\|_{X})
  <\widetilde{\varepsilon},
\end{equation*}
where $C\ge1$. This shows that $\mathcal{F}$ is totally bounded in $X$ and completes the proof of the theorem.  
\end{proof}

As mentioned earlier, we now present the statements and proofs of the three auxiliary lemmata used in the proof of Theorem \ref{thm: general RK criterion}.

\begin{lemma}\label{lem: first prop. of F_R}
The family $\mathcal{F}_R$ satisfies the following conditions:
\begin{enumerate}[label=\textup{(\alph*)}]
\item $\mathcal{F}_R$ is uniformly bounded in $X$, i.e., 
  \begin{equation*}
   \sup_{g\in\mathcal{F}_R}\|g\|_{X}<\infty;    
  \end{equation*}\label{itm: (a)}
\item $\mathcal{F}_R$ is $b$-equicontinuous in $X$, i.e.,
  \begin{equation*}
  \lim_{r\rightarrow 0}\sup_{g\in\mathcal{F}_R}\left\|\left(\displaystyle \mint{-}_{B(\cdot,r)}|g(\cdot)-g(y)|^{b}\dd y\right)^{1/b}\right\|_{X}=0.     
  \end{equation*}\label{itm: (b)}
\end{enumerate}
\end{lemma}

\begin{proof}
Let $g$ be an element of $\mathcal{F}_R$. Then, $g$ is of the form $g=f\varphi_R$ with $f\in\mathcal{F}$ and the first part \ref{itm: (a)} of the lemma follows immediately by combining the properties of the Lipschitz cut-off function $\varphi_R$ with condition \ref{item: (i)}. We now proceed to prove condition \ref{itm: (b)}. From the triangle inequality, \eqref{eq: kn. ineq.} and the properties of the cut-off function $\varphi_R$, it follows that
\begin{align*}
  &\displaystyle \mint{-}_{B(x,r)}|g(x)-g(y)|^{b}\dd y\\
  &\leq 2^{b}\left(\displaystyle \mint{-}_{B(x,r)}|f(x)-f(y)|^{b}\varphi_R^{b}(y)\dd y+\displaystyle \mint{-}_{B(x,r)}|\varphi_R(x)-\varphi_R(y)|^{b}|f(x)|^{b}\dd y\right)\\
  &\leq 2^{b}\left(\displaystyle \mint{-}_{B(x,r)}|f(x)-f(y)|^{b}\dd y+r^{b}|f(x)|^{b}\right).
\end{align*}
Thus, this in conjunction with \eqref{eq: kn. ineq.}, the ideal property of $X$, the quasi-triangle inequality for the quasi-norm on $X$ and condition \ref{item: (i)} allow us to conclude that   
\begin{align}\label{eq: equic. of g}
  &\left\|\left(\displaystyle \mint{-}_{B(x,r)}|g(x)-g(y)|^{b}\dd y\right)^{1/b}\right\|_X\\
  &\leq C2^{1+1/b}\left(\left\|\left(\displaystyle \mint{-}_{B(x,r)}|f(x)-f(y)|^{b}\right)^{1/b}\right\|_X+r\left\|f\right\|_X\right)\notag\\
  &\leq C2^{1+1/b}\left(\left\|\left(\displaystyle \mint{-}_{B(x,r)}|f(x)-f(y)|^{b}\right)^{1/b}\right\|_X+rM\right)\notag,
\end{align}
where $C\ge 1$ and $M>0$. By condition \ref{item: (ii)} for given $\varepsilon>0$, we can fix $0<r<<1$ such that the right-hand side of \eqref{eq: equic. of g} is less than $\varepsilon$. This completes the proof of the lemma.     
\end{proof}

\begin{lemma}\label{lem: p-equiv.-int. of F_R}
For every $\varepsilon>0$, there exists $\delta>0$, such that
\begin{equation*}
  \sup_{f\in\mathcal{F}}\|f\chi_E\|_{X}<\varepsilon 
\end{equation*}
for all $E\subset B(x_0,2R)$ satisfying $|E|<\delta$.       
\end{lemma}

\begin{proof}
Fix $0<r<1$ and $x\in B(x_0,2R)$. By the triangle inequality and \eqref{eq: kn. ineq.}, we have 
\begin{equation}\label{eq:ineq. for f}
  |f(x)|^{b}\leq 2^{b}(|f(x)-f(y)|^{b}+|f(y)|^{b}),
\end{equation}
for any $y\in B(x,r)$ and $f\in\mathcal{F}_R$. Averaging \eqref{eq:ineq. for f} with respect to $y\in B(x,r)$, we get
\begin{equation*}
   |f(x)|^{b}\leq 2^{b}\left(\frac{1}{|B(x,r)|}\int_{B(x,r)}|f(x)-f(y)|^{b}\dd y+\frac{1}{|B(x,r)|}\int_{B(x,r)}|f(y)|^{b}\dd y\right).     
\end{equation*}
Next, we raise the preceding inequality to the power $\frac{1}{b}$ and multiply the above inequality by $\chi_E$, where $E\subset B(x_0,2R)$ is measurable. Then, by the ideal property of $X$, the quasi-triangle inequality for the quasi-norm on $X$, H\"older's inequality (see \eqref{eq: X-Holders}) and \eqref{eq: kn. ineq.}, we have   
\begin{align}\label{eq: estimate for f in E}
  \|f\chi_E\|_X&\leq C2^{1+1/b}\Biggl(\left\|\left(\displaystyle \mint{-}_{B(x,r)}|f(x)-f(y)|^{b}\dd y\right)^{1/b}\chi_E\right\|_X\\
  &\qquad+\left\|\left(\frac{1}{|B(x,r)|}\int_{B(x,r)}|f(y)|^{b}\dd y\right)^{1/b}\chi_E\right\|_X\Biggl)\notag\\
  &\leq C2^{1+1/b}\left\|\left(\displaystyle \mint{-}_{B(x,r)}|f(x)-f(y)|^{b}\dd y\right)^{1/b}\chi_E\right\|_X\notag\\
  &\qquad+ C2^{1+1/b}\left\|\left(\frac{1}{|B(x,r)|}\|f^{b}\|_{X^b}\|\chi_{B(x,r)}\|_{(X^b)'}\right)^{1/b}\chi_E\right\|_X.\notag
\end{align}
By Definition \ref{def: PQBS} and condition \ref{itm: (a)} of Lemma \ref{lem: first prop. of F_R}, we have 
\begin{equation}\label{eq: estimate for f in ball}
 \|f^{b}\|_{X^b}=\||f|\|_{X}^b\leq\widetilde{M}^{b},
\end{equation} where $\widetilde{M}>0$.
Now, for $r<1<R$ and $x\in E\subset B(x_0,2R)$ we have $B(x,r)\subset B(x_0,3R)$, therefore
\begin{equation}\label{eq: estimate for χ}
  \frac{1}{|B(x,r)|}\cdot\|\chi_{B(x,r)}\|_{(X^b)'}\leq\frac{\|\chi_{B(x_0,3R)}\|_{_{(X^b)'}}}{|B(x,r)|}\leq\widetilde{C}r^{-n},
\end{equation}
where $\widetilde{C}$ is some positive constant that depends on $n,R, x_0$ and $b$. Here, notice that we used the fact that $X^b\subseteq L^1_{\loc}(\R^n)$ if and only if $\chi_B\in (X^b)'$ for all balls $B\subset \R^n$ (this follows by \cite[Proposition 2.2]{Nieareth2023}) and the ideal property of $X$.

Hence, by collecting the preceding estimates \eqref{eq: estimate for f in E}, \eqref{eq: estimate for f in ball}, \eqref{eq: estimate for χ}, we obtain  

\begin{equation}\label{eq: final estimate for f in E}
  \|f\chi_E\|_X\leq C2^{1+1/b}\left(\left\|\left(\displaystyle \mint{-}_{B(x,r)}|f(x)-f(y)|^{\widetilde{q}}\dd y\right)^{1/b}\right\|_X+\widetilde{C}^{1/b}\widetilde{M}r^{-n/b}\|\chi_E\|_X\right).
\end{equation}
By condition \ref{itm: (b)} of Lemma \ref{lem: first prop. of F_R} for given $\varepsilon>0$, we can fix $0<r<<1$ such that the first term on the right-hand side of \eqref{eq: final estimate for f in E} is less than $\varepsilon$. Next, since $\chi_B\in X^b$ if and only if $\chi_B\in X$ for every ball $B\subset\R^n$, we can apply Proposition \ref{prop: X-ABC}  with $f=\chi_{B(x_0,2R)}\in X$ so that we find $\delta>0$ such that if the measurable set $E\subset B(x_0,2R)$ satisfies $|E|<\delta$ then the second term on the right-hand side of \eqref{eq: final estimate for f in E} is less than $\widetilde{C}^{1/b}\widetilde{M}r^{-n/b}\varepsilon$. This completes the proof of the lemma.
\end{proof}

\begin{lemma}\label{lem: total bdd. of F_R}
The family $\mathcal{F}_R$ is totally bounded in $L^0(B(x_0,2R))$.
\end{lemma}

\begin{proof}

Let us fix $0<\varepsilon<1$. Then, by condition \ref{itm: (b)} of Lemma \ref{lem: first prop. of F_R}, there exists a $\delta>0$ such that the following condition holds for all $f\in\mathcal{F}_R$:
\begin{equation*}
  \left\|\left(\displaystyle \mint{-}_{B(x,2\delta)}|f(x)-f(y)|^{b}\dd y\right)^{1/b}\right\|_X<
  \frac{\varepsilon^{1+1/b}}{\|\chi_{B(x_0,2R)}\|_{(X^b)'}^{1/b}}.
\end{equation*}

Hence, by applying Markov's inequality and H\"older's inequality (see \eqref{eq: X-Holders}), we obtain
\begin{align*}
  &\left|\left\{x\in B(x_0,2R):\;\left(\displaystyle \mint{-}_{B(x,2\delta)}|f(x)-f(y)|^{b}\dd y\right)^{1/b}>\varepsilon\right\}\right|\\
  &=\left|\left\{x\in B(x_0,2R):\;\displaystyle \mint{-}_{B(x,2\delta)}|f(x)-f(y)|^{b}\dd y>\varepsilon^{b}\right\}\right|\\
  &\leq\frac{1}{\varepsilon^b}\int_{B(x_0,2R)}\left(\displaystyle \mint{-}_{B(x,2\delta)}|f(x)-f(y)|^{b}\dd y\right)\dd x\\
  &\leq\frac{1}{\varepsilon^b}\left\|\displaystyle \mint{-}_{B(x,2\delta)}|f(x)-f(y)|^{b}\dd y\right\|_{X^b}\left\|\chi_{B(x_0,2R)}\right\|_{(X^b)'}\\
  &=\frac{1}{\varepsilon^b}\left\|\left(\displaystyle \mint{-}_{B(x,2\delta)}|f(x)-f(y)|^{b}\dd y\right)^{1/b}\right\|_X^b\left\|\chi_{B(x_0,2R)}\right\|_{(X^b)'}\\
  &<\varepsilon.
\end{align*}
\newline Next, for any $x,y\in B(x_0,2R)$ satisfying $|x-y|<\delta$, we deduce from the triangle inequality and \eqref{eq: kn. ineq.} that 
\begin{align*}
  |f(x)-f(y)|&=\left(\displaystyle \mint{-}_{B(x,\delta)}|f(x)-f(y)|^{b}\dd z\right)^{1/b}\\
  &\leq2^{1+\frac{1}{b}}\left(\left(\displaystyle \mint{-}_{B(x,\delta)}|f(x)-f(z)|^{b}\dd z\right)^{1/b}+\left(\displaystyle \mint{-}_{B(x,\delta)}|f(y)-f(z)|^{b}\dd z\right)^{1/b}\right).
\end{align*}
Since $B(x,\delta)\subset B(y,2\delta)$ and using the property of the Lebesgue measure that $|B(x,2\delta)|=2^n|B(x,\delta)|$, the previous  inequality implies that for any $x,y\in B(x_0,2R)$ with $|x-y|<\delta$
\begin{equation*}
  |f(x)-f(y)|\leq2^{1+\frac{1}{b}}2^{\frac{n}{b}}\left(\left(\displaystyle \mint{-}_{B(x,2\delta)}|f(x)-f(z)|^{b}\dd z\right)^{1/b}+\left(\displaystyle \mint{-}_{B(y,2\delta)}|f(y)-f(z)|^{b}\dd z\right)^{1/b}\right).  
\end{equation*}
Let $\lambda>1$. By applying Markov's inequality, H\"older's inequality (see \eqref{eq: X-Holders}) and invoking condition \ref{itm: (a)} of Lemma \ref{lem: first prop. of F_R}, we deduce
\begin{align*}
  \left|\left\{x\in B(x_0,2R):\; |f(x)|>\lambda\right\}\right|
  &= \left|\left\{x\in B(x_0,2R):\; |f(x)|^{b}>\lambda^{b}\right\}\right|\\
  &\leq\frac{1}{\lambda^{b}}\int_{B(x_0,2R)}|f(x)|^{b}\dd x\\
  &\leq\frac{1}{\lambda^{b}}\||f|^b\|_{X^b}\|\chi_{B(x_0,2R)}\|_{(X^b)'} \quad\\
  &=\frac{1}{\lambda^{b}}\|f\|_X^b\|\chi_{B(x_0,2R)}\|_{(X^b)'}.
\end{align*}
Thus, if $\lambda>1$ is large enough then 
\begin{equation*}
  \sup_{f\in\mathcal{F}_R} \left|\left\{x\in B(x_0,2R):\; |f(x)|>\lambda\right\}\right|<\varepsilon.
\end{equation*}
We define
\begin{align*}
  E(f)&:=\left\{x\in B(x_0,2R):\; |f(x)|>\lambda\right\}\\
  &\qquad\cup\left\{x\in B(x_0,2R):\;\left(\displaystyle \mint{-}_{B(x,2\delta)}|f(x)-f(y)|^{b}\dd y\right)^{1/b}>\varepsilon\;\right\}\subset B(x_0,2R).
\end{align*}
Then,
\begin{align*}
  &|E(f)|<2\varepsilon,\quad\text{as long as}\;\;\lambda>1 \;\;\text{is large enough},\\ 
  &|f(x)|\leq\lambda,\quad\text{for}\;\;x\in B(x_0,2R)\setminus E(f),\\
  &|f(x)-f(y)|\leq 2^{2+\frac{1}{b}}2^{\frac{n}{b}}\varepsilon,\quad\text{for}\;\;x,y\in B(x_0,2R)\setminus E(f),\;|x-y|<\delta.
\end{align*}
Hence, it follows from Lemma \ref{lem: X-Krotov's result} that $\mathcal{F}_R$ is totally bounded in $L^0(B(x_0,2R))$. This completes the proof of the lemma.
\end{proof}

\subsection{The setting of weighted variable Lebesgue spaces}

Let us first recall the following two auxiliary results that will be needed for the proof of Theorem~\ref{thm: RK criterion}. 

\begin{lemma}\label{lem: HL}
Let $p(\cdot)\in\mathscr{P}_0\cap\mathrm{LH}$, $\widetilde{q}\in(0,p_{-})$ and $w^{\widetilde{q}}\in\mathcal{A}_{\frac{p(\cdot)}{\widetilde{q}}}$. Then
\begin{equation*}
  \|M_{\widetilde{q}}(f)\|_{L^{p(\cdot)}(w)}\leq C\|f\|_{L^{p(\cdot)}(w)},
\end{equation*}
where $M_{\widetilde{q}}(f)(x):=\sup\limits_{r>0}\left(\displaystyle \mint{-}_{B(x,r)}|f(y)|^{\widetilde{q}}\; \dd y\right)^{1/\widetilde{q}}$ is the $\widetilde{q}$-Hardy--Littlewood maximal operator and $C$ is a positive constant that depends on $n$, $p(\cdot)$, $\widetilde{q}$, $C_{\log}(p)$ and the weight $w$.
\end{lemma}

\begin{remark}\label{rmk:bdd. of HL}
Notice that the proof of Lemma \ref{lem: HL} follows by combining Lemma \ref{lem:homog} with \cite[Theorem 1.5]{CUFN2012} (see also \cite[Theorem 1.3]{CruzUribe2011}).
\end{remark}

Let us observe that the weighted variable Lebesgue spaces considered in this paper are quasi-Banach function spaces satisfying additional properties.

\begin{lemma}\label{lem:Weighted_Variable_Lebesgue_QBFS}
Let $p(\cdot)\in\mathscr{P}_0$ and let $w$ be a weight on $\R^n$. Then, $L^{p(\cdot)}(w)$ is a quasi-Banach function space over $\R^n$ (equipped with Lebesgue measure) that has absolutely continuous quasi-norm.
\end{lemma}

\begin{proof}
It follows from \cite[page 42]{Nieareth2023} that $L^{p(\cdot)}(w)$ satisfies the ideal and saturation property of Definition \ref{Def: quasi-BFS}. Now, we know that $L^{p(\cdot)}(\R^n)$ has the Fatou property, see Proposition \ref{prop:monotone_convergence}, and has an absolutely continuous quasi-norm, see for instance \cite[Remark 3.4 (ii)]{WYY2020}. Then, by \cite[Proposition 3.12]{LoNie2024} and \cite[Proposition 3.17]{LoNie2024} we immediately obtain that $L^{p(\cdot)}(w)$ has the Fatou and absolutely continuous quasi-norm properties.
\end{proof}

Now, we are in a position to present the proof of the weighted Riesz--Kolmogorov Theorem \ref{thm: RK criterion} in the variable exponent setting.

\begin{proof}[Proof of Theorem~\ref{thm: RK criterion}]
The proof follows immediately from Theorem \ref{thm: general RK criterion} 
applied to $X=L^{p(\cdot)}(w)$. Indeed, the assumptions of the Theorem \ref{thm: general RK criterion} are satisfied in view of Lemmata \ref{lem: HL} and \ref{lem:Weighted_Variable_Lebesgue_QBFS} and the fact that $w^{\widetilde{q}}\in\mathcal{A}_{\frac{p(\cdot)}{\widetilde{q}}}$.
\end{proof}

\section{Interpolation of compactness for multilinear operators in the variable exponent setting}\label{sec: interpolation}

\subsection{Some terminology and notation}

In this section, we will establish the weighted interpolation for multilinear compact operators within the framework of variable exponent spaces, see Theorem \ref{thm:interpolation_compactness} below. To streamline the presentation, we divide the proof of this result into several subsections below. Now, we begin by defining two important concepts that will be used subsequently.

By a \emph{generalized simple function} we mean a function $f:\Omega\to\C$ of the form
\begin{equation*}
  f = \sum_{i=1}^{n}a_i\chi_{E_i}
\end{equation*}
where $n\in\N$, $a_1,\ldots,a_n\in\C$ and $E_1,\ldots, E_n\subset\Omega$ are measurable sets. By a \emph{simple function}  we mean a function $f:\Omega\to\C$ of the form
\begin{equation*}
  f = \sum_{i=1}^{n}a_i\chi_{E_i}
\end{equation*}
where $n\in\N$, $a_1,\ldots,a_n\in\C$ and $E_1,\ldots, E_n\subset\Omega$ are measurable sets of \emph{finite} measure.

\subsection{Interpolation of boundedness I}
The first key step is to provide the proof of the following counterpart of \cite[Theorem 3.1]{CaoOlivoYabuta2022} in the variable exponent setting.

\begin{theorem}\label{thm:interpolation_boundedness}
Let $p_{0,j}(\cdot),p_{1,j}(\cdot)\in(\mathscr{P}\cap\mathrm{LH})(\Omega_j)$, $j=1,\ldots,m$ and $q_0(\cdot),q_1(\cdot)\in(\mathscr{P}_0\cap\mathrm{LH})(\Omega)$. Let $w_{0,j},w_{1,j}$, $j=1,\ldots,m$ be weights on $\Omega_j$ and let also $v_{0},v_{1}$ be weights on $\Omega$. Let $T$ be a $m$-linear operator such that
\begin{equation*}
        T : L^{p_{0,1}(\cdot)}(w_{0,1})\times\cdots\times L^{p_{0,m}(\cdot)}(w_{0,m}) \to L^{q_0(\cdot)}(v_0)
\end{equation*}
and
\begin{equation*}
  T : L^{p_{1,1}(\cdot)}(w_{1,1})\times\cdots\times L^{p_{1,m}(\cdot)}(w_{1,m}) \to L^{q_1(\cdot)}(v_1)
\end{equation*}
boundedly. For each $i=0,1$, let $M_i\in(0,\infty)$ with
\begin{equation}\label{eq:interpolation_assumption}
  \Vert T(\vec{f})\Vert_{L^{q_i(\cdot)}(v_i)} \leq M_i\prod_{j=1}^{m}\Vert f_j\Vert_{L^{p_{i,j}(\cdot)}(w_{i,j})},
\end{equation}
for all $\vec{f}=(f_1,\ldots,f_m)\in L^{p_{i,1}(\cdot)}(w_{i,1})\times\cdots\times L^{p_{i,m}(\cdot)}(w_{i,m})$. Let $\theta\in(0,1)$ and define
\begin{align*}
  &\frac{1}{p_{j}(\cdot)} := \frac{1-\theta}{p_{0,j}(\cdot)} + \frac{\theta}{p_{1,j}(\cdot)}, \quad j=1,\ldots,m\\
  &\frac{1}{q(\cdot)} := \frac{1-\theta}{q_{0}(\cdot)} + \frac{\theta}{q_{1}(\cdot)},
\end{align*}
as well as
\begin{align*}
  &w_{j} := w_{0,j}^{1-\theta}w_{1,j}^{\theta},\quad j=1,\ldots,m,\\
  & v := v_{0}^{1-\theta}v_{1}^{\theta}.
\end{align*}
Then, we have
\begin{equation*}
  T : L^{p_{1}(\cdot)}(w_{1})\times\cdots\times L^{p_{m}(\cdot)}(w_{m}) \to L^{q(\cdot)}(v)
\end{equation*}
boundedly and it holds
\begin{equation*}
  \Vert T(\vec{f})\Vert_{L^{q(\cdot)}(v)} \leq M_0^{1-\theta}M_{1}^{\theta}\prod_{j=1}^{m}\Vert f_j\Vert_{L^{p_{j}(\cdot)}(w_{j})},
\end{equation*}
for all $\vec{f}=(f_1,\ldots,f_m)\in L^{p_{1}(\cdot)}(w_{1})\times\cdots\times L^{p_{m}(\cdot)}(w_{m})$.
\end{theorem}

To prove Theorem~\ref{thm:interpolation_boundedness}, we adapt the strategy of the proof of \cite[Theorem 3.1]{CaoOlivoYabuta2022}. However, several steps are significantly more complicated in the present variable exponent setting.

The lengthy proof of Theorem~\ref{thm:interpolation_boundedness} begins with a couple of reductions.

\begin{lemma}\label{lem:first_red}
Assume the same setup as in Theorem~\ref{thm:interpolation_boundedness}. For each $j=1,\ldots,m$ define
\begin{equation*}
  A_j := \{f_j\text{ simple functions on }\Omega_j:~\Vert f_j\Vert_{L^{p_{i,j}(\cdot)}(w_{i,j})}<\infty\text{ for all }i=0,1\}.
\end{equation*}
Then, to prove Theorem~\ref{thm:interpolation_boundedness}, it suffices to show that
\begin{equation}\label{eq:central_estimate}
  \Vert T(\vec{f})\Vert_{L^{q(\cdot)}(v)} \leq M_0^{1-\theta}M_{1}^{\theta}\prod_{j=1}^{m}\Vert f_j\Vert_{L^{p_{j}(\cdot)}(w_{j})},\quad\forall \vec{f}\in\vec{A}:=A_1\times\cdots\times A_m.
\end{equation}
\end{lemma}

\begin{proof}
Assume that \eqref{eq:central_estimate} has been shown. Note that it follows from Lemma \ref{lem:density_simple_functions} below that
\begin{equation*}
  A_j \subset L^{p_{j}(\cdot)}(w_{j}),\quad\forall j=1,\ldots,m.
\end{equation*}
Since $T$ is $m$-linear, an application of the well known extension theorem for bounded multilinear operators on quasi-normed vector spaces yields that there exists a unique $m$-linear operator
\begin{equation*}
\widetilde{T} : L^{p_{1}(\cdot)}(w_{1})\times\cdots\times L^{p_{m}(\cdot)}(w_{m})\to L^{q(\cdot)}(v)
\end{equation*}
such that
\begin{equation*}
  \widetilde{T}|_{\vec{A}} = T|_{\vec{A}}.
\end{equation*}
    
It remains to check that $\widetilde{T}$ agrees with $T$ on the entirety of the intersection of the respective domains of definitions. So let $i\in\{0,1\}$ and
\begin{equation*}
  \vec{f} \in (L^{p_{1}(\cdot)}(w_{1})\times\cdots\times L^{p_{m}(\cdot)}(w_{m}))\cap(L^{p_{i,1}(\cdot)}(w_{i,1})\times\cdots\times L^{p_{i,m}(\cdot)}(w_{i,m}))
\end{equation*}
be arbitrary. Then, by Lemma~\ref{lem:density_simple_functions} we have that for all $j=1,\ldots,m$ there exists a sequence $(f_{j,k})^{\infty}_{k=1}$ in $A_j$ such that
\begin{equation*}
  f_{j,k}\to f_j\quad\text{ as }k\to\infty\text{ in }L^{p_{i,j}(\cdot)}(w_{i,j})
\end{equation*}
and
\begin{equation*}
  f_{j,k}\to f_j\quad\text{ as }k\to\infty\quad\text{in }L^{p_{j}(\cdot)}(w_{j}).
\end{equation*}
By the multilinearity of $T$ and \eqref{eq:central_estimate}, we conclude that
\begin{equation*}
  T(\vec{f}_{k}) \to T(\vec{f})\quad\text{ as }k\to\infty\quad\text{in }L^{q_{i}(\cdot)}(v_{i})
\end{equation*}
and
\begin{equation*}
  \widetilde{T}(\vec{f}_{k}) \to \widetilde{T}(\vec{f})\quad\text{ as }k\to\infty\text{ in }L^{q(\cdot)}(v),
\end{equation*}
that is
\begin{equation*}
  T(\vec{f}_{k})v_i \to T(\vec{f})v_i\quad\text{ as }k\to\infty\quad\text{in }L^{q_{i}(\cdot)}
\end{equation*}
and
\begin{equation*}
  \widetilde{T}(\vec{f}_{k})v \to \widetilde{T}(\vec{f})v\quad\text{ as }k\to\infty\text{ in }L^{q(\cdot)}.
\end{equation*}
Then, by Proposition~\ref{prop:norm_convergence_to_subs_pointwise_convergence} we have that there exists a subsequence $(\vec{f}_{m_k})^{\infty}_{k=1}$ such that
\begin{equation*}
  T(\vec{f}_{m_k})v_i \to T(\vec{f})v_i\quad\text{ as }k\to\infty\quad\text{a.e. pointwise on }\Omega
\end{equation*}
and
\begin{equation*}
  \widetilde{T}(\vec{f}_{m_k})v \to \widetilde{T}(\vec{f})v\quad\text{ as }k\to\infty\quad\text{a.e. pointwise on }\Omega,
\end{equation*}
and thus
\begin{equation*}
        T(\vec{f}_{m_k}) \to T(\vec{f})\quad\text{ as }k\to\infty\quad\text{a.e. pointwise on }\Omega
\end{equation*}
and
\begin{equation*}
  \widetilde{T}(\vec{f}_{m_k}) \to \widetilde{T}(\vec{f})\quad\text{ as }k\to\infty\quad\text{a.e. pointwise on }\Omega,
\end{equation*}
Since $T(\vec{f}_{m_k})=\widetilde{T}(\vec{f}_{m_k})$ a.e., for all $k=1,2,\ldots$, we conclude that $T(\vec{f})=\widetilde{T}(\vec{f})$ a.e. This completes the proof.
\end{proof}

\begin{lemma}\label{lem:density_simple_functions_simple_form}
Let $p(\cdot),q(\cdot)\in(\mathscr{P}\cap\mathrm{LH})(\Omega)$ and let $w,v$ be weights on $\Omega$.
\begin{enumerate}[label=\textup{(\arabic*)}]
\item Define
\begin{equation*}
  A := \{f\text{ simple function on }\Omega:~\Vert f\Vert_{L^{p(\cdot)}(w)}<\infty\text{ and }\Vert f\Vert_{L^{q(\cdot)}(v)}<\infty\}.
\end{equation*}
Then, we have
\begin{equation*}
  A = \mathrm{Span}(\{\chi_{E}:~E\text{ measurable subset of }\Omega\text{ with }|E|<\infty\text{ and }\chi_{E}\in L^{p(\cdot)}(w)\cap L^{q(\cdot)}(v)\}).
\end{equation*}
\item Let $f\in L^{p(\cdot)}(w)\cap L^{q(\cdot)}(v)$. Then, there exists a sequence $(f_n)^{\infty}_{n=1}$ of simple functions in $L^{p(\cdot)}(w)\cap L^{q(\cdot)}(v)$, such that
\begin{equation*}
  f_n\to f\quad\text{as }n\to\infty\quad\text{in }L^{p(\cdot)}(w)
\end{equation*}
and
\begin{equation*}
  f_n\to f\quad\text{as }n\to\infty\quad\text{in }L^{q(\cdot)}(v).
\end{equation*}
\end{enumerate}
\end{lemma}

\begin{proof}
\begin{enumerate}
\item First of all, it is clear that $A$ is a vector subspace of $L^{p(\cdot)}(w)\cap L^{q(\cdot)}(v)$ and $\chi_{E} \in A$, for every measurable subset of $E$ of $\Omega$ with $|E|<\infty$, $\Vert w\chi_{E}\Vert_{L^{p(\cdot)}}<\infty$ and $\Vert v\chi_{E}\Vert_{L^{q(\cdot)}}<\infty$. This shows the inclusion $\supset$.

We now prove the inclusion $\subset$. Let $f\in A$ be arbitrary. We can write
\begin{equation*}
  f = \sum_{i=1}^{n}a_i\chi_{E_i}
\end{equation*}
where $n\in\N$, $a_1,\ldots,a_n\in\C\setminus\{0\}$ and $E_1,\ldots, E_n\subset\Omega$ are pairwise disjoint measurable sets of finite measure. Then, for all $i=1,\ldots,n$ we have
\begin{equation*}
  \Vert\chi_{E_i}\Vert_{L^{p(\cdot)}(w)} \leq \frac{1}{|a_i|}\Vert|f|\Vert_{L^{p(\cdot)}(w)} = \frac{1}{|a_i|}\Vert f\Vert_{L^{p(\cdot)}(w)}<\infty.
\end{equation*}
Similarly we obtain $\Vert\chi_{E_i}\Vert_{L^{q(\cdot)}(v)}<\infty$, concluding the proof.

\item It is well known that there exists a sequence $(h_n)^{\infty}_{n=1}$ of generalized simple functions on $\Omega$ such that $|h_n|\leq |f|$, for all $n=1,2,\ldots$ and $h_n\to f$ pointwise on $\Omega$. Write
\begin{equation*}
  \Omega = \bigcup_{n=1}^{\infty}\Omega_n,
\end{equation*}
where $(\Omega_n)^{\infty}_{n=1}$ is an increasing sequence of measurable subsets of $\Omega$ of finite measure, and let $g_n := h_n\chi_{\Omega_n}$, $n=1,2,\ldots$. Then $g_1,g_2,\ldots$ are simple functions, $|g_n|\leq |f|$, for all $n=1,2,\ldots$ and $g_n\to f$ pointwise on $\Omega$.
        
Since $|g_n|\leq |f|$, we deduce $g_n\in L^{p(\cdot)}(w)$, for all $n=1,2,\ldots$. Moreover, since
\begin{equation*}
  |g_nw| \leq |fw|,\quad\forall n=1,2,\ldots
\end{equation*}
and $fw\in L^{p(\cdot)}$, by Proposition \ref{prop:dominated_convergence} we deduce that $g_nw\to fw$ as $n\to\infty$ in $L^{p(\cdot)}$, that is $g_n\to f$ as $n\to\infty$ in $L^{p(\cdot)}(w)$. The exact same argument shows that $g_n\to f$ as $n\to\infty$ in $L^{q(\cdot)}(v)$, concluding the proof.
\end{enumerate}
\end{proof}

The following is an analogue of \cite[Lemma 3.3]{CaoOlivoYabuta2022}.

\begin{lemma}\label{lem:density_simple_functions}
Let $p_{0}(\cdot),p_{1}(\cdot)\in(\mathscr{P}\cap\mathrm{LH})(\Omega)$. Let $w_{0},w_{1}$ be weights on $\Omega$. Let $\theta\in(0,1)$ and define
\begin{equation*}
        \frac{1}{p(\cdot)} := \frac{1-\theta}{p_{0}(\cdot)} + \frac{\theta}{p_{1}(\cdot)}
\end{equation*}
as well as
\begin{equation*}
        w := w_{0}^{1-\theta}w_{1}^{\theta}.
\end{equation*}
Define
\begin{equation*}
  A := \{f\text{ simple function on }\Omega:~\Vert f\Vert_{L^{p_{i}(\cdot)}(w_{i})}<\infty\text{ for all }i=0,1\}.
\end{equation*}
\begin{enumerate}[label=\textup{(\arabic*)}]
  \item The family $A$ is dense in $L^{p(\cdot)}(w)$.
  \item Let $i\in\{0,1\}$ be arbitrary. Then, for all $f\in L^{p_i(\cdot)}(w_i)\cap L^{p(\cdot)}(w)$, there exists a sequence $(f_k)_{k=1}^{\infty}$ in $A$ such that $f_k\to f$ as $k\to\infty$ in $L^{p_i(\cdot)}(w_i)$ and $f_k\to f$ as $k\to\infty$ in $L^{p(\cdot)}(w)$.
\end{enumerate}
\end{lemma}

\begin{proof}
    We adapt the arguments from the proof of \cite[Lemma 3.3]{CaoOlivoYabuta2022}. We will be suppressing the dependence of implicit constants on $p_{0}(\cdot),p_{1}(\cdot)$ and $\theta$ from the notation.
\begin{enumerate}[label=\textup{(\arabic*)}]
\item First of all, let us note that $A$ is a subset of $L^{p(\cdot)}(w)$. Indeed, for all $f\in A$, by Lemma \ref{lem:homog} and H\"{o}lder's inequality, Lemma~\ref{lem:Hoelder}, we have
\begin{align*}
  \Vert f\Vert_{L^{p(\cdot)}(w)} &= \Vert fw\Vert_{L^{p(\cdot)}}
  =\Vert |fw_{0}|^{1-\theta}|fw_{1}|^{\theta}\Vert_{L^{p(\cdot)}}\\
  &\lesssim \Vert |fw_{0}|^{1-\theta}\Vert_{L^{p_{0}(\cdot)/(1-\theta)}}\Vert |fw_{1}|^{\theta}\Vert_{L^{p_{1}(\cdot)/\theta}}\\
  &=\Vert |fw_{0}|\Vert_{L^{p_{0}(\cdot)}}^{1-\theta}\Vert |fw_{1}|\Vert_{L^{p_{1}(\cdot)}}^{\theta}
  =\Vert f\Vert_{L^{p_{0}(\cdot)}(w_{0})}^{1-\theta}\Vert f\Vert_{L^{p_{1}(\cdot)}(w_{1})}^{\theta} < \infty.
\end{align*}
To prove that $A$ is dense in $L^{p(\cdot)}(w)$, by Lemma~\ref{lem:density_simple_functions_simple_form} we deduce that it suffices to prove that for all measurable subsets $E$ of $\Omega$ with $|E|<\infty$ and $\Vert w\chi_{E}\Vert_{L^{p(\cdot)}}<\infty$ and for all $\eps>0$ we can find $f\in A$ with $\Vert f-\chi_{E}\Vert_{L^{p(\cdot)}(w)}<\eps$. Since $\Vert w\chi_{E}\Vert_{L^{p(\cdot)}}<\infty$ and $\Vert\cdot\Vert_{L^{p(\cdot)}}$ is an absolutely continuous norm (see for instance~\cite[page 73]{CF2013}), by \cite[Chapter 1, Lemma 3.4]{Interpolation_operators} we have that there exists $\delta>0$, such that for all measurable subsets $F$ of $\Omega$ with $|F|<\delta$ there holds $\Vert w\chi_{E}\chi_{F}\Vert_{L^{p(\cdot)}}<\eps$. Set
\begin{equation*}
  F_N := \{x\in E:~w_0(x)^{p_0(x)}>N\text{ or }w_1(x)^{p_1(x)}>N\},\quad\forall N=1,2,\ldots.
\end{equation*}
Then, $(F_{N})^{\infty}_{N=1}$ is a decreasing sequence of subsets of $E$ with
\begin{equation*}
  \left|\bigcap_{N=1}^{\infty}F_{N}\right| = |\{x\in E:~w_0(x)^{p_0(x)}=\infty\text{ or }w_1(x)^{p_1(x)}=\infty\}| = 0.
\end{equation*}
Therefore, there exists $N\geq 1$ such that $|F_{N}| < \delta$. Set $f := \chi_{E\setminus F_{N}}$. Observe that since $|E|<\infty$, by construction we have $f\in A$. Moreover, we compute
\begin{align*}
  \Vert \chi_{E} - \chi_{E\setminus F_{N}}\Vert_{L^{p(\cdot)}(w)}
  = \Vert w\chi_{F_{N}}\Vert_{L^{p(\cdot)}} = \Vert w\chi_{E}\chi_{F_{N}}\Vert_{L^{p(\cdot)}} < \eps,
\end{align*}
for $|F_{N}|<\delta$, concluding the proof.\label{itm: first}

\item The proof is similar to the second part of the proof of part \ref{itm: first}, so we only explain the necessary changes. First of all, by Lemma~\ref{lem:density_simple_functions_simple_form} we deduce that it suffices to prove that for all measurable subsets $E$ of $\Omega$ with $|E|<\infty$, $\Vert w_i\chi_{E}\Vert_{L^{p_i(\cdot)}}<\infty$ and $\Vert w\chi_{E}\Vert_{L^{p(\cdot)}}<\infty$ and for all $\eps>0$ we can find $f\in A$ with $\Vert f-\chi_{E}\Vert_{L^{p_i(\cdot)}(w_i)}<\eps$ and $\Vert f-\chi_{E}\Vert_{L^{p(\cdot)}(w)}<\eps$.
Since $\Vert w_i\chi_{E}\Vert_{L^{p_i(\cdot)}}<\infty$ and $\Vert w\chi_{E}\Vert_{L^{p(\cdot)}}<\infty$, by \cite[Chapter 1, Lemma 3.4]{Interpolation_operators} we have that there exists $\delta>0$, such that for all measurable subsets $F$ of $\Omega$ with $|F|<\delta$ there holds $\Vert w_i\chi_{E}\chi_{F}\Vert_{L^{p_i(\cdot)}}<\eps$ and $\Vert w\chi_{E}\chi_{F}\Vert_{L^{p(\cdot)}}<\eps$. The rest of the proof proceeds exactly as in part \ref{itm: first}.
\end{enumerate}
\end{proof}

\begin{lemma}\label{lem:second_red}
Assume the same setup as in Lemma~\ref{lem:first_red}. Then, in proving \eqref{eq:central_estimate} we may assume without loss of generality that $|\Omega| < \infty$.
\end{lemma}

\begin{proof}
By assumption, we can write
\begin{equation*}
  \Omega = \bigcup_{k=1}^{\infty}\Omega_k,
\end{equation*}
where $(\Omega_k)^{\infty}_{k=1}$ is an increasing sequence of measurable subsets of $\Omega$ with finite measure. For each $k=1,2,\ldots$, it is clear that the operator $T_{k} := \chi_{\Omega_k}T$ satisfies the assumptions of Theorem \ref{thm:interpolation_boundedness}. Moreover, by Proposition~\ref{prop:monotone_convergence} we have
\begin{equation*}
  \lim_{k\to\infty}\Vert g\chi_{\Omega_k}\Vert_{L^{q(\cdot)}(v)} = 
  \lim_{k\to\infty}\Vert |gv\chi_{\Omega_k}|\Vert_{L^{q(\cdot)}}
  = \Vert |gv|\Vert_{L^{q(\cdot)}}
  =\Vert g\Vert_{L^{q(\cdot)}(v)}.
\end{equation*}
for all measurable functions $g$ on $\Omega$. Therefore, in proving \eqref{eq:central_estimate} we may assume without loss of generality that $|\Omega| < \infty$.
\end{proof}

\begin{proposition}\label{prop:third_red}
Assume the same setup as in Lemma~\ref{lem:first_red}. Moreover, assume that $|\Omega| < \infty$. Then, in proving \eqref{eq:central_estimate} we may assume without loss of generality that all weights $w_{i,j}$, $j=1,\ldots,m$, $i=0,1$ and $v_0$, $v_1$ are simple functions.
\end{proposition}

Proposition~\ref{prop:third_red} follows from the following lemma.

\begin{lemma}
Assume that in any setup as in Lemma~\ref{lem:first_red} with $|\Omega| < \infty$ and such that all weights $w_{i,j}$, $j=1,\ldots,m$, $i=0,1$ and $v_0$, $v_1$ are simple functions, we have managed to establish \eqref{eq:central_estimate}.
\begin{enumerate}[label=\textup{(\arabic*)}]
\item In any setup as in Lemma~\ref{lem:first_red} with $|\Omega| < \infty$, the following is true:

Let $\eps\in(0,1]$ be arbitrary. Let $F_j$ be a measurable subset of $\Omega_j$ of finite measure for $j=1,\ldots,m$. Set
\begin{equation*}
  \widetilde{F}_j := \left\{x\in F_j:~\eps \leq w_{0,j}(x),~w_{1,j}(x)\leq\frac{1}{\eps}\right\},\quad j=1,\ldots,m.
\end{equation*}
For each $j=1,\ldots,m$, let $\widetilde{w}_{0,j}$, $\widetilde{w}_{1,j}$ be nonnegative simple functions on $\Omega_j$ vanishing outside of $\widetilde{F}_j$ with $w_{0,j}(x)\leq \widetilde{w}_{0,j}(x)$ and $w_{1,j}(x)\leq \widetilde{w}_{1,j}(x)$, for all $x\in\widetilde{F}_j$. Moreover, let $\widetilde{v}_{0}$, $\widetilde{v}_{1}$ be nonnegative simple functions on $\Omega$ such that $\widetilde{v}_i\leq v_i$ for $i=0,1$. Set $\widetilde{w}_j := \widetilde{w}_{0,j}^{1-\theta}\widetilde{w}_{1,j}^{\theta}$, $j=1,\ldots,m$ and $\widetilde{v} := \widetilde{v}_{0}^{1-\theta}\widetilde{v}_{1}^{\theta}$. Then, for all tuples $\vec{f} = (f_1,\ldots,f_m)$ such that $f_j$ is a simple function on $\Omega_j$ vanishing outside of $\widetilde{F}_{j}$ for $j=1,\ldots,m$, we have
\begin{equation*}
  \Vert T(\vec{f})\widetilde{v}\Vert_{L^{q(\cdot)}} \leq M_0^{1-\theta}M_{1}^{\theta}\prod_{j=1}^{m}\Vert f_j\widetilde{w}_j\Vert_{L^{p_{j}(\cdot)}}.
\end{equation*}

\item In any setup as in Lemma~\ref{lem:first_red} with $|\Omega| < \infty$, the following is true:

Let $\eps\in(0,1]$ be arbitrary. Let $F_j$ be a measurable subset of $\Omega_j$ of finite measure for $j=1,\ldots,m$. Set
\begin{equation*}
  \widetilde{F}_j := \left\{x\in F_j:~\eps \leq w_{0,j}(x),~w_{1,j}(x)\leq\frac{1}{\eps}\right\},\quad j=1,\ldots,m.
\end{equation*}
Then, for all tuples $\vec{f} = (f_1,\ldots,f_m)$ such that $f_j$ is a simple function on $\Omega_j$ vanishing outside of $\widetilde{F_j}$ for $j=1,\ldots,m$, we have
\begin{equation*}
  \Vert T(\vec{f})\Vert_{L^{q(\cdot)}(v)} \leq M_0^{1-\theta}M_{1}^{\theta}\prod_{j=1}^{m}\Vert f_j\Vert_{L^{p_{j}}(w_j)}.
\end{equation*}

\item In any setup as in Lemma~\ref{lem:first_red} with $|\Omega| < \infty$, the main estimate \eqref{eq:central_estimate} holds.
\end{enumerate}
\end{lemma}

\begin{proof}
\begin{enumerate}[label=\textup{(\arabic*)}]
\item Set
\begin{equation*}
  \widetilde{\Omega} := \{\widetilde{v}\neq0\} = \{\widetilde{v}_0\neq0\}\cap\{\widetilde{v}_1\neq0\}.
\end{equation*}
Observe that
\begin{equation*}
  \Vert T(\vec{f})\Vert_{L^{q_i(\cdot)}(\widetilde{\Omega}, \widetilde{v}_i)} \leq M_i\prod_{j=1}^{m}\Vert f_j\Vert_{L^{p_{i,j}(\cdot)}(\widetilde{F}_j,\widetilde{w}_{i,j})},
\end{equation*}
for all $\vec{f}=(f_1,\ldots,f_m)\in L^{p_{i,1}(\cdot)}(\widetilde{F}_1,\widetilde{w}_{i,1})\times\cdots\times L^{p_{i,m}(\cdot)}(\widetilde{F}_m,\widetilde{w}_{i,m})$, for all $i=0,1$. Moreover, notice that if $f$ is any simple function on $\Omega_j$ vanishing outside of $\widetilde{F}_{j}$, then clearly
\begin{equation*}
  \Vert f\Vert_{L^{p_{i,j}(\cdot)}(\widetilde{F}_j,\widetilde{w}_{i,j})} < \infty,\quad\forall i=0,1,
\end{equation*}
for all $j=1,\ldots,m$. Therefore, the desired estimate follows immediately by applying the assumption.\label{itm: (first)}

\item For each $i=0,1$ and $j=1,\ldots,m$ there exists a decreasing sequence $(w_{i,j,n})^{\infty}_{n=1}$ of nonnegative simple functions on $\Omega_j$ vanishing outside of $\widetilde{F}_j$ with $\lim_{n\to\infty}w_{i,j,n}(x)=w_{i,j}(x)$ for each $x\in\widetilde{F}_j$. Moreover, for each $i=0,1$ there exists an increasing sequence $(v_{i,n})^{\infty}_{n=1}$ of nonnegative simple functions on $\Omega$ such that $v_{i,n}\to v$ as $n\to\infty$ pointwise on $\Omega$. Set
\begin{equation*}
  w_{j,n} := w_{0,j,n}^{1-\theta}w_{1,j}^{\theta},\quad j=1,\ldots,m
\end{equation*}
and
\begin{equation*}
  v_n := v_{0,n}^{1-\theta}v_{1,n}^{\theta}
\end{equation*}
for $n=1,2,\ldots$.
        
Consider a tuple $\vec{f} = (f_1,\ldots,f_m)$ such that $f_j$ is a simple function on $\Omega_j$ vanishing outside of $\widetilde{F}_{j}$ for $j=1,\ldots,m$. Then, by part \ref{itm: (first)} we have
\begin{equation}\label{eq:inter}
  \Vert T(\vec{f})v_k\Vert_{L^{q(\cdot)}} \leq M_0^{1-\theta}M_{1}^{\theta}\prod_{j=1}^{m}\Vert f_jw_{j,n}\Vert_{L^{p_{j}(\cdot)}},
\end{equation}
for all $n,k=1,2,\ldots$. Observe that there exists a constant $0<C<\infty$ such that
\begin{equation*}
  |f_jw_{j,n}| \leq C\chi_{\widetilde{F}_j},\quad\forall n=1,2,\ldots,
\end{equation*}
and $\chi_{\widetilde{F}_j}\in L^{p_{j}(\cdot)}$, for all $j=1,\ldots,m$. Therefore, by letting $n\to\infty$ in \eqref{eq:inter} and using Proposition \ref{prop:dominated_convergence} we obtain
\begin{equation}\label{eq:inter1}
  \Vert T(\vec{f})v_k\Vert_{L^{q(\cdot)}} \leq M_0^{1-\theta}M_{1}^{\theta}\prod_{j=1}^{m}\Vert f_j\Vert_{L^{p_{j}(\cdot)}(w_j)},
\end{equation}
for all $k=1,2,\ldots$. Letting then $k\to\infty$ in \eqref{eq:inter1} and using Proposition \ref{prop:monotone_convergence} we can conclude the proof.\label{itm: (second)}

\item Let $\vec{f}=(f_1,\ldots,f_m)\in\vec{A}$ be arbitrary. For each $j=1,\ldots,m$ pick a measurable subset $F_j$ of $\Omega_j$ of finite measure such that $f_j$ vanishes outside of $\Omega_j$. For all $j=1,\ldots,m$ and all $k=1,2,\ldots$ define
\begin{equation*}
            F_{j,k} := \left\{x\in F_j:~\frac{1}{k}\leq w_{0,j}(x),~w_{1,j}(x)\leq k\right\},\quad f_{j,k} := f_{j}\chi_{F_{j,k}}.
\end{equation*}
By Proposition \ref{prop:dominated_convergence} we obtain
\begin{equation*}
  f_{j,k}\to f_j\quad\text{as }k\to\infty\quad\text{in }L^{p_{j,0}(\cdot)}(w_{j,0}).
\end{equation*}
Since the operator $T$ is $m$-linear, from \eqref{eq:interpolation_assumption} we deduce
\begin{equation*}
  T(\vec{f}_k)\to T(\vec{f})\quad\text{as }k\to\infty\quad\text{in }L^{q_{0}(\cdot)}(v_{0}).
\end{equation*}
Additionally, by Proposition \ref{prop:dominated_convergence} we obtain
\begin{equation*}
  f_{j,k}\to f_j\quad\text{as }k\to\infty\quad\text{in }L^{p_{j}(\cdot)}(w_{j}).
\end{equation*}
Notice that for all $n,k\geq1$, by part \ref{itm: (second)} we deduce
\begin{equation*}
  \Vert T(g_1,\ldots,g_m)\Vert_{L^{q(\cdot)}(v)} \leq M_0^{1-\theta}M_{1}^{\theta}\prod_{j=1}^{m}\Vert g_j\Vert_{L^{p_{j}}(w_j)}
\end{equation*}
for \emph{all} choices
\begin{equation*}
  g_j \in \{f_{j,n}, f_{j,k}, f_{j,n}-f_{j,k}\},\quad j=1,\ldots,m.
\end{equation*}
Thus, since $T$ is $m$-linear, we deduce that $(T(\vec{f}_k))^{\infty}_{k=1}$ is a Cauchy sequence in $L^{q(\cdot)}(v)$. As in the proof of Lemma~\ref{lem:first_red} we deduce that the limit of $(T(\vec{f}_k))^{\infty}_{k=1}$ in $L^{q(\cdot)}(v)$ must be $T(\vec{f})$. Since part \ref{itm: (second)} implies
\begin{equation*}
  \Vert T(f_{1,k},\ldots,f_{m,k})\Vert_{L^{q(\cdot)}(v)} \leq M_0^{1-\theta}M_{1}^{\theta}\prod_{j=1}^{m}\Vert f_{j,k}\Vert_{L^{p_{j}}(w_j)},\quad\forall k=1,2,\ldots,
\end{equation*}
by letting $k\to\infty$ we finally deduce \eqref{eq:central_estimate}.
\end{enumerate}
\end{proof}

\begin{lemma}\label{lem:fourth_red}
Assume the same setup as in Lemma~\ref{lem:first_red}. Moreover, assume that $|\Omega| < \infty$ and that all weights $w_{i,j}$, $j=1,\ldots,m$, $i=0,1$ and $v_0$, $v_1$ are simple functions. Then, in proving \eqref{eq:central_estimate} we may assume without loss of generality that
\begin{equation*}
  \Vert f_j\Vert_{L^{p_{j}(\cdot)}(w_{j})} = 1,\quad\forall j=1,\ldots,m.
\end{equation*}
\end{lemma}

\begin{proof}
Let $\vec{f}\in\vec{A}$ be arbitrary. If $f_j=0$ a.e.~on $\Omega_j$ for some $j=1,2,\ldots,m$, then we have nothing to show in \eqref{eq:central_estimate}. Therefore, we may assume without loss of generality that $\Vert f_j\Vert_{L^{p_j(\cdot)}(w_j)}>0$, for all $j=1,2,\ldots,m$. Clearly, by the homogeneity of $T$ in each argument, \eqref{eq:central_estimate} is then equivalent to
\begin{equation*}
  \left\Vert T\left(\frac{f_1}{\Vert f_1\Vert_{L^{p_{1}(\cdot)}(w_{1})}},\ldots,\frac{f_m}{\Vert f_m\Vert_{L^{p_{m}(\cdot)}(w_{m})}}\right)\right\Vert_{L^{q(\cdot)}(v)} \leq M_0^{1-\theta}M_{1}^{\theta}.
\end{equation*}
This shows the claim.
\end{proof}

The proof of Theorem \ref{thm:interpolation_boundedness} is completed in the next subsection.

\subsection{Proof of interpolation of boundedness I}

\begin{proposition}\label{prop:central_estimate}
Assume the same setup as in Lemma~\ref{lem:first_red}. Moreover, assume that $|\Omega| < \infty$, that all weights $w_{i,j}$, $j=1,\ldots,m$, $i=0,1$ and $v_0$, $v_1$ are simple functions, and that for all $j=1,\ldots,m$, $f_j$ is a simple function on $\Omega_j$ with $\Vert f_j\Vert_{L^{p_{j}(\cdot)}(w_{j})} = 1$. Then, \eqref{eq:central_estimate} holds.
\end{proposition}

This subsection is devoted to proving Proposition~\ref{prop:central_estimate}. We note that $|\Omega_j|<\infty$, for $w_{0,j}$ is a simple function on $\Omega_j$ that is a.e.~nonnegative, for all $j=1,\ldots,m$.
Our strategy is inspired by the proof of \cite[Lemma 3.2]{CaoOlivoYabuta2022}. We begin by picking
\begin{equation*}
  s \in (0, \min\{(q_{0})_{-}, (q_{1})_{-}, 1\}).
\end{equation*}
A simple calculation shows $s<q_{-}$. By the homogeneity property (Lemma \ref{lem:homog}) we have
\begin{equation*}
  \Vert T(\vec{f})\Vert_{L^{q(\cdot)}(v)} = \Vert |T(\vec{f})v|\Vert_{L^{q(\cdot)}}
    = \Vert |T(\vec{f})v|^{s}\Vert_{L^{r(\cdot)}}^{1/s},
\end{equation*}
where $r(\cdot) := \frac{q(\cdot)}{s} \in \mathscr{P}$. In addition, by \cite[Theorem 2.34]{CF2013} we have
\begin{equation*}
  \Vert |T(\vec{f})v|^{s}\Vert_{L^{r(\cdot)}(v)} \sim 
  \sup_{g}\int_{\Omega}|T(\vec{f})(x)v(x)|^{s}g(x)\dd x,
\end{equation*}
where the supremum is taken over all nonnegative measurable functions $g$ on $\Omega$ with $\Vert g\Vert_{L^{r'(\cdot)}} \leq 1$ and the implicit constants depend only on $r(\cdot)$. Thus, we only have to show that
\begin{equation}\label{eq:to_prove_1}
  \int_{\Omega}|T(\vec{f})(x)v(x)|^{s}g(x)\dd x \leq M_0^{s(1-\theta)}M_{1}^{s\theta}
\end{equation}
for every such function $g$. It is well-known fact that there exists an increasing sequence $(g_n)_{n=1}^{\infty}$ of nonnegative simple functions converging pointwise to $g$. Then $\Vert g_n\Vert_{L^{r'(\cdot)}} \leq 1$, for all $n=1,2,\ldots$ and Proposition \ref{prop:monotone_convergence} yields moreover
\begin{equation*}
  \int_{\Omega}|T(\vec{f})(x)v(x)|^{s}g(x)\dd x = \lim_{n\to\infty}\int_{\Omega}|T(\vec{f})(x)v(x)|^{s}g_n(x)\dd x.
\end{equation*}
Therefore, we only have to prove \eqref{eq:to_prove_1} in the special case that $g$ is a nonnegative simple function.

\subsubsection{Setup of the interpolation scheme}

Define
\begin{equation*}
  S := \{z\in\C:~0< \mathrm{Re}(z)<1\},\quad \bar{S} := \{z\in\C:~0\leq \mathrm{Re}(z)\leq1\}.
\end{equation*}
For each $j=1,\ldots,m$ set $\widetilde{f}_j(\cdot):= f_j(\cdot)w_j(\cdot)$ and consider its polar representation
\begin{equation*}
  \widetilde{f}_j(\cdot) = |\widetilde{f_j}(\cdot)|e^{is_j(\cdot)}
\end{equation*}
for some real-valued simple function $s_j$ on $\Omega_j$. For every $z\in\bar{S}$ define the \emph{complex valued} functions
\begin{equation*}
  m_{z,j}(\cdot) := \frac{1-z}{p_{0,j}(\cdot)} + \frac{z}{p_{1,j}(\cdot)},\quad j=1,\ldots,m,\quad n_{z}(\cdot) := \frac{1-z}{q_{0}(\cdot)} + \frac{z}{q_{1}(\cdot)}.
\end{equation*}
Then, for every $\ell=1,2,\ldots$ and for every $z\in\bar{S}$ define
\begin{equation*}
  \Phi_{\ell}(z) := \int_{\Omega}|U_{\ell}(z, \omega)|^{s}\dd\omega,
\end{equation*}
where
\begin{equation*}
  U_{\ell}(z,\cdot) := \exp\left(\dfrac{z^2-1}{s\ell}\right)M_0^{z-1}M_1^{-z}T(\vec{F}_{z})(\cdot)v_0(\cdot)^{1-z}v_1(\cdot)^{z}G_{z}(\cdot),
\end{equation*}
\begin{equation*}
  G_{z}(\cdot) := g(\cdot)^{t(z,\cdot)/s}\chi_{\{g\neq0\}}(\cdot),\quad t(z,\cdot) := r'(\cdot) - sr'(\cdot)n_{z}(\cdot),
\end{equation*}
\begin{equation*}
  F_{z,j}(\cdot) := |\widetilde{f}_j(\cdot)|^{p_j(\cdot)m_{z,j}(\cdot)}\chi_{\{\widetilde{f}_j\neq0\}}(\cdot)e^{is_j(\cdot)}w_{0,j}(\cdot)^{z-1}w_{1,j}(\cdot)^{-z}.
\end{equation*}

\begin{lemma}
    \label{lem:phi_ell_con_subh}
    Fix $\ell\in\N$. Then $\Phi_{\ell}$ is continuous on $\bar{S}$ and subharmonic in $S$.
\end{lemma}

\begin{proof}
This is perhaps the hardest part of the proof of Theorem~\ref{thm:interpolation_boundedness} and is much more involved than the analogous argument in \cite[Lemma 3.2]{CaoOlivoYabuta2022}. Observe also that since we allow explicitly for the possibility that $q_{-} < 1$, it is not clear how to adapt for instance the ideas from \cite[Lemma 3.1]{Meskhi_2016}, because one had to deal there only with \emph{analytic} functions.
    
We begin by observing that for each $j=1,\ldots,m$, one can write
\begin{equation*}
  p_j(\cdot)m_{z,j}(\cdot) = e_{1,j}(\cdot)z + e_{2,j}(\cdot)
\end{equation*}
for some bounded real valued measurable functions $e_{1,j}(\cdot)$ and $e_{2,j}(\cdot)$ on $\Omega_j$. Thus, one can find sequences $(e_{1,k,j})_{k=1}^{\infty}$, $(e_{2,k,j})_{k=1}^{\infty}$ of simple functions on $\Omega$ such that for all $i=1,2$ we have
\begin{equation*}
  |e_{i,k,j}| \leq |e_{i,j}|,\quad\forall k=1,2,\ldots 
\end{equation*}
and
\begin{equation*}
  e_{i,k,j}\to e_{i,j} \quad\text{uniformly on }\Omega_j\text{ as }k\to\infty.
\end{equation*}
Now, for every $z\in\bar{S}$ we define
\begin{equation*}
  F_{z,k,j}(\cdot) := |\widetilde{f}_j(\cdot)|^{e_{1,k,j}(\cdot)z + e_{2,k,j}(\cdot)}\chi_{\{\widetilde{f}_j\neq0\}}(\cdot)e^{is_j(\cdot)}w_{0,j}(\cdot)^{z-1}w_{1,j}(\cdot)^{-z}.
\end{equation*}
and
\begin{equation*} 
  U_{\ell,k}(z,\cdot) := \exp\left(\dfrac{z^2-1}{s\ell}\right)M_0^{z-1}M_1^{-z}T(\vec{F}_{z,k})(\cdot)v_0(\cdot)^{1-z}v_1(\cdot)^{z}G_{z}(\cdot),
\end{equation*}
\begin{equation*}
  \Phi_{\ell,k}(z) := \int_{\Omega}|U_{\ell,k}(z, \omega)|^{s}\dd\omega.
\end{equation*}
By combining Lemmas~\ref{lem:unif_comp} and \ref{lem:cont_and_subh} below we deduce that $\Phi_{\ell}$ is continuous on $\bar{S}$ and subharmonic in $S$.
\end{proof}

\begin{lemma}\label{lem:unif_comp}
Fix $\ell\in\N$. It holds that
\begin{equation*}
  \Phi_{\ell,k}(z) \to \Phi_{\ell}(z)
\end{equation*}
as $k\to\infty$ uniformly on compact sets for $z\in\bar{S}$.
\end{lemma}

\begin{proof}
Fix a compact subset $K$ of $\bar{S}$. It is well known that for all $a\in(0,\infty)$, the function
\begin{equation*}
        z\mapsto a^{z} = \exp((\ln a)z)
\end{equation*}
in continuous on $\C$ and thus uniformly continuous on every compact subset of $\C$. It is then easy to see that
\begin{equation*}
  \lim_{k\to\infty}F_{z,k,j}(\omega) \to F_{z,j}(\omega)\quad\text{uniformly for }(\omega,z)\in\Omega_j\times K.
\end{equation*}
Since $w_{0,j}$ is bounded and $|\Omega_j|<\infty$, it follows that
\begin{equation*}
  F_{z,k,j} \to F_{z,j}\quad\text{as }k\to\infty\text{ in }L^{p_{0,j}(\cdot)}(w_{0,j})
\end{equation*}
uniformly for $z\in K$. Moreover, it easy to see that
\begin{equation*}
  \sup_{z \in K}\left(\sup_{k\in\N}\esssup |F_{z,k,j}|, \esssup |F_{z,j}|\right) < \infty.
\end{equation*}
Thus, by the same token as before we have
\begin{equation*}
  \sup_{z\in K}\left(\sup_{k\in\N}\Vert F_{z,k,j}\Vert_{L^{p_{0,j}(\cdot)}(w_{0,j})}, \Vert F_{z,j}\Vert_{L^{p_{0,j}(\cdot)}(w_{0,j})}\right) < \infty.
\end{equation*}
Therefore, using assumption \eqref{eq:interpolation_assumption} for $i=0$ and the multilinearity of $T$, we deduce
\begin{equation*}
  T(\vec{F}_{z,k}) \to T(\vec{F}_{z})\quad\text{as }k\to\infty\text{ in }L^{q_{0}(\cdot)}(v_{0})
\end{equation*}
uniformly for $z\in K$. An easy computation shows that there exists some constant $0<C<\infty$ such that for all $z\in K$ and all $k\in\N$ we have
\begin{equation*}
  |U_{\ell,k}(z,\cdot) - U_{\ell}(z,\cdot)| \leq C|T(\vec{F}_{z,k})(\cdot) - T(\vec{F}_{z})|\quad\text{a.e.~on }\Omega.
\end{equation*}
Therefore, since $s<1$, for each $z\in K$ and $k\in\N$ we can compute
\begin{align*}
  |\Phi_{\ell,k}(z) - \Phi_{\ell}(z)| &\leq \int_{\Omega}|U_{\ell,k}(z,\omega) - U_{\ell}(z,\omega)|^{s} \dd\omega\\
  &\leq C^{s} \int_{\Omega}|T(\vec{F}_{z,k})(\omega) - T(\vec{F}_{z})(\omega)|^{s} \dd\omega\\
  & = C^{s} \int_{\Omega}|T(\vec{F}_{z,k})(\omega) - T(\vec{F}_{z})(\omega)|^{s}v_{0}(\omega)^{s}v_{0}(\omega)^{-s}\dd\omega\\
  &\leq C^{s}C_1 \Vert |T(\vec{F}_{z,k}) - T(\vec{F}_{z})|^{s}v_{0}^{s} \Vert_{L^{q_0(\cdot)/s}} \cdot \Vert v_{0}^{-s} \Vert_{L^{(q_0(\cdot)/s)'}}\\
  &\leq C^{s}C_1C_2 \Vert T(\vec{F}_{z,k}) - T(\vec{F}_{z}) \Vert_{L^{q_0(\cdot)}(v_0)}^{s}.
\end{align*}
In the transition from the third to the fourth line we have used H\"older's inequality, Lemma~\ref{lem:Hoelder}, so that $0<C_1<\infty$ depends only on $q_0(\cdot)$ and $s$. Moreover, we have used the fact that, since $v_0$ is a.e.~bounded from below by a positive constant and $|\Omega|<\infty$, we have
\begin{equation*}
  C_2 := \Vert v_{0}^{-s} \Vert_{L^{(q_0(\cdot)/s)'}} < \infty.
\end{equation*}
This completes the proof.
\end{proof}

\begin{lemma}\label{lem:cont_and_subh}
Fix $\ell,k\in\N$. The function $\Phi_{\ell,k}$ is continuous on $\bar{S}$ and subharmonic in $S$.
\end{lemma}

\begin{proof}
It is clear that one can write
\begin{equation*}
  F_{z,k,j} = \sum_{t = 1}^{N_{j}}a_{t,j}^{z}b_{t,j}e^{ic_{t,j}}\chi_{A_{t,j}},\quad z\in\bar{S}
\end{equation*}
where
\begin{equation*}
  a_{t,j}, b_{t,j} \in (0,\infty),\quad c_{t,j} \in \R,
\end{equation*}
and $A_{t,j}$ is a measurable subset of $\Omega_j$ of finite measure. It follows that for all $z\in\bar{S}$, using the multilinearity of $T$ we obtain
\begin{align*}
  U_{\ell,k}(z,\omega) &= \sum_{\substack{t_j\in\{1,\ldots,N_j\}\\j=1,\ldots,m}}\exp\left(\dfrac{z^2-1}{s\ell}\right)M_0^{z-1}M_1^{-z}v_0(\omega)^{1-z}v_1(\omega)^{z}G_{z}(\omega)\\
  &\qquad\cdot\prod_{j=1}^{m}(a_{t_j,j}^{z}b_{t_j,j}e^{ic_{t_j,j}})
  T(\chi_{A_{t_1,1}},\ldots,\chi_{A_{t_m,m}})(\omega) := R(z,\omega)
\end{align*}
for a.e.~$\omega\in\Omega$. Thus, we have
\begin{equation*}
  \Phi_{\ell,k}(z) = \int_{\Omega}|R(z,\omega)|^{s}\dd\omega,
\end{equation*}
for every $z\in\bar{S}$.

Let us first prove that $\Phi_{\ell,k}$ is continuous on the whole of $\bar{S}$. Fix $z_0\in\bar{S}$. Pick a compact neighborhood $K$ of $z_0$ in $\bar{S}$. It is clear that
\begin{equation*}
  R(z,\cdot) \to R(z_0,\cdot)\text{ as }z\to z_0\text{ pointwise on }\Omega.
\end{equation*}
Moreover, it is easy to see that
\begin{equation*}
  \sup_{z\in K}\esssup|R(z,\cdot)| < \infty.
\end{equation*}
Therefore, since $|\Omega|<\infty$, an application of the dominated convergence theorem  for classical Lebesgue spaces yields
\begin{equation*}
  \lim_{z\to z_0}\Phi_{\ell,k}(z) = \lim_{z\to z_0}\int_{\Omega}|R(z,\omega)|^{s}\dd\omega = \int_{\Omega}|R(z_0,\omega)|^{s}\dd\omega = \Phi_{\ell,k}(z_0).
\end{equation*}
    
Next, it is clear that for a.e.~$\omega\in\Omega$, the function
\begin{equation*}
  z\mapsto R(z,\omega)
\end{equation*}
is analytic on $S$. Therefore, since we have already proved that $\Phi_{\ell,k}$ is continuous on $S$, similarly to the top of page 5031 in \cite{CaoOlivoYabuta2022} we deduce that $\Phi_{\ell,k}$ is subharmonic in $S$, concluding the proof.
\end{proof}

\begin{lemma}
\label{lem:vanish}
Fix $\ell\in\N$. Then, there holds
\begin{equation*}
  \lim_{|z|\to\infty}\Phi_{\ell}(z) = 0.
\end{equation*}
\end{lemma}

\begin{proof}
It suffices to prove that there exists a constant $0<C<\infty$ such that
\begin{equation*}
  0\leq\Phi_{\ell}(z) \leq C\left|\exp\left(\frac{z^2-1}{\ell}\right)\right|,\quad\forall z\in\bar{S},
\end{equation*}
for
\begin{equation*}
  \left|\exp\left(\frac{z^2-1}{\ell}\right)\right|
  = \exp\left(\mathrm{Re}\left(\frac{z^2-1}{\ell}\right)\right)
  = \exp\left(\frac{(\mathrm{Re}(z))^2-1-(\mathrm{Im}(z))^2}{\ell}\right),\quad\forall z\in\C,
\end{equation*}
and $0<\mathrm{Re}(z) \leq 1$, for all $z\in\bar{S}$.
    
We begin by estimating
\begin{equation*}
  \Phi_{\ell}(z) = \left|\exp\left(\frac{z^2-1}{\ell}\right)\right|\int_{\Omega}|M_0^{z-1}M_1^{-z}T(\vec{F}_{z})(\omega)v_0(\omega)^{1-z}v_1(\omega)^{z}G_{z}(\omega)|^{s}\dd\omega.
\end{equation*}
Next, we have
\begin{align*}
  &\int_{\Omega}|M_0^{z-1}M_1^{-z}T(\vec{F}_{z})(\omega)v_0(\omega)^{1-z}v_1(\omega)^{z}G_{z}(\omega)|^{s}\dd\omega\\
  & = \int_{\Omega}M_0^{s\mathrm{Re}(z)-s}M_1^{-s\mathrm{Re}(z)}|T(\vec{F}_{z})(\omega)|^{s}v_0(\omega)^{s-s\mathrm{Re}(z)}v_1(\omega)^{\mathrm{Re}(z)}g(\omega)^{\mathrm{Re}(t(z,\omega))}\chi_{\{g\neq0\}}(\omega)\dd\omega\\
  &\leq C_1\int_{\Omega}|T(\vec{F}_{z})(\omega)|^{s}\dd\omega
\end{align*}
whenever $0\leq\mathrm{Re}(z)\leq1$, for
\begin{equation*}
  \mathrm{Re}(t(z,\omega)) = r'(\omega) - sr'(\omega)\left[\frac{1-\mathrm{Re}(z)}{q_0(\omega)}+\frac{\mathrm{Re}(z)}{q_1(\omega)}\right].
\end{equation*}
As in the proof of Lemma~\ref{lem:unif_comp} we estimate
\begin{equation*}
  \int_{\Omega}|T(\vec{F}_{z})(\omega)|^{s}\dd\omega \leq C_2\Vert T(\vec{F}_{z}) \Vert_{L^{q_0(\cdot)}(v_0)}^{s}.
\end{equation*}
Assumption \eqref{eq:interpolation_assumption} for $i=0$ gives
\begin{align*}
  \Vert T(\vec{F}_{z}) \Vert_{L^{q_0(\cdot)}(v_0)}
  \leq C_3\prod_{j=1}^{m}\Vert F_{z,j}\Vert_{L^{p_{0,j}(\cdot)}(w_{0,j})}
\end{align*}
Now, we have
\begin{align*}
  |F_{z,j}(\omega)| = |\widetilde{f}_{j}(\omega)|^{\mathrm{Re}(p_j(\omega)m_{z,j}(\omega))}\chi_{\{\widetilde{f}_j\neq0\}}(\omega)w_{0,j}(\omega)^{\mathrm{Re}(z)-1}w_{1,j}(\omega)^{-\mathrm{Re}(z)}
\end{align*}
Thus, since
\begin{equation*}
  \mathrm{Re}((p_jm_{z,j})(\omega)) = 
  \frac{p_j(\omega)(1-\mathrm{Re}(z))}{p_0(\omega)}+\frac{p_{j}(\omega)\mathrm{Re}(z)}{p_1(\omega)},
\end{equation*}
we deduce
\begin{equation*}
  \sup_{z\in\bar{S}}\esssup|F_{z,j}| < \infty.
\end{equation*}
Thus, as in the proof of Lemma~\ref{lem:unif_comp} we deduce
\begin{equation*}
  \sup_{z\in\bar{S}}\Vert F_{z,j}\Vert_{L^{p_{0,j}(\cdot)}(w_{0,j})} < \infty,
\end{equation*}
yielding the desired result.
\end{proof}

\subsubsection{Estimating the boundary}

\begin{lemma}
\label{lem:boundary}
Let $\ell\in\N$ and $y\in\R$ be arbitrary. Then, we have
\begin{equation}\label{eq:left_line}
  0\leq\Phi_{\ell}(iy) \leq 1
\end{equation}
and
\begin{equation}\label{eq:right_line}
  0\leq\Phi_{\ell}(1+iy) \leq 1.
\end{equation}
\end{lemma}

\begin{proof}
Let us prove \eqref{eq:left_line}.  With $z := iy$ we estimate
\begin{align*}
  &\int_{\Omega}|M_0^{z-1}M_1^{-z}T(\vec{F}_{z})(\omega)v_0(\omega)^{1-z}v_1(\omega)^{z}G_{z}(\omega)|^{s}\dd\omega\\
  & = \int_{\Omega}M_0^{s\mathrm{Re}(z)-s}M_1^{-s\mathrm{Re}(z)}|T(\vec{F}_{z})(\omega)|^{s}v_0(\omega)^{s-s\mathrm{Re}(z)}v_1(\omega)^{\mathrm{Re}(z)}g(\omega)^{\mathrm{Re}(t(z,\omega))}\chi_{\{g\neq0\}}(\omega)\dd\omega\\
  & =\int_{\Omega}|T(\vec{F}_{z})(\omega)|^{s}v_0(\omega)^{s}g(\omega)^{r'(\omega)/r_0'(\omega)}\dd\omega,
\end{align*}
where $r_0(\cdot) := \frac{q_0(\cdot)}{s}$. Using H\"{o}lder's inequality, Lemma~\ref{lem:Hoelder}, as well as Lemma \ref{lem:homog}, we get
\begin{align*}
  \int_{\Omega}|T(\vec{F}_{z})(\omega)|^{s}v_0(\omega)^{s}g(\omega)^{r'(\omega)/r_0'(\omega)}\dd\omega
  & \lesssim \Vert |T(\vec{F}_{z})|^{s}v_0^{s}\Vert_{L^{r_0(\cdot)}}
  \Vert g(\cdot)^{r'(\cdot)/r_0'(\cdot)}\Vert_{L^{r_0'(\cdot)}}\\
  & = \Vert T(\vec{F}_{z})\Vert_{L^{q_0(\cdot)}(v_0)}^{s}
  \Vert g(\cdot)^{r'(\cdot)/r_0'(\cdot)}\Vert_{L^{r_0'(\cdot)}}.
\end{align*}
We compute
\begin{align*}
  \rho_{r_0'(\cdot)}(g(\cdot)^{r'(\cdot)/r_0'(\cdot)}) = 
  \int_{\Omega}g(\omega)^{r'(\omega)}\dd\omega \leq 1,
\end{align*}
since $\Vert g\Vert_{L^{r'(\cdot)}}\leq 1$. By Lemma~\ref{lem: min/max} this implies that
\begin{align*}
  \Vert g(\cdot)^{r'(\cdot)/r_0'(\cdot)}\Vert_{L^{r_0'(\cdot)}} \leq 1.
\end{align*}
Finally, using assumption \eqref{eq:interpolation_assumption} for $i=0$ we compute
\begin{align*}
  \Vert T(\vec{F}_{z})\Vert_{L^{q_0(\cdot)}(v_0)}
  \leq M_0\prod_{j=1}^{m}\Vert F_{z,j}\Vert_{L^{p_{0,j}(\cdot)}(w_{0,j})}
  = M_0\prod_{j=1}^{m}\Vert F_{z,j}w_{0,j}\Vert_{L^{p_{0,j}(\cdot)}}.
\end{align*}
Similarly, we compute
\begin{align*}
  \rho_{p_{0,j}(\cdot)}(F_{z,j}w_{0,j}) = \rho_{p_{j}(\cdot)}(\widetilde{f_j})
  = \rho_{p_{j}(\cdot)}(f_jw_{j}) \leq 1
\end{align*}
by Lemma~\ref{lem: min/max}, since $\Vert f_jw_{j}\Vert_{L^{p_{j}(\cdot)}}\leq1$, and therefore
\begin{align*}
  \Vert F_{z,j}\Vert_{L^{p_{0,j}(\cdot)}(w_{0,j})} \leq 1.
\end{align*}
We conclude that
\begin{equation*}
  \Vert T(\vec{F}_{z})\Vert_{L^{q_0(\cdot)}(v_0)} \leq M_0.
\end{equation*}
Thus
\begin{align*}
  0\leq\Phi_{\ell}(z) \leq 
  \left|\exp\left(\frac{z^2-1}{\ell}\right)\right| = \exp(-(y^2+1)/\ell) \leq 1.
\end{align*}

The proof of \eqref{eq:right_line} is similar.
\end{proof}

\subsubsection{Applying the maximum principle} Fix $\ell\in\N$. In view of the continuity of $\Phi_{\ell}$ on the whole of $\bar{S}$ and the facts that
\begin{align*}
  &\Phi_{\ell}(z) \geq 0,\quad\forall z\in\bar{S},\\
  &\lim_{|z|\to\infty}\Phi_{\ell}(z) = 0,
\end{align*}
which we proved above in Lemmata~\ref{lem:phi_ell_con_subh},~\ref{lem:vanish} and~\ref{lem:boundary}, we deduce that $\Phi_{\ell}$ attains a maximum value on $\bar{S}$.

Since $\Phi_{\ell}$ is subharmonic in $S$ by Lemma~\ref{lem:phi_ell_con_subh}, the maximum principle yields that it cannot attain a global maximum on $S$, unless it is constant on $S$ and hence by continuity on the whole of $\bar{S}$. In view of the facts that
\begin{gather*}
  \Phi_{\ell}(iy)\leq1,\quad\forall y\in\R,\\
  \Phi_{\ell}(1+iy)\leq1,\quad\forall y\in\R,
\end{gather*}
we conclude finally that
\begin{equation*}
  \Phi_{\ell}(z) \leq 1,\quad\forall z\in\bar{S}.
\end{equation*}

\subsubsection{Concluding the proof} In particular, we obtain
\begin{equation*}
  |\Phi_{\ell}(\theta)| \leq 1,\quad\forall \ell=1,2,\ldots.
\end{equation*}
Fatou's lemma for classical Lebesgue spaces yields
\begin{equation*}
  \int_{\Omega}\liminf_{\ell\to\infty}|U_{\ell}(\theta,\omega)|^{s}\dd\omega \leq 
  \liminf_{\ell\to\infty}\int_{\Omega}|U_{\ell}(\theta,\omega)|^{s}\dd\omega 
  = \liminf_{\ell\to\infty}|\Phi_{\ell}(\theta)| \leq 1.
\end{equation*}
Observe that
\begin{equation*}
  U_{\ell}(\theta,\omega) = \exp\left(\frac{\theta^2-1}{s\ell}\right)
  M_0^{\theta-1}M_1^{-\theta}T(f_1w_{1}^{-1},\ldots, f_mw_m^{-1})(\omega)v(\omega)g(\omega)^{1/s},
\end{equation*}
therefore
\begin{equation*}
  \lim_{\ell\to\infty}U_{\ell}(\theta,\omega) = M_0^{\theta-1}M_1^{-\theta}T(f_1,\ldots, f_m)(\omega)v(\omega)g(\omega)^{1/s}.
\end{equation*}
It follows that
\begin{align*}
  \int_{\Omega}|T(f_1,\ldots, f_m)(\omega)v(\omega)|^{s}g(\omega)\dd\omega \leq M_0^{s(1-\theta)}M_1^{s\theta},
\end{align*}
completing the proof.

\subsection{Interpolation of boundedness II}
One of the last remaining steps is to provide the interpolation for multilinear bounded operators with values in a particular type of mixed-variable weighted Lebesgue spaces. Concretely, for $q(\cdot)\in\mathscr{P}_0(\Omega)$, a weight $\nu$ on $\Omega$, a constant $\widetilde{q}\in(0,\infty)$ and a set $\widetilde{\Omega}$, we define the mixed-variable weighted Lebesgue space $L^{q(\cdot)}(v)(L^{\widetilde{q}}(\widetilde{\Omega}))$ as the set of all measurable functions $g:\Omega\times\widetilde{\Omega}\to\C$ with
\begin{equation*}
    \Vert \Vert g(x,y)\Vert_{L^{\widetilde{q}}_{y}}\Vert_{L^{q(\cdot)}_{x}(\nu)} < \infty.
\end{equation*}
The following is the counterpart of a special case of \cite[Theorem 3.5]{CaoOlivoYabuta2022} in the variable exponent setting.

\begin{theorem}
\label{thm:interpolation_boundedness_mixed}
Let $q_0(\cdot),q_1(\cdot)\in(\mathscr{P}_0\cap\mathrm{LH})(\Omega)$, $p_{0,j}(\cdot),p_{1,j}(\cdot)\in(\mathscr{P}\cap\mathrm{LH})(\Omega_j)$, $j=1,\ldots,m$. Moreover, let
\begin{equation*}
  \widetilde{q} \in (0, \min\{(q_0)_{-}, (q_1)_{-}\}).
\end{equation*}
Let $w_{0,j},w_{1,j}$, $j=1,\ldots,m$ be weights on $\Omega_j$ and let also $v_{0},v_{1}$ be weights on $\Omega$. Let $T$ be a $m$-linear operator such that
\begin{equation*}
  T : L^{p_{0,1}(\cdot)}(w_{0,1})\times\cdots\times L^{p_{0,m}(\cdot)}(w_{0,m}) \to L^{q_0(\cdot)}(v_0)(L^{\widetilde{q}}(\widetilde{\Omega}))
\end{equation*}
and
\begin{equation*}
  T : L^{p_{1,1}(\cdot)}(w_{1,1})\times\cdots\times L^{p_{1,m}(\cdot)}(w_{1,m}) \to L^{q_1(\cdot)}(v_1)(L^{\widetilde{q}}(\widetilde{\Omega}))
\end{equation*}
boundedly. For each $i=0,1$, let $M_i\in(0,\infty)$ with
\begin{equation*}
  \Vert \Vert T(\vec{f})(x,y)\Vert_{L^{\widetilde{q}}_{y}}\Vert_{L^{q_i(\cdot)}_{x}(v_i)} \leq M_i\prod_{j=1}^{m}\Vert f_j\Vert_{L^{p_{i,j}(\cdot)}(w_{i,j})}.
\end{equation*}
for all $\vec{f}=(f_1,\ldots,f_m)\in L^{p_{i,1}(\cdot)}(w_{i,1})\times\cdots\times L^{p_{i,m}(\cdot)}(w_{i,m})$. Let $\theta\in(0,1)$ and define
\begin{align*}
  &\frac{1}{p_{j}(\cdot)} := \frac{1-\theta}{p_{0,j}(\cdot)} + \frac{\theta}{p_{1,j}(\cdot)},\quad j=1,\ldots,m,\\
  &\frac{1}{q(\cdot)} := \frac{1-\theta}{q_{0}(\cdot)} + \frac{\theta}{q_{1}(\cdot)}
\end{align*}
as well as
\begin{align*}
  &w_{j} := w_{0,j}^{1-\theta}w_{1,j}^{\theta},\quad j=1,\ldots,m,\\
  &v := v_{0}^{1-\theta}v_{1}^{\theta}.
\end{align*}
Then, we have
\begin{equation*}
  T : L^{p_{1}(\cdot)}(w_{1})\times\cdots\times L^{p_{m}(\cdot)}(w_{m}) \to L^{q(\cdot)}(v)(L^{\widetilde{q}}(\widetilde{\Omega}))
\end{equation*}
boundedly and it holds
\begin{equation*}
  \Vert \Vert T(\vec{f})(x,y)\Vert_{L^{\widetilde{q}}_{y}}\Vert_{L^{q(\cdot)}_{x}(v)} \leq M_0^{1-\theta}M_{1}^{\theta}\prod_{j=1}^{m}\Vert f_j\Vert_{L^{p_{j}(\cdot)}(w_{j})},
\end{equation*}
for all $\vec{f}=(f_1,\ldots,f_m)\in L^{p_{1}(\cdot)}(w_{1})\times\cdots\times L^{p_{m}(\cdot)}(w_{m})$.
\end{theorem}

\begin{proof}
We adapt the strategy of the proof of \cite[Theorem 3.5]{CaoOlivoYabuta2022}.

We begin by performing the same reductions as in the proof of Theorem~\ref{thm:interpolation_boundedness}. First of all, observe that one has
\begin{equation*}
  \Vert \Vert g(x,y)\Vert_{L^{\widetilde{q}}_{y}}\Vert_{L^{q(\cdot)}_{x}(v)}
  = \Vert \Vert g(x,y)v(x)\Vert_{L^{\widetilde{q}}_{y}}\Vert_{L^{q(\cdot)}_{x}}.
\end{equation*}
Therefore, working as in the proof of Lemma~\ref{lem:first_red}, with the only difference of using Proposition~\ref{prop:almost_everywhere_equality_mixed} in the place of Proposition~\ref{prop:norm_convergence_to_subs_pointwise_convergence}, we deduce that it suffices to show that 
\begin{equation*}
  \Vert \Vert T(\vec{f})(x,y)\Vert_{L^{\widetilde{q}}_{y}}\Vert_{L^{q(\cdot)}_{x}(v)} \leq M_0^{1-\theta}M_{1}^{\theta}\prod_{j=1}^{m}\Vert f_j\Vert_{L^{p_{j}(\cdot)}(w_{j})}
\end{equation*}
for all $\vec{f}\in\vec{A}$. Next, working as in the proofs of Lemma~\ref{lem:second_red}, Proposition~\ref{prop:third_red} and Lemma~\ref{lem:fourth_red}, we see that without loss of generality, we may assume that
\begin{itemize}
  \item $|\Omega| < \infty$ and $|\widetilde{\Omega}|<\infty$.

  \item All weights $w_{i,j}$, $j=1,\ldots,m$, $i=0,1$ and $u_0,v_1$ are simple functions.

  \item There holds $\Vert f_j\Vert_{L^{p_j(\cdot)}(w_j)}=1$, for all $j=1,\ldots,m$.
\end{itemize}

We now proceed with the main part of the proof, which is similar to the proof of Proposition~\ref{prop:central_estimate}. Observe that
\begin{equation*}
  0 < \widetilde{q} < q_{-}.
\end{equation*}
Pick
\begin{equation*}
  s \in (0, \widetilde{q}).
\end{equation*}
Define $r(\cdot) := q(\cdot)/s$ and $\widetilde{r} := \widetilde{q}/s$. By using Lemma~\ref{lem:homog} twice we have
\begin{align*}
  &\Vert \Vert T(\vec{f})(x,y)\Vert_{L^{\widetilde{q}}_{y}}\Vert_{L^{q(\cdot)}_{x}(v)}
  = \Vert \Vert T(\vec{f})(x,y)v(x)\Vert_{L^{\widetilde{q}}_{y}}\Vert_{L^{q(\cdot)}_{x}}
  = \Vert \Vert |T(\vec{f})(x,y)v(x)|^{s}\Vert_{L^{\widetilde{r}}_{y}}^{1/s}\Vert_{L^{q(\cdot)}_{x}}\\
  & = \Vert \Vert |T(\vec{f})(x,y)v(x)|^{s}\Vert_{L^{\widetilde{r}}_{y}}\Vert_{L^{r(\cdot)}_{x}}^{1/s}.
\end{align*}
By Lemma~\ref{lem:dual_mixed_norm} we have
\begin{equation*}
  \Vert \Vert |T(\vec{f})(x,y)v(x)|^{s}\Vert_{L^{\widetilde{r}}_{y}}\Vert_{L^{r(\cdot)}_{x}} \sim 
  \sup_{g}\int_{\Omega\times\widetilde{\Omega}}|T(\vec{f})(x,y)v(x)|^{s}g(x,y)\dd x\dd y,
\end{equation*}
where the supremum is taken over all nonnegative measurable functions $g$ on $\Omega\times\widetilde{\Omega}$ with
\begin{equation*}
  \Vert \Vert g(x,y)\Vert_{L_{y}^{\widetilde{r}'}}\Vert_{L_{x}^{r'(\cdot)}} \leq 1
\end{equation*}
and the implicit constants depend only on $r(\cdot)$. Thus, we only have to show that
\begin{equation*}
  \int_{\Omega\times\widetilde{\Omega}}|T(\vec{f})(x,y)v(x)|^{s}g(x,y)\dd x\dd y \leq M_0^{s(1-\theta)}M_{1}^{s\theta}
\end{equation*}
for every such function $g$. Finally, as in the proof of Proposition~\ref{prop:central_estimate}, we can assume $g$ is a simple function.
    
Define
\begin{equation*}
  S := \{z\in\C:~0< \mathrm{Re}(z)<1\},\quad \bar{S} := \{z\in\C:~0\leq \mathrm{Re}(z)\leq1\}.
\end{equation*}
For each $j=1,\ldots,m$ write
\begin{equation*}
  \widetilde{f}_j(\cdot):= f_j(\cdot)w_j(\cdot),\quad \widetilde{f}_j(\cdot) = |\widetilde{f_j}(\cdot)|e^{is_j(\cdot)},\quad j=1,\ldots,m
\end{equation*}
for some real-valued simple function $s_j$ on $\Omega_j$. For every $z\in\bar{S}$ define the \emph{complex valued} functions
\begin{equation*}
  m_{z,j}(\cdot) := \frac{1-z}{p_{0,j}(\cdot)} + \frac{z}{p_{1,j}(\cdot)},\quad j=1,\ldots,m,\quad n_{z}(\cdot) := \frac{1-z}{q_{0}(\cdot)} + \frac{z}{q_{1}(\cdot)}.
\end{equation*}
Then, for every $\ell=1,2,\ldots$ and for every $z\in\bar{S}$ define
\begin{equation*}
  \Phi_{\ell}(z) := \int_{\Omega\times\widetilde{\Omega}}|U_{\ell}(z, x, y)|^{s}\dd x\dd y,
\end{equation*}
where
\begin{equation*}
  U_{\ell}(z, x, y) := \exp\left(\dfrac{z^2-1}{s\ell}\right)M_0^{z-1}M_1^{-z}T(\vec{F}_{z})(x, y)v_0(x)^{1-z}v_1(x)^{z}G_{z}(x, y),\quad (x,y)\in\Omega\times\widetilde{\Omega},
\end{equation*}
\begin{equation*}
  G_{z}(x, y) := g(x, y)^{1/s}\Vert g(x, u) \Vert_{L^{\widetilde{r}'}_{u}}^{t(x,z)/s}\chi_{_{\{\Vert g(\cdot, u) \Vert_{L^{\widetilde{r}'}_{u}}\neq0\}}}(x),\quad t(x, z) := r'(x) - sr'(x)n_{z}(x) - 1,
\end{equation*}
\begin{equation*}
  F_{z,j}(\cdot) := |\widetilde{f}_j(\cdot)|^{p_j(\cdot)m_{z,j}(\cdot)}\chi_{\{\widetilde{f}_j\neq0\}}(\cdot)e^{is_j(\cdot)}w_{0,j}(\cdot)^{z-1}w_{1,j}(\cdot)^{-z}.
\end{equation*}
As in the proof of Proposition~\ref{prop:central_estimate}, $\Phi_{\ell}$ is continuous on $\bar{S}$, subharmonic in $S$ and satisfies
\begin{equation*}
  \lim_{|z|\to\infty}\Phi_{\ell}(z) = 0.
\end{equation*}
Let now $b\in\R$ be arbitrary. Then, with $z := ib$ we estimate
\begin{align*}
  &\int_{\Omega\times\widetilde{\Omega}}|M_0^{z-1}M_1^{-z}T(\vec{F}_{z})(x,y)v_0(x)^{1-z}v_1(x)^{z}G_{z}(x,y)|^{s}\dd x\dd y\\
  & = \int_{\Omega\times\widetilde{\Omega}}M_0^{s\mathrm{Re}(z)-s}M_1^{-s\mathrm{Re}(z)}|T(\vec{F}_{z})(x,y)|^{s}v_0(x)^{s-s\mathrm{Re}(z)}v_1(x)^{\mathrm{Re}(z)}g(x,y)\Vert g(x, u) \Vert_{L^{\widetilde{r}'}_{u}}^{\mathrm{Re}(t(z,x))}\\
  &\qquad\cdot\chi_{_{\{\Vert g(\cdot, u) \Vert_{L^{\widetilde{r}'}_{u}}\neq0\}}}(x)\dd x\dd y \\
  & =\int_{\Omega\times\widetilde{\Omega}}|T(\vec{F}_{z})(x, y)|^{s}v_0(x)^{s}g(x, y)\Vert g(x, u) \Vert_{L^{\widetilde{r}'}_{u}}^{\frac{r'(x)}{r'_{0}(x)}-1}\chi_{_{\{\Vert g(\cdot, u) \Vert_{L^{\widetilde{r}'}_{u}}\neq0\}}}(x)\dd x\dd y,
\end{align*}
where $r_0(\cdot) := \frac{q_0(\cdot)}{s}$. Again by Lemmata~\ref{lem:homog} and~\ref{lem:dual_mixed_norm} we estimate 
\begin{align*}
  &\int_{\Omega\times\widetilde{\Omega}}|T(\vec{F}_{z})(x, y)|^{s}v_0(x)^{s}g(x, y)\Vert g(x, u) \Vert_{L^{\widetilde{r}'}_{u}}^{\frac{r'(x)}{r'_{0}(x)}-1}\chi_{_{\{\Vert g(\cdot, u) \Vert_{L^{\widetilde{r}'}_{u}}\neq0\}}}(x)\dd x\dd y\\
  & \lesssim \Vert \Vert|T(\vec{F}_{z})(x,y)|^{s}v_0(x)^{s}\Vert_{L^{\widetilde{r}}_{y}}\Vert_{L^{r_0(\cdot)}_{x}}\cdot
  \Vert \Vert g(x, y)\Vert g(x, u) \Vert_{L^{\widetilde{r}'}_{u}}^{\frac{r'(x)}{r'_{0}(x)}-1}\chi_{_{\{\Vert g(\cdot, u) \Vert_{L^{\widetilde{r}'}_{u}}\neq0\}}}(x)\Vert_{L^{\widetilde{r}'}_{y}}\Vert_{L^{r_0'(\cdot)}_{x}}\\
  & = \Vert \Vert T(\vec{F}_{z})(x,y)\Vert_{L^{\widetilde{q}}_{y}}\Vert_{L^{q_0(\cdot)}_{x}(v_0)}^{s}
  \Vert \Vert g(x, y) \Vert_{L^{\widetilde{r}'}_{y}}^{\frac{r'(x)}{r'_{0}(x)}}\Vert_{L^{r_0'(\cdot)}_{x}}.
\end{align*}
We compute
\begin{align*}
  \rho_{r_0'(\cdot), x}\left(\Vert g(x, y) \Vert_{L^{\widetilde{r}'}_{y}}^{\frac{r'(x)}{r'_{0}(x)}}\right) = 
  \int_{\Omega}\Vert g(x, y) \Vert_{L^{\widetilde{r}'}_{y}}^{r'(x)}\dd x \leq 1
\end{align*}
by Lemma~\ref{lem: min/max}, since $\Vert \Vert g(x,y)\Vert_{L_{y}^{\widetilde{r}'}}\Vert_{L_{x}^{r'(\cdot)}} \leq 1$. The rest of the proof proceeds exactly as that of Proposition~\ref{prop:central_estimate}.
\end{proof}

The following lemma collects special cases of some results from \cite{Ho_2021}. Since in the relevant places in \cite{Ho_2021} only $\R$ was considered as underlying set, we include the proof for the sake of completeness.

\begin{lemma}\label{lem:dual_mixed_norm}
Let $p(\cdot)\in\mathscr{P}(\Omega_1)$ and let $q\in(1,\infty)$.
\begin{enumerate}[label=\textup{(\arabic*)}]
  \item For all measurable functions $f,g:\Omega_1\times\Omega_2\to\C$ we have
\begin{equation*}
  \int_{\Omega_1\times\Omega_2}|f(x,y)g(x,y)|\dd x\dd y \lesssim \Vert \Vert f(x,y)\Vert_{L^{q}_{y}}\Vert_{L^{p(\cdot)}_{x}}\Vert \Vert g(x,y)\Vert_{L^{q'}_{y}}\Vert_{L^{p'(\cdot)}_{x}},
\end{equation*}
where the implicit constant depends only on $p(\cdot)$.\label{itm: (first itm)}

\item For all measurable functions $f:\Omega_1\times\Omega_2\to\C$ we have
\begin{equation*}
  \Vert \Vert f(x,y)\Vert_{L^{q}_{y}}\Vert_{L^{p(\cdot)}_{x}} \sim
  \sup_{g}\int_{\Omega_1\times\Omega_2}|f(x,y)g(x,y)|\dd x \dd y,
\end{equation*}
where the supremum ranges over all measurable functions $g:\Omega_1\times\Omega_2\to\C$ with
\begin{equation*}
  \Vert \Vert g(x,y)\Vert_{L^{q'}_{y}}\Vert_{L^{p'(\cdot)}_{x}}\leq1,   
\end{equation*}
and the implicit constants depend only on $p(\cdot)$.
\end{enumerate}
\end{lemma}

\begin{proof}
\begin{enumerate}
  \item[(1)] Using first usual H\"older's inequality and then H\"older's inequality, Lemma~\ref{lem:Hoelder}, we compute
\begin{align*}
  &\int_{\Omega_1\times\Omega_2}|f(x,y)g(x,y)|\dd x\dd y 
  =\int_{\Omega_1}\left(\int_{\Omega_2}|f(x,y)g(x,y)|\dd y \right)\dd x \\
  &\leq\int_{\Omega_1}\Vert f(x,y)\Vert_{L^{q}_{y}}\Vert g(x,y)\Vert_{L^{q'}_{y}}\dd x \lesssim\Vert \Vert f(x,y)\Vert_{L^{q}_{y}}\Vert_{L^{p(\cdot)}}
  \Vert g(x,y)\Vert_{L^{q'}_{y}}\Vert_{L^{p'(\cdot)}}.
\end{align*}

\item[(2)] Inequality $\gtrsim$ follows immediately from part \ref{itm: (first itm)}. Therefore, we only have to show the $\lesssim$. Clearly, we can assume that the right-hand side is finite. We denote it by $C$.

For each $i=0,1$ write
\begin{equation*}
  \Omega_i = \bigcup_{N=1}^{\infty}\Omega_i^{N} 
\end{equation*}
where $(\Omega_i^{N})_{N=1}^{\infty}$ is an increasing sequence of measurable subsets of $\Omega_i$ of finite measure. For every $N=1,2,\ldots$ define
\begin{equation*}
  f_{N} := \chi_{\Omega_1^{N}\times\Omega_2^{N}}f\chi_{\{|f|\leq N\}}.
\end{equation*}

Fix $N\in\N$. It is clear that
\begin{equation}\label{eq:finiteness}
  \Vert \Vert f_{N}(x,y)\Vert_{L^{q}_{y}}\Vert_{L^{p(\cdot)}_{x}} < \infty.    
\end{equation}
We show that
\begin{equation*}
  \Vert \Vert f_{N}(x,y)\Vert_{L^{q}_{y}}\Vert_{L^{p(\cdot)}_{x}} \leq C.   
\end{equation*}       
Let $\eps>0$ be arbitrary. We already know \cite[Theorem 2.34]{CF2013} that there exist a constant $M$ depending only on $p(\cdot)$ and not on $\eps$ as well as a measurable function $h:\Omega_1\to[0,\infty)$ with $\Vert h\Vert_{L^{p'(\cdot)}}\leq 1$ such that
\begin{align*}
  \Vert \Vert f_{N}(x,y)\Vert_{L^{q}_{y}}\Vert_{L^{p(\cdot)}_{x}} \leq (1+\eps)M\int_{\Omega_1}\Vert f_{N}(x,y)\Vert_{L^{q}_{y}}h(x)\dd x ,
\end{align*}
Now, notice that \eqref{eq:finiteness} implies 
\begin{equation*}
  \Vert f_{N}(x,y)\Vert_{L^{q}_{y}} < \infty\quad\text{for a.e. }x\in\Omega_1.
\end{equation*}
Consider the function $k:\Omega_1\times\Omega_2\to[0,\infty)$ with
\begin{equation*}
  k(x,y) := \frac{|f_{N}(x,y)|^{q-1}}{\Vert f_{N}(x,y)\Vert_{L^{q}_{y}}^{q-1}}\chi_{\{\Vert f_{N}(\cdot,z)\Vert_{L^{q}_{z}}\neq 0\}}(x),\quad (x,y)\in\Omega_1\times\Omega_2.
\end{equation*}
It is clear that this function is measurable. Then consider the measurable function $g:\Omega_1\times\Omega_2\to[0,\infty)$ with $g(x,y) := k(x,y)h(x)$, $(x,y)\in\Omega_1\times\Omega_2$. It is easy to compute
\begin{equation*}
  \Vert k(x,y)\Vert_{L^{q'}_{y}} \leq 1,\quad\text{for a.e. }x\in\Omega_1,
\end{equation*}
and therefore
\begin{equation*}
  \Vert \Vert g(x,y)\Vert_{L^{q'}_{y}}\Vert_{L^{p'(\cdot)}_{x}} \leq \Vert h\Vert_{L^{p'(\cdot)}} \leq 1.
\end{equation*}
Moreover, it is straightforward to compute
\begin{equation*}
  \int_{\Omega_2}|f_{N}(x,y)|k(x,y)\dd y  = \Vert f_{N}(x,y)\Vert_{L^{q}_{y}}
\end{equation*}
for a.e.~$x\in\Omega_1$, thus
\begin{align*}
  &\Vert \Vert f_{N}(x,y)\Vert_{L^{q}_{y}}\Vert_{L^{p(\cdot)}_{x}} \leq
  (1+\eps)M\int_{\Omega_1\times\Omega_2}|f_{N}(x,y)g(x,y)|\dd x\dd y \\
  &\leq(1+\eps)M\int_{\Omega_1\times\Omega_2}|f(x,y)g(x,y)|\dd x\dd y 
  \leq(1+\eps)MC.
\end{align*}
Letting $\eps\to0$ we deduce
\begin{equation*}
  \Vert \Vert f_{N}(x,y)\Vert_{L^{q}_{y}}\Vert_{L^{p(\cdot)}_{x}} \leq MC.
\end{equation*}
Finally, letting $N\to\infty$ and using the monotone convergence theorem for classical Lebesgue spaces as well as by Proposition \ref{prop:monotone_convergence} we deduce the desired inequality.
\end{enumerate}
\end{proof}

\subsection{Interpolation of compactness}
The main goal of this last subsection is to finalize the proof of the main result of Section \ref{sec: interpolation}, namely, Theorem \ref{thm:interpolation_compactness}. For this purpose, we first collect and prove three lemmata involving the multilinear $\mathcal{A}_{(\vec{p}(\cdot),q(\cdot)),(\vec{r},s)}$ weights.
In this subsection we always take $\Omega_1 = \ldots = \Omega_m = \Omega = \R^n$ for a fixed $n$, unless explicitly otherwise specified. 

\begin{lemma}
\label{lem:interpolation_weights}
Let $\gamma\in[0,\infty)$ and let $(\vec{p}_i(\cdot), q_i(\cdot), \vec{r}, s)$, $i=0,1$ be proper $m$-admissible quadruples with
\begin{equation*}
  \frac{1}{p_0(\cdot)} - \frac{1}{q_0(\cdot)} = \frac{1}{p_1(\cdot)} - \frac{1}{q_1(\cdot)} = \gamma.
\end{equation*}
Let
\begin{gather*}
  \vec{w}_0 \in \mathcal{A}_{(p_0(\cdot),q_0(\cdot)),(\vec{r}, s)},\\
  \vec{w}_1 \in \mathcal{A}_{(p_1(\cdot),q_1(\cdot)),(\vec{r}, s)}.
\end{gather*}
Let $\theta\in(0,1)$ and define $p_{j}(\cdot) \in \mathscr{P}_{0}$, $q(\cdot)\in\mathscr{P}_{0}$ by
\begin{align*}
  &\frac{1}{p_{j}(\cdot)} := \frac{1-\theta}{p_{0,j}(\cdot)} + \frac{\theta}{p_{1,j}(\cdot)}, \quad j=1,\ldots,m,\\
  &\frac{1}{q(\cdot)} := \frac{1-\theta}{q_{0}(\cdot)} + \frac{\theta}{q_{1}(\cdot)}
\end{align*}
as well as the vector of weights $\vec{w}$ by
\begin{equation*}
  w_{j} := w_{0,j}^{1-\theta}w_{1,j}^{\theta},\quad j=1,\ldots,m.
\end{equation*}
Then
\begin{enumerate}[label=\textup{(\arabic*)}]
  \item We have that $(\vec{p}(\cdot),q(\cdot),\vec{r},s)$ is a proper $m$-admissible quadruple with
\begin{equation*}
  \frac{1}{p(\cdot)} - \frac{1}{q(\cdot)} = \gamma.
\end{equation*}

\item It holds $\vec{w}\in\mathcal{A}_{(\vec{p}(\cdot),q(\cdot)),(\vec{r},s)}$ with
\begin{equation*}
  [\vec{w}]_{\mathcal{A}_{(\vec{p}(\cdot),q(\cdot)),(\vec{r},s)}}
  \lesssim [\vec{w}_0]_{\mathcal{A}_{(\vec{p}_0(\cdot),q_0(\cdot)),(\vec{r},s)}}^{1-\theta}[\vec{w}_1]_{\mathcal{A}_{(\vec{p}_1(\cdot),q_1(\cdot)),(\vec{r},s)}}^{\theta},
\end{equation*}
where the implicit constant depends on $\vec{p}_0(\cdot), q_0(\cdot), \vec{p}_1(\cdot), q_1(\cdot), \vec{r}, s$ and $\theta$.
\end{enumerate}
\end{lemma}

\begin{proof}
\begin{enumerate}
  \item Fix $j\in\{1,\ldots,m\}$. It is clear that $0 < (p_{j})_{-} \leq (p_{j})_{+} <\infty$. Moreover, making repeatedly use of \cite[Proposition 2.3]{CF2013} together with the fact that $\mathrm{LH}$ is closed under linear combinations, we deduce that $p_{j}(\cdot) \in \mathscr{P}_{0}\cap\mathrm{LH}$. Let us show that $(p_{j})_{-} > r_j$. Pick $t\in (r_j, \min\{(p_{0,j})_{-},(p_{1,j})_{-}\})$. Then, we have that
\begin{equation*}
  \frac{1}{p_j(\cdot)} = \frac{1-\theta}{p_{0,j}(\cdot)} + \frac{\theta}{p_{1,j}(\cdot)} < \frac{1-\theta}{t} + \frac{\theta}{t} = \frac{1}{t},
\end{equation*}
therefore $(p_{j})_{-} \geq t > r_j$.

Similarly, $q(\cdot)\in\mathscr{P}_{0}\cap\mathrm{LH}$ and $q_{+} < s$. Finally, we compute
\begin{equation*}
  \frac{1}{p(\cdot)} - \frac{1}{q(\cdot)} = (1-\theta)\left(\frac{1}{p_0(\cdot)} - \frac{1}{q_0(\cdot)}\right) + \theta\left(\frac{1}{p_1(\cdot)} - \frac{1}{q_1(\cdot)}\right) = (1-\theta)\gamma + \theta \gamma = \gamma.
\end{equation*}
This concludes the proof.

\item Let $Q\subset\R^n$ be any cube. Observe that
\begin{equation*}
  \frac{1}{\frac{1}{(1-\theta)\left(\frac{1}{q_0(\cdot)} - \frac{1}{s}\right)}} + 
  \frac{1}{\frac{1}{\theta\left(\frac{1}{q_1(\cdot)} - \frac{1}{s}\right)}}
  = \frac{1}{\frac{1}{\frac{1}{q(\cdot)} - \frac{1}{s}}}
\end{equation*}
and
\begin{equation*}
  \frac{1}{\frac{1}{(1-\theta)\left(\frac{1}{r_j} - \frac{1}{p_{0,j}(\cdot)}\right)}} + 
  \frac{1}{\frac{1}{\theta\left(\frac{1}{r_j} - \frac{1}{p_{1,j}(\cdot)}\right)}}
  = \frac{1}{\frac{1}{\frac{1}{r_j} - \frac{1}{p_{j}(\cdot)}}},\quad j=1,\ldots,m.
\end{equation*}
Thus, using Lemma~\ref{lem:homog} and H\"{o}lder's inequality, Lemma~\ref{lem:Hoelder}, we obtain
\begin{align*}
  &|Q|^{\gamma - \left(\frac{1}{r} - \frac{1}{s}\right)}\Vert \nu_{\vec{w}}\chi_{Q}\Vert_{\frac{1}{\frac{1}{q(\cdot)}-\frac{1}{s}}}
  \prod_{j=1}^{m}\Vert w_j^{-1}\chi_{Q}\Vert_{\frac{1}{\frac{1}{r_j}-\frac{1}{p_j(\cdot)}}}\\
  &= |Q|^{\gamma - \left(\frac{1}{r} - \frac{1}{s}\right)}\Vert \nu_{\vec{w}_0}^{1-\theta}\chi_{Q}\cdot\nu_{\vec{w}_1}^{\theta}\chi_{Q}\Vert_{\frac{1}{\frac{1}{q(\cdot)}-\frac{1}{s}}}
  \prod_{j=1}^{m}\Vert w_{0,j}^{-(1-\theta)}\chi_{Q}\cdot w_{1,j}^{-\theta}\chi_{Q}\Vert_{\frac{1}{\frac{1}{r_j}-\frac{1}{p_j(\cdot)}}}\\
  &\lesssim |Q|^{\gamma - \left(\frac{1}{r} - \frac{1}{s}\right)}
  \Vert \nu_{\vec{w}_0}^{1-\theta}\chi_{Q}\Vert_{\frac{1}{(1-\theta)\left(\frac{1}{q_0(\cdot)}-\frac{1}{s}\right)}}
  \Vert \nu_{\vec{w}_1}^{\theta}\chi_{Q}\Vert_{\frac{1}{\theta\left(\frac{1}{q_1(\cdot)}-\frac{1}{s}\right)}}\\
  &\qquad\cdot\prod_{j=1}^{m}\Vert w_{0,j}^{-(1-\theta)}\chi_{Q}\Vert_{\frac{1}{(1-\theta)\left(\frac{1}{r_j}-\frac{1}{p_{0,j}(\cdot)}\right)}}
  \qquad\cdot \Vert w_{1,j}^{-\theta}\chi_{Q}\Vert_{\frac{1}{\theta\left(\frac{1}{r_j}-\frac{1}{p_{1,j}(\cdot)}\right)}}\\
  & = |Q|^{\gamma - \left(\frac{1}{r} - \frac{1}{s}\right)}
  \Vert \nu_{\vec{w}_0}\chi_{Q}\Vert_{\frac{1}{\frac{1}{q_0(\cdot)}-\frac{1}{s}}}^{1-\theta}
  \Vert \nu_{\vec{w}_1}\chi_{Q}\Vert_{\frac{1}{\frac{1}{q_1(\cdot)}-\frac{1}{s}}}^{\theta}\\
  &\qquad\cdot\prod_{j=1}^{m}\Vert w_{0,j}^{-1}\chi_{Q}\Vert_{\frac{1}{\frac{1}{r_j}-\frac{1}{p_{0,j}(\cdot)}}}^{1-\theta}
  \cdot \Vert w_{1,j}^{-1}\chi_{Q}\Vert_{\frac{1}{\frac{1}{r_j}-\frac{1}{p_{1,j}(\cdot)}}}^{\theta}\\
  & = \left(|Q|^{\gamma - \left(\frac{1}{r} - \frac{1}{s}\right)}
  \Vert \nu_{\vec{w}_0}\chi_{Q}\Vert_{\frac{1}{\frac{1}{q_0(\cdot)}-\frac{1}{s}}}\prod_{j=1}^{m}\Vert w_{0,j}^{-1}\chi_{Q}\Vert_{\frac{1}{\frac{1}{r_j}-\frac{1}{p_{0,j}(\cdot)}}}\right)^{1-\theta}\\
  &\qquad\cdot\left(|Q|^{\gamma - \left(\frac{1}{r} - \frac{1}{s}\right)}
  \Vert \nu_{\vec{w}_1}\chi_{Q}\Vert_{\frac{1}{\frac{1}{q_1(\cdot)}-\frac{1}{s}}}\prod_{j=1}^{m}\Vert w_{1,j}^{-1}\chi_{Q}\Vert_{\frac{1}{\frac{1}{r_j}-\frac{1}{p_{1,j}(\cdot)}}}\right)^{\theta}.
\end{align*}
This yields immediately the desired result.
\end{enumerate}
\end{proof}

\begin{lemma}\label{lem:two_to_one_exp}
Let $(p(\cdot),q(\cdot),r,s)$ be a generalized proper 1-admissible quadruple with
\begin{equation*}
  \frac{1}{p(\cdot)} - \frac{1}{q(\cdot)} = \gamma.
\end{equation*}
Let $w$ be a weight. Then, we have
\begin{equation*}
  w\in\mathcal{A}_{(p(\cdot),q(\cdot)),(r,s)} \Leftrightarrow w^{a} \in \mathcal{A}_{t(\cdot)},
\end{equation*}
where
\begin{equation*}
  a := \frac{1}{\frac{1}{r}-\frac{1}{s}-\gamma}>0,\quad t(\cdot) = \frac{\frac{1}{r}-\frac{1}{s}-\gamma}{\frac{1}{q(\cdot)}-\frac{1}{s}}.
\end{equation*}
In fact,
\begin{equation*}
  [w]_{\mathcal{A}_{(p(\cdot),q(\cdot)),(r,s)}} = [w^{a}]_{\mathcal{A}_{t(\cdot)}}^{1/a}.
\end{equation*}
\end{lemma}

\begin{proof}
First of all, we observe that $t_{-} > 1$. Indeed, pick $x\in(r, p_{-})$. Then, we have
\begin{equation*}
  0<\frac{1}{q(\cdot)} - \frac{1}{s} = \frac{1}{p(\cdot)} - \gamma - \frac{1}{s} 
  < \frac{1}{x} - \gamma - \frac{1}{s}
\end{equation*}
and therefore
\begin{equation*}
  t(\cdot) > \frac{\frac{1}{r} - \frac{1}{s} - \gamma}{\frac{1}{x} - \gamma - \frac{1}{s}} > 1.
\end{equation*}
Let now $Q\subset\R^n$ be any cube. It holds
\begin{equation*}
  at(\cdot) = \frac{1}{\frac{1}{q(\cdot)}-\frac{1}{s}}.
\end{equation*}
Moreover, we compute
\begin{equation*}
  t'(\cdot) = \frac{t(\cdot)}{t(\cdot)-1}
  = \frac{\frac{1}{r}-\frac{1}{s}-\gamma}{\frac{1}{r}-\frac{1}{s}-\gamma - \frac{1}{q(\cdot)} + \frac{1}{s}} = \frac{\frac{1}{r}-\frac{1}{s}-\gamma}{\frac{1}{r} - \frac{1}{p(\cdot)}},
\end{equation*}
therefore
\begin{equation*}
  at'(\cdot) = \frac{1}{\frac{1}{r} - \frac{1}{p(\cdot)}}.
\end{equation*}
Thus, using Lemma \ref{lem:homog} we have
\begin{align*}
  |Q|^{\gamma-\left(\frac{1}{r}-\frac{1}{s}\right)}\Vert w\chi_{Q}\Vert_{\frac{1}{\frac{1}{q(\cdot)}-\frac{1}{s}}}\Vert w^{-1}\chi_{Q}\Vert_{\frac{1}{\frac{1}{r} - \frac{1}{p(\cdot)}}}
  &= |Q|^{-1/a}\Vert w\chi_{Q}\Vert_{at(\cdot)}\Vert w^{-1}\chi_{Q}\Vert_{at'(\cdot)}\\
  &= |Q|^{-1/a}\Vert w^{a}\chi_{Q}\Vert_{t(\cdot)}^{1/a}\Vert w^{-a}\chi_{Q}\Vert_{t'(\cdot)}^{1/a}\\
  &= (|Q|^{-1}\Vert w^{a}\chi_{Q}\Vert_{t(\cdot)}\Vert w^{-a}\chi_{Q}\Vert_{t'(\cdot)})^{1/a}.
\end{align*}
This yields immediately the desired result.
\end{proof}

\begin{lemma}\label{lem:containment}
Let $(\vec{p}(\cdot),q(\cdot),\vec{r},s)$ be a proper $m$-admissible quadruple. Let $\vec{w}$ be a $m$-tuple of weights. Then, we have
\begin{equation*}
  \vec{w}\in\mathcal{A}_{(\vec{p}(\cdot),q(\cdot)),(\vec{r},s)} \Rightarrow \vec{w}\in\mathcal{A}_{(\vec{p}(\cdot),q(\cdot)),(\vec{r},\infty)}
\end{equation*}
and in fact,
\begin{equation*}
  [\vec{w}]_{\mathcal{A}_{(\vec{p}(\cdot),q(\cdot)),(\vec{r},\infty)}}
  \lesssim [\vec{w}]_{\mathcal{A}_{(\vec{p}(\cdot),q(\cdot)),(\vec{r},s)}},
\end{equation*}
where the implicit constant depends only on $q(\cdot)$ and $s$.
\end{lemma}

\begin{proof}
It is clear that $(\vec{p}(\cdot),q(\cdot),\vec{r},\infty)$ is also a proper $m$-admissible quadruple. Let now $Q$ be any cube. Noticing that
\begin{equation*}
  \frac{1}{s} + \frac{1}{\frac{1}{\frac{1}{q(\cdot)}-\frac{1}{s}}}=\frac{1}{q(\cdot)},
\end{equation*}
using H\"{o}lder's inequality, Lemma~\ref{lem:Hoelder}, we compute
\begin{align*}
  &|Q|^{\gamma - \left(\frac{1}{r}-\frac{1}{s}\right)}
  \Vert \nu_{\vec{w}}\chi_{Q}\Vert_{\frac{1}{\frac{1}{q(\cdot)}-\frac{1}{s}}}
  \prod_{j=1}^{m}\Vert w_j^{-1}\chi_{Q}\Vert_{\frac{1}{\frac{1}{r_j}-\frac{1}{p_j(\cdot)}}}\\
  &= |Q|^{\gamma - \frac{1}{r}}\Vert \chi_{Q}\Vert_{s}
  \Vert \nu_{\vec{w}}\chi_{Q}\Vert_{\frac{1}{\frac{1}{q(\cdot)}-\frac{1}{s}}}
  \prod_{j=1}^{m}\Vert w_j^{-1}\chi_{Q}\Vert_{\frac{1}{\frac{1}{r_j}-\frac{1}{p_j(\cdot)}}}\\
  &\gtrsim |Q|^{\gamma - \frac{1}{r}}
  \Vert \chi_{Q}\nu_{\vec{w}}\chi_{Q}\Vert_{q(\cdot)}
  \prod_{j=1}^{m}\Vert w_j^{-1}\chi_{Q}\Vert_{\frac{1}{\frac{1}{r_j}-\frac{1}{p_j(\cdot)}}}\\
  &= |Q|^{\gamma - \frac{1}{r}}
  \Vert\nu_{\vec{w}}\chi_{Q}\Vert_{q(\cdot)}
  \prod_{j=1}^{m}\Vert w_j^{-1}\chi_{Q}\Vert_{\frac{1}{\frac{1}{r_j}-\frac{1}{p_j(\cdot)}}}.
\end{align*}
This yields immediately the desired result.
\end{proof}

The following is the main result of this subsection and Section \ref{sec: interpolation}. It can be thought of as a counterpart of \cite[Corollary 3.7]{CaoOlivoYabuta2022} in the variable exponent setting.

\begin{theorem}\label{thm:interpolation_compactness}
Let $\gamma\in[0,\infty)$ and let $(\vec{p}_i(\cdot), q_i(\cdot), \vec{r}, s)$, $i=0,1$ be proper $m$-admissible quadruples with
\begin{equation*}
  \frac{1}{p_0(\cdot)} - \frac{1}{q_0(\cdot)} = \frac{1}{p_1(\cdot)} - \frac{1}{q_1(\cdot)} = \gamma.
\end{equation*}
Let
\begin{gather*}
  \vec{w}_0 \in \mathcal{A}_{(p_0(\cdot),q_0(\cdot)),(\vec{r}, s)},\\
  \vec{w}_1 \in \mathcal{A}_{(p_1(\cdot),q_1(\cdot)),(\vec{r}, s)}.
\end{gather*}
Let $\theta\in(0,1)$ and define $p_{j}(\cdot) \in \mathscr{P}_{0}$, $q(\cdot)\in\mathscr{P}_{0}$ by
\begin{align*}
  &\frac{1}{p_{j}(\cdot)} := \frac{1-\theta}{p_{0,j}(\cdot)} + \frac{\theta}{p_{1,j}(\cdot)},\quad j=1,\ldots,m,\\
  &\frac{1}{q(\cdot)} := \frac{1-\theta}{q_{0}(\cdot)} + \frac{\theta}{q_{1}(\cdot)}
\end{align*}
as well as the vector of weights $\vec{w}$ by
\begin{equation*}
  w_{j} := w_{0,j}^{1-\theta}w_{1,j}^{\theta},\quad j=1,\ldots,m.
\end{equation*}
Note that by Lemma~\ref{lem:interpolation_weights} we have that $(\vec{p}(\cdot),q(\cdot),\vec{r},s)$ is a proper $m$-admissible quadruple with
\begin{equation*}
  \frac{1}{p(\cdot)} - \frac{1}{q(\cdot)} = \gamma
\end{equation*}
and that $\vec{w}\in\mathcal{A}_{(\vec{p}(\cdot),q(\cdot)),(\vec{r},s)}$.

Let $T$ be a $m$-linear operator such that
\begin{equation*}
  T : L^{p_{0,1}(\cdot)}(w_{0,1})\times\cdots\times L^{p_{0,m}(\cdot)}(w_{0,m}) \to L^{q_0(\cdot)}(\nu_{\vec{w}_0})
\end{equation*}
boundedly and
\begin{equation*}
  T : L^{p_{1,1}(\cdot)}(w_{1,1})\times\cdots\times L^{p_{1,m}(\cdot)}(w_{1,m}) \to L^{q_1(\cdot)}(\nu_{\vec{w}_1})
\end{equation*}
compactly. Then, we have
\begin{equation*}
  T : L^{p_{1}(\cdot)}(w_{1})\times\cdots\times L^{p_{m}(\cdot)}(w_{m}) \to L^{q(\cdot)}(\nu_{\vec{w}})
\end{equation*}
compactly.
\end{theorem}

\begin{proof}
We adapt the strategy of \cite[Theorem 3.6]{CaoOlivoYabuta2022}.
    
Notice that Theorem~\ref{thm:interpolation_boundedness} implies already that
\begin{equation*}
  T : L^{p_{1}(\cdot)}(w_{1})\times\cdots\times L^{p_{m}(\cdot)}(w_{m}) \to L^{q(\cdot)}(\nu_{\vec{w}})
\end{equation*}
boundedly. Define
\begin{equation*}
  \widetilde{q} := \frac{1}{\frac{1}{r} - \gamma}.
\end{equation*}
Observe that by Lemma~\ref{lem:interpolation_weights} we obtain
\begin{equation*}
  \vec{w} \in \mathcal{A}_{(\vec{p}(\cdot),q(\cdot)),(\vec{r},s)},
\end{equation*}
which in view of Lemma~\ref{lem:Apqrs char.} implies
\begin{equation*}
  \nu_{\vec{w}} \in \mathcal{A}_{(p(\cdot),q(\cdot)),(r,s)}.
\end{equation*}
Therefore, by Lemma~\ref{lem:containment} we get
\begin{equation*}
  \nu_{\vec{w}} \in \mathcal{A}_{(p(\cdot),q(\cdot)),(r,\infty)},
\end{equation*}
which with Lemma~\ref{lem:two_to_one_exp} gives
\begin{equation*}
  \nu_{\vec{w}}^{\widetilde{q}} \in \mathcal{A}_{q(\cdot)/\widetilde{q}}.
\end{equation*}
Therefore, to show that this operator $T$ is compact, by Theorem~\ref{thm: RK criterion} it suffices to show the following:
\begin{enumerate}[label=\textup{(\arabic*)}]
  \item
\begin{equation*}
  \sup_{\vec{f}\in\mathcal{F}}\Vert T(\vec{f})\Vert_{L^{q(\cdot)}(\nu_{\vec{w}})} < \infty;
\end{equation*}\label{i.t.m: (1)}
\item
\begin{equation*}
  \lim_{r\to0}\sup_{\vec{f}\in\mathcal{F}}\left\Vert\left(\mint{-}_{B(0,r)}|T(\vec{f})(x)-T(\vec{f})(x+y)|^{\widetilde{q}}\dd y\right)^{1/\widetilde{q}}\right\Vert_{L^{q(\cdot)}(\nu_{\vec{w}})} = 0;
\end{equation*}\label{i.t.m: (2)}
\item
\begin{equation*}
  \lim_{R\to\infty}\sup_{\vec{f}\in\mathcal{F}}\Vert T(\vec{f})\chi_{\R^n\setminus B(0,R)}\Vert_{L^{q(\cdot)}(\nu_{\vec{w}})} = 0,
\end{equation*}\label{i.t.m: (3)}
\end{enumerate}
where
\begin{equation*}
  \mathcal{F} := \{\vec{f}\in L^{p_{1}(\cdot)}(w_{1})\times\cdots\times L^{p_{m}(\cdot)}(w_{m}):~\Vert f_j\Vert_{L^{p_j(\cdot)}{(w_j)}}=1,~j=1,\ldots,m\}.
\end{equation*}
Define $L^{\vec{p}(\cdot)}(\vec{w}) := L^{p_{1}(\cdot)}(w_{1})\times\cdots\times L^{p_{m}(\cdot)}(w_{m})$. Since $T$ is multi-homogeneous, conditions \ref{i.t.m: (1)} to \ref{i.t.m: (3)} are obviously equivalent to the following:
\begin{enumerate}[label=(\arabic*'), start=1]
  \item There exists $M>0$ such that
\begin{equation*}
  \Vert T(\vec{f})\Vert_{L^{q(\cdot)}(\nu_{\vec{w}})} \leq M\prod_{j=1}^{m}\Vert f_j\Vert_{L^{p_j(\cdot)}(w_j)},\quad\forall \vec{f}\in L^{\vec{p}}(\vec{w});
\end{equation*}\label{i.t.m: (1')}

  \item For all $\eps>0$ there exists $r_0=r_0(\eps)>0$, such that for all $r\in(0,r_0)$ we have
\begin{equation*}
  \left\Vert\left(\mint{-}_{B(0,r)}|T(\vec{f})(x)-T(\vec{f})(x+y)|^{\widetilde{q}}\dd y\right)^{1/\widetilde{q}}\right\Vert_{L^{q(\cdot)}(\nu_{\vec{w}})} \leq \eps\prod_{j=1}^{m}\Vert f_j\Vert_{L^{p_j(\cdot)}(w_j)},\quad\forall \vec{f}\in L^{\vec{p}}(\vec{w});
\end{equation*}\label{i.t.m: (2')}

  \item For all $\eps>0$ there exists $R_0=R_0(\eps)>0$, such that for all $R\in(R_0,\infty)$ we have
\begin{equation*}
  \Vert T(\vec{f})\chi_{\R^n\setminus B(0,R)}\Vert_{L^{q(\cdot)}(\nu_{\vec{w}})} \leq \eps\prod_{j=1}^{m}\Vert f_j\Vert_{L^{p_j(\cdot)}(w_j)},\quad\forall \vec{f}\in L^{\vec{p}}(\vec{w}).
\end{equation*}\label{i.t.m: (3')}
\end{enumerate}
Condition \ref{i.t.m: (1')} says just that $T : L^{\vec{p}(\cdot)}(\vec{w}) \to L^{q(\cdot)}(\nu_{\vec{w}})$ is bounded, which follows by Theorem~\ref{thm:interpolation_boundedness}, as noted earlier in the proof. Therefore, we only need to show \ref{i.t.m: (2')} and \ref{i.t.m: (3')}.

Let $\eps>0$ be arbitrary. Set
\begin{equation*}
  M_0 := \Vert T\Vert_{L^{\vec{p}_0(\cdot)}(\vec{w}_0) \to L^{q_0(\cdot)}(\nu_{\vec{w}_0})}<\infty.
\end{equation*}
For all $r>0$, we consider
\begin{equation*}
  N_0(g, r) := \left\Vert\left(\mint{-}_{B(0,r)}|g(x)-g(x+y)|^{\widetilde{q}}\dd y\right)^{1/\widetilde{q}}\right\Vert_{L^{q_0(\cdot)}_{x}(\nu_{\vec{w}_0})}.
\end{equation*}
By combining Lemma~\ref{lem:Apqrs char.} with Lemmata \ref{lem:two_to_one_exp} and \ref{lem:containment} we obtain
\begin{equation}\label{eq:mult. weight}
  \nu_{\vec{w}_0}^{\widetilde{q}} \in \mathcal{A}_{q_0(\cdot)/\widetilde{q}}.
\end{equation}
We claim that
\begin{equation}\label{eq:basic_mixed_bound}
  N_0(g,r) \leq C_0 \Vert g\Vert_{L^{q_0(\cdot)}(\nu_{\vec{w}_0})}.
\end{equation}
Therefore, by the triangle inequality, \eqref{eq: kn. ineq.} and using Lemma~\ref{lem: HL} (since $\nu_{\vec{w}_0}^{\widetilde{q}} \in \mathcal{A}_{q_0(\cdot)/\widetilde{q}}$ by \eqref{eq:mult. weight}) we estimate
\begin{align*}
  N_0(g,r) &\leq C_1\left\Vert\left(\mint{-}_{B(0,r)}|g(x)|^{\widetilde{q}}\dd y\right)^{1/\widetilde{q}}+\left(\mint{-}_{B(0,r)}|g(x+y)|^{\widetilde{q}}\dd y\right)^{1/\widetilde{q}}\right\Vert_{L^{q_0(\cdot)}_{x}(\nu_{\vec{w}_0})}\\
  &\leq C_1\left\Vert |g|+M_{\widetilde{q}}(g)\right\Vert_{L^{q_0(\cdot)}(\nu_{\vec{w}_0})}
  \leq C_2\left\Vert g\right\Vert_{L^{q_0(\cdot)}(\nu_{\vec{w}_0})} + \left\Vert M_{\widetilde{q}}(g)\right\Vert_{L^{q_0(\cdot)}(\nu_{\vec{w}_0})}\\
  &\leq C_0\left\Vert g\right\Vert_{L^{q_0(\cdot)}(\nu_{\vec{w}_0})}.
\end{align*}
This proves \eqref{eq:basic_mixed_bound}.

Now, pick $\delta>0$ such that
\begin{equation*}
  M_0^{1-\theta}\delta^{\theta} < \eps\quad\text{and}\quad (C_0M_0)^{1-\theta}\delta^{\theta} < \eps.
\end{equation*}
Since $T : L^{\vec{p}_1(\cdot)}(\vec{w})\to L^{q_1(\cdot)}(\nu_{\vec{w}_1})$ is compact, by Theorem~\ref{thm: RK criterion} we have that there exists $r_0>0$ such that for all $r\in(0,r_0)$ we have
\begin{equation}\label{eq:cond2_scale1}
  \left\Vert\left(\mint{-}_{B(0,r)}|T(\vec{f})(x)-T(\vec{f})(x+y)|^{\widetilde{q}}\dd y\right)^{1/\widetilde{q}}\right\Vert_{L^{q_1(\cdot)}_x(\nu_{\vec{w}_1})} \leq \delta\prod_{j=1}^{m}\Vert f_j\Vert_{L^{p_{1,j}(\cdot)}(w_{1,j})},\quad\forall \vec{f}\in L^{\vec{p}_1(\cdot)}(\vec{w}_1),
\end{equation}
and furthermore that there exists $R_0>0$ such that for all $R\in(R_0,\infty)$ we have
\begin{equation}\label{eq:cond3_scale1}
  \Vert T(\vec{f})\chi_{\R^n\setminus B(0,R)}\Vert_{L^{q_1(\cdot)}(\nu_{\vec{w}_1})} \leq \delta\prod_{j=1}^{m}\Vert f_j\Vert_{L^{p_{1,j}(\cdot)}(w_{1,j})},\quad\forall \vec{f}\in L^{\vec{p}_1(\cdot)}(\vec{w}_1).
\end{equation}

Let $r\in(0,r_0)$ be arbitrary. We claim that
\begin{equation}\label{eq:cond2}
  \left\Vert\left(\mint{-}_{B(0,r)}|T(\vec{f})(x)-T(\vec{f})(x+y)|^{\widetilde{q}}\dd y\right)^{1/\widetilde{q}}\right\Vert_{L^{q(\cdot)}(\nu_{\vec{w}})} \leq \eps\prod_{j=1}^{m}\Vert f_j\Vert_{L^{p_j(\cdot)}(w_j)},\quad\forall \vec{f}\in L^{\vec{p}(\cdot)}(\vec{w}).
\end{equation}
Indeed, we consider the operator $S$ defined by
\begin{equation*}
  S(\vec{f})(x,y) := T(\vec{f})(x) - T(\vec{f})(x+y),\quad x,y\in\R^n.
\end{equation*}
Observe that $S$ is $m$-linear. For all $\vec{f}\in L^{\vec{p}_{0}(\cdot)}(\vec{w}_0)$, by \eqref{eq:basic_mixed_bound} we compute
\begin{align}\label{eq:cond2_scale0}
  &\left\Vert\left(\mint{-}_{B(0,r)}|S(\vec{f})(x,y)|^{\widetilde{q}}\dd y\right)^{1/\widetilde{q}}\right\Vert_{L^{q_0(\cdot)}_{x}(\nu_{\vec{w}_0})} = N_0(T(\vec{f}), r)\\
  \nonumber&\leq C_0 \left\Vert T(\vec{f})\right\Vert_{L^{q_0(\cdot)}(\nu_{\vec{w}_0})}
  \leq C_0M_0\prod_{j=1}^{m}\Vert f_j\Vert_{L^{p_{0,j}(\cdot)}(w_{0,j})}.
\end{align}
Therefore, combining \eqref{eq:cond2_scale0} and \eqref{eq:cond2_scale1} with Theorem~\ref{thm:interpolation_boundedness_mixed} we obtain \eqref{eq:cond2}.

Finally, let $R\in(R_0,\infty)$ be arbitrary. We claim that
\begin{equation}\label{eq:cond3}
  \Vert T(\vec{f})\chi_{\R^n\setminus B(0,R)}\Vert_{L^{q(\cdot)}(\nu_{\vec{w}})} \leq \eps\prod_{j=1}^{m}\Vert f_j\Vert_{L^{p_j(\cdot)}(w_j)},\quad\forall \vec{f}\in L^{\vec{p}}(\vec{w}).
\end{equation}
Indeed, it is clear that the operator $\chi_{\R^n\setminus B(0,R)}T$ is still $m$-linear and
\begin{equation}\label{eq:cond3_scale0}
  \Vert T(\vec{f})\chi_{\R^n\setminus B(0,R)}\Vert_{L^{q_0(\cdot)}(\nu_{\vec{w}_0})} \leq M_0\prod_{j=1}^{m}\Vert f_j\Vert_{L^{p_{0,j}(\cdot)}(w_{0,j})},\quad\forall\vec{f}\in L^{\vec{p}_0(\cdot)}(\vec{w}_0).
\end{equation}
Thus, combining \eqref{eq:cond3_scale0} and \eqref{eq:cond3_scale1} with Theorem~\ref{thm:interpolation_boundedness} we deduce \eqref{eq:cond3}. This concludes the proof.    
\end{proof}

\section{Proof of abstract extrapolation of compactness principle}\label{sec: pf. abstract results}

This section is devoted to an abstract extrapolation of compactness principle for multilinear operators on weighted variable Lebesgue spaces, namely Theorem \ref{thm:main_result_extr}. To establish it, we employ the weighted interpolation of compactness Theorem~\ref{thm:interpolation_compactness} as well as the following two factorization results that can be found in \cite{extr_comp_var_bil}.

\begin{lemma}[\cite{extr_comp_var_bil}, Lemma 4.1]
\label{lem:KeyLemma}
Let $(\vec{p}(\cdot),q(\cdot),\vec{r},s)$ and $(\vec{p}_1(\cdot),q_1(\cdot),\vec{r},s)$ be proper $m$-admissible quadruples such that
\begin{equation*}
  \frac{1}{p_1(\cdot)}-\frac{1}{q_1(\cdot)}=\frac{1}{p(\cdot)}-\frac{1}{q(\cdot)}=\gamma\in[0,\infty).
\end{equation*}
Moreover, let
\begin{equation*}
  \vec{w}\in\mathcal{A}_{(\vec{p}(\cdot),q(\cdot)),(\vec{r},s)}
\end{equation*}
and
\begin{equation*}
  \vec{w}_1\in\mathcal{A}_{(\vec{p}_1(\cdot),q_1(\cdot)),(\vec{r},s)}.
\end{equation*}
Given any $\theta\in(0,1)$, define $\vec{p}_0(\cdot)\in(\mathscr{P}_0)^m$, $q_0(\cdot)\in\mathscr{P}_0$ by
\begin{align*}
  \frac{1}{p_j(\cdot)}&=\frac{1-\theta}{p_{0,j}(\cdot)}+\frac{\theta}{p_{1,j}(\cdot)},\quad j=1,\ldots,m,\\
  \frac{1}{q(\cdot)}&=\frac{1-\theta}{q_0(\cdot)}+\frac{\theta}{q_1(\cdot)},
\end{align*}
as well as the vector of weights $\vec{w}_0$ by
\begin{align*}
  w_j&=w_{0,j}^{1-\theta}w_{1,j}^{\theta},\quad j=1,\ldots,m.
\end{align*}
Then, there exists $\eta>0$ such that for all $\theta\in(0,\eta)$ all of the following hold:
\begin{enumerate}
  \item $(\vec{p}_0(\cdot), q_0(\cdot), \vec{r}, s)$ is a proper $m$-admissible quadruple with 
\begin{equation*}
  \quad\frac{1}{p_0(\cdot)}-\frac{1}{q_0(\cdot)}=\gamma.
\end{equation*}

  \item There holds
\begin{equation*}
  \vec{w}_0\in\mathcal{A}_{(\vec{p}_0(\cdot),q_0(\cdot)),(\vec{r},s)}.
\end{equation*}

  \item If $(q_1)_{-},q_{-}>1$, then also $(q_{0})_{-}>1$.
\end{enumerate}
\end{lemma}

\begin{lemma}[\cite{extr_comp_var_bil}, Lemma 4.3]
\label{lem:KeyLemma_with_t}
Let $(\vec{p}(\cdot),q(\cdot),\vec{1},\infty)$ and $(\vec{p}_1(\cdot),q_1(\cdot),\vec{1},\infty)$ be proper \\$m$-admissible quadruples with
\begin{equation*}
  \frac{1}{p(\cdot)}-\frac{1}{q(\cdot)}=\frac{1}{p_1(\cdot)}-\frac{1}{q_1(\cdot)}=\gamma\in[0,\infty).
\end{equation*}
Let $t\in(0,\infty)$ be a fixed constant with
\begin{equation*}
  t<\min\{(p_j)_{-}:~j=1,\ldots,m\}
\end{equation*}
and
\begin{equation*}
  t<\min\{(p_{1,j})_{-}:~j=1,\ldots,m\}.
\end{equation*}
Let
\begin{equation*}
  \vec{w}\in\mathcal{A}_{\vec{p}(\cdot),q(\cdot)}
\end{equation*}
and
\begin{equation*}
  \vec{w}_1\in\mathcal{A}_{\vec{p}_1(\cdot),q_1(\cdot)}
\end{equation*}
such that in addition one has
\begin{equation*}
  \vec{w}^{t}\in\mathcal{A}_{\frac{\vec{p}(\cdot)}{t},\frac{q(\cdot)}{t}}
\end{equation*}
and
\begin{equation*}
  \vec{w}_1\in\mathcal{A}_{\frac{\vec{p}_1(\cdot)}{t},\frac{q_1(\cdot)}{t}}.
\end{equation*}
Given any $\theta\in(0,1)$, define $\vec{p}_0(\cdot)\in(\mathscr{P}_0)^m$, $q_0(\cdot)\in\mathscr{P}_0$ by
\begin{align*}
  \frac{1}{p_j(\cdot)}&=\frac{1-\theta}{p_{0,j}(\cdot)}+\frac{\theta}{p_{1,j}(\cdot)},\quad j=1,\ldots,m,\\
  \frac{1}{q(\cdot)}&=\frac{1-\theta}{q_0(\cdot)}+\frac{\theta}{q_1(\cdot)},
\end{align*}
as well as the vector of weights $\vec{w}_0$ by
\begin{align*}
  w_j&=w_{0,j}^{1-\theta}w_{1,j}^{\theta},\quad j=1,\ldots,m.
\end{align*}
Then, there exists $\eta>0$ such that for all $\theta\in(0,\eta)$ all of the following hold:
\begin{enumerate}
  \item $(\vec{p}_0(\cdot),q_0(\cdot),\vec{1},\infty)$ is a proper $m$-admissible quadruple with $\frac{1}{p_0(\cdot)}-\frac{1}{q_0(\cdot)} = \gamma$.\label{itm:no 1}

  \item $\vec{w}_0\in\mathcal{A}_{\vec{p}_0(\cdot),q_0(\cdot)}$.\label{itm:no 2}

  \item If $(q_1)_{-},q_{-}>1$, then also $(q_{0})_{-}>1$.\label{itm:no 3}

  \item  $t<\min\{(p_{0,j})_{-}:~j=1,\ldots,m\}$.

  \item $\vec{w}_0^{t}\in\mathcal{A}_{\frac{\vec{p}_0(\cdot)}{t},\frac{q_0(\cdot)}{t}}$.
\end{enumerate}
\end{lemma}

We are finally in a position to prove the second main theorem of this paper.

\begin{proof}[Proof of Theorem~\ref{thm:main_result_extr}]

We first treat part \ref{eq: main thm 1}. Let $(\vec{p}(\cdot), q(\cdot),\vec{1},\infty)$ be a proper $m$-admissible quadruple such that
\begin{equation*}
  t<\min\{(p_j)_{-}:~j=1,\dots,m\}
\end{equation*}
and
\begin{equation*}
  \frac{1}{p(\cdot)}-\frac{1}{q(\cdot)}=\gamma.
\end{equation*}
Let also $\vec{w}$ be a $m$-tuple of weights with $\vec{w}\in\mathcal{A}_{\vec{p}(\cdot),q(\cdot)}$ and $\vec{w}^{t}\in\mathcal{A}_{\frac{\vec{p}(\cdot)}{t},\frac{q(\cdot)}{t}}$. We will prove that $T:L^{p_1(\cdot)}(w_1)\times\ldots\times L^{p_m(\cdot)}(w_m)\to L^{q(\cdot)}(\nu_{\vec{w}})$ compactly.

By the second assumption of theorem, there exist a proper $m$-admissible quadruple $(\vec{p}_1(\cdot), q_1(\cdot),\vec{1},\infty)$ such that
\begin{equation*}
  t<\min\{(p_{1,j})_{-}:~j=1,\dots,m\}
\end{equation*}
and
\begin{equation*}
  \frac{1}{p_1(\cdot)}-\frac{1}{q_1(\cdot)}=\gamma,
\end{equation*}
as well as a $m$-tuple of weights $\vec{w}_1$ with $\vec{w}_1\in\mathcal{A}_{\vec{p}_1(\cdot),q_1(\cdot)}$ and $\vec{w}_1^{t}\in\mathcal{A}_{\frac{\vec{p}_1(\cdot)}{t},\frac{q_1(\cdot)}{t}}$, such that
\begin{equation}
    \label{eq:ass_comp}
    T:L^{p_{1,1}(\cdot)}(w_{1,1})\times\ldots\times L^{p_{1,m}(\cdot)}(w_{1,m})\to L^{q_1(\cdot)}(\nu_{\vec{w}_1})\quad\text{compactly}.
\end{equation}

Next, applying Lemma~\ref{lem:KeyLemma_with_t} we obtain $\theta>0$ and a proper $m$-admissible quadruple $(\vec{p}_0(\cdot), q_0(\cdot),\vec{1},\infty)$, such that
\begin{equation*}
  t<\min\{(p_0,j)_{-}:~j=1,\dots,m\},
\end{equation*}
\begin{equation*}
\frac{1}{\vec{p}(\cdot)}=\frac{1-\theta}{\vec{p}_0(\cdot)}+\frac{\theta}{\vec{p}_1(\cdot)},\quad \frac{1}{q(\cdot)}=\frac{1-\theta}{q_0(\cdot)}+\frac{\theta}{q_1(\cdot)}
\end{equation*}
and
\begin{equation*}
  \quad\frac{1}{p_0(\cdot)}-\frac{1}{q_0(\cdot)}=\gamma,
\end{equation*}
as well as a $m$-tuple of weights $\vec{w}_0$ with
\begin{equation*}
  \vec{w}=\vec{w}_0^{1-\theta}\vec{w}_1^{\theta}    
\end{equation*}
such that $\vec{w}_0\in\mathcal{A}_{\vec{p}_0(\cdot),q_0(\cdot)}$ and $\vec{w}_0^{t}\in\mathcal{A}_{\frac{\vec{p}_0(\cdot)}{t},\frac{q_0(\cdot)}{t}}$. By the first assumption of the theorem, we have
\begin{equation}
    \label{eq:ass_bound}
    T:L^{p_{0,1}(\cdot)}(w_{0,1})\times\ldots\times L^{p_{0,m}(\cdot)}(w_{0,m})\to L^{q_0(\cdot)}(\nu_{\vec{w}_0})\quad\text{boundedly}.
\end{equation}
In view of~\eqref{eq:ass_bound} and~\eqref{eq:ass_comp}, Theorem~\ref{thm:interpolation_compactness} applies yielding immediately
\begin{equation*}
  T:L^{p_{1}(\cdot)}(w_{1})\times\ldots\times L^{p_{m}(\cdot)}(w_{m})\to L^{q(\cdot)}(\nu_{\vec{w}})\quad\text{compactly}.
\end{equation*}
This shows part \ref{eq: main thm 1}.

Part \ref{eq: main thm 2} is proved similarly, the only difference being the use of Lemma~\ref{lem:KeyLemma} in the place of Lemma~\ref{lem:KeyLemma_with_t}. We omit the details.
\end{proof}

\section{Applications}\label{sec: applic.}

In this section, we describe several applications of Theorem~\ref{thm:main_result_extr}. Since much of the preparatory work has already been carried out in Section 7 of the paper \cite{extr_comp_var_bil} by the authors, we will be briefer at the corresponding places and refer the reader to the aforementioned work for details.

\subsection{Notation for multilinear commutators}

We employ some standard notation for commutators of multilinear operators. Namely, given some $m$-linear operator $T$ accepting as arguments $m$-tuples $\vec{f}=(f_1,\ldots,f_m)$ of appropriate functions on $\R^n$ and some appropriately smooth function $b$ on $\R^n$, we define for each $j=1,\ldots,m$ the commutator
\begin{equation*}
    [T,b]_{e_j}(\vec{f})(x) :=b(x)T(\vec{f})(x)-T(f_1, \ldots,bf_j,\ldots,f_m)(x),\quad x\in\R^n.
\end{equation*}
Moreover, given a $m$-tuple $\vec{b}=(b_1,\ldots,b_m)$ of appropriately smooth functions on $\R^n$, we consider the \emph{sum commutator}
\begin{equation*}
    [T,\vec{b}]_{\Sigma}(\vec{f}) := \sum_{j=1}^{m}[T,b_j]_{e_j}(\vec{f}).
\end{equation*}

\subsection{The symbol of the commutators}

Typically, the symbol $b$ of our commutators will belong to the class
\begin{equation*}
    \CMO(\R^n):=\overline{C_c^\infty(\R^n)}^{\BMO(\R^n)},
\end{equation*}
where $C_c^\infty(\R^n)$ is the space of smooth functions with compact supports and $\BMO(\R^n)$ is the space of bounded mean oscillation functions defined by
\begin{equation*}
    \BMO(\R^n):=\Big\{f:\R^n\to\C\ \Big|\ \Norm{f}{\BMO(\R^n)}:=\sup_{Q\subset\R^n}\ave{\abs{f-\ave{f}_Q}}_Q<\infty\Big\}.
\end{equation*}

It turns out that this is the most natural choice in order to be able to apply our abstract extrapolation Theorem~\ref{thm:main_result_extr}.

\subsection{Multilinear \texorpdfstring{$\omega$}{ω}-Calder\'on--Zygmund operators}\label{subsec:CZ}

Let $\omega:[0,\infty)\rightarrow[0,\infty)$ be a modulus of continuity, namely, $\omega$ is increasing, subadditive and satisfies $\omega(0)=0$. Consider a function $K(x,y_1,\ldots,y_m)$ defined on $(\R^n)^{m+1}\setminus\{x=y_1=\ldots=y_m\}$. We say that $K$ is an \emph{$\omega$-Calder\'{o}n--Zygmund kernel} if it fulfills the following \emph{size and smoothness conditions}:  
\begin{itemize}
  \item
\begin{equation*}
  |K(x,y_1,\ldots,y_m)|\lesssim \frac{1}{(\sum_{j=1}^m|x-y_j|)^{mn}},
\end{equation*} 
  \item
\begin{equation*} 
  |K(x,y_1,\ldots,y_m)-K(x',y_1,\ldots,y_m)|\lesssim \frac{1}{(\sum_{j=1}^m|x-y_j|)^{mn}}\omega\bigg(\frac{|x-x'|}{\sum_{j=1}^m|x-y_j|}\bigg),
\end{equation*}     
whenever $|x-x'|\leq\frac{1}{2}\max_{j=1,\ldots,m}|x-y_j|$, and for each $i=1,\dots,m,$
  \item
\begin{equation*}
  |K(x,y_1,\ldots,y_i,\ldots,y_m)-K(x,y_1,\ldots,y_i',\ldots,y_m)| \lesssim \frac{1}{(\sum_{j=1}^m|x-y_j|)^{mn}}\omega\bigg(\frac{|y_i-y_i'|}{\sum_{j=1}^m|x-y_j|}\bigg),
\end{equation*}
whenever $|y_i-y_i'|\leq\frac{1}{2}\max_{j=1,\ldots,m}|x-y_j|$.
\end{itemize}

Now, we say that the $m$-linear operator $T:\mathcal{S}(\R^n)\times\cdots\times\mathcal{S}(\R^n)\rightarrow\mathcal{S}'(\R^n)$ is an $\omega$-Calder\'on--Zygmund operator if it extends to a bounded multilinear operator from $L^{q_1}(\R^n)\times\cdots\times L^{q_m}(\R^n)$ to $L^q(\R^n)$ for some $\frac{1}{q}=\frac{1}{q_1}+\cdots+\frac{1}{q_m}$ with $1<q_1,\dots,q_m<\infty$, and there exists an $\omega$-Calder\'on--Zygmund kernel $K$ such that
\begin{equation*}
  T(\vec{f})(x) := \int_{(\R^n)^m}K(x,y_1,\ldots,y_m)f(y_1)\cdots f(y_m)\,\dd y_1\cdots\dd y_m,
\end{equation*}
for all $x\notin\cap_{j=1}^m\supp f_j$, $j=1,\dots,m$.

We say that a modulus of continuity $\omega$ satisfies the Dini condition, denoted by $\omega\in$ Dini, if 
\begin{equation*}
  \|\omega\|_{\text{Dini}}:=\int_0^1\omega(t)\frac{\dd t}{t}<\infty.
\end{equation*}
We refer the reader to \cite{Yabuta1985, GT2002, MN2009} for further discussions concerning the Dini condition.

In \cite[Theorem 5.1]{CaoOlivoYabuta2022}, it was shown that commutators of $m$-linear $\omega$-Calder\'on--Zygmund operators are compact on constant exponent weighted Lebesgue spaces. Additionally, the authors of \cite[Theorem 7.5]{extr_comp_var_bil} obtained the compactness of commutators of bilinear $\omega$-Calder\'on--Zygmund operators on weighted variable Lebesgue spaces and under the restriction that the target space is a Banach space. Theorem~\ref{thm:main_result_extr} allows us to extend this to the multilinear setting and remove the restriction on the target space.

To achieve this, we first need the following boundedness result for the commutators obtained by the authors in \cite[Theorem 7.4]{extr_comp_var_bil}, which had built upon ideas from \cite[Theorem 2.4]{CHSW2025}.

\begin{theorem}[\cite{extr_comp_var_bil}, Theorem 7.4]\label{thm:bound_mul_CZO}
Let $b\in\mathrm{BMO}(\R^n)$. Let $t$ be a fixed constant with $t>1$. Let $j\in\{1,\ldots,m\}$. Then, for all proper $m$-admissible quadruples $(\vec{p}(\cdot), q(\cdot),\vec{1},\infty)$ with $q(\cdot)=p(\cdot)$ and
\begin{equation*}
  t<\min\{(p_k)_{-}:~k=1,\ldots,m\},
\end{equation*}
and for all $m$-tuples of weights $\vec{w}$ with $\vec{w}\in\mathcal{A}_{\vec{p}(\cdot)}$ and $\vec{w}^{t}\in\mathcal{A}_{\frac{\vec{p}(\cdot)}{t}}$, it holds that $[T,b]_{e_j}$ maps $L^{p_1(\cdot)}(w_1)\times\cdots\times L^{p_m(\cdot)}(w_m)$ into $L^{p(\cdot)}(\nu_{\vec{w}})$ boundedly.
\end{theorem}

Next, we also need the following unweighted compactness result, which had already been obtained as part of the proof of \cite[Theorem 5.1]{CaoOlivoYabuta2022}.

\begin{theorem}[\cite{CaoOlivoYabuta2022}]\label{thm:comp_muly_CZO}
Let $b\in\mathrm{CMO}(\R^n)$, $1<p_1,\ldots,p_m<\infty$ and $0<p<\infty$ with $\frac{1}{p}:=\sum_{j=1}^{m}\frac{1}{p_j}$. Then, $[T,b]_{e_j}$ maps $L^{p_1}(\R^n)\times\cdots\times L^{p_m}(\R^n)$ into $L^{p}(\R^n)$ compactly for every $j=1,\ldots,m$.    
\end{theorem}

Now, we can state and prove the main result of this subsection.

\begin{theorem}\label{thm:main_result_CZO}
Let $\vec{b}=(b_1,b_2,\dots,b_m)$ be a $m$-tuple of functions in $\mathrm{CMO}(\R^n)$ and $t$ be a fixed constant with $t>1$. Then, for all proper $m$-admissible quadruples $(\vec{p}(\cdot), q(\cdot),\vec{1},\infty)$ with $q(\cdot)=p(\cdot)$,
\begin{equation*}
    t<\min\{(p_j)_{-}:~j=1,\dots,m\},
\end{equation*}
and for all $m$-tuples of weights $\vec{w}$ with $\vec{w}\in\mathcal{A}_{\vec{p}(\cdot)}$ and $\vec{w}^{t}\in\mathcal{A}_{\frac{\vec{p}(\cdot)}{t}}$, it holds that $[T,b_j]_{e_j}$ for each $j=1,\dots,m$ as well as $[T,\vec{b}]_{\Sigma}$ map $L^{p_1(\cdot)}(w_1)\times\dots\times L^{p_m(\cdot)}(w_m)$ into $L^{p(\cdot)}(\nu_{\vec{w}})$ compactly.
\end{theorem}

\begin{proof}
We apply part \ref{eq: main thm 1} of Theorem~\ref{thm:main_result_extr}. The class $\Theta$ we will be considering consists of all pairs $(\vec{Y},Y)$ of the form
\begin{align*}
    &Y_j = L^{p_j(\cdot)}(w_j),\quad j=1,\ldots,m,\\
    &Y = L^{q(\cdot)}(\nu_{\vec{w}}),
\end{align*}
where $(\vec{p}(\cdot),q(\cdot),\vec{1},\infty)$ is a proper $m$-admissible quadruple with
\begin{equation*}
    t<\min\{(p_{j})_{-}:~j=1,\ldots,m\}
\end{equation*}
and
\begin{equation*}
    \frac{1}{p(\cdot)}-\frac{1}{q(\cdot)}=\gamma = 0,
\end{equation*}
and $\vec{w} = (w_1,\ldots,w_m)$ is a $m$-tuple of weights such that $\vec{w}\in\mathcal{A}_{\vec{p}(\cdot)}$ and $\vec{w}^{t}\in\mathcal{A}_{\frac{\vec{p}(\cdot)}{t}}$.

Fix $(\vec{Z},Z)\in\Theta$. First of all, observe that the sum of multilinear compact operators remains compact per \cite[Proposition 2]{BenTor2013}. While the proof given there considered formally only bilinear operators, it holds with minor changes in the general multilinear setting. Thus, we only need to show that $[T,b_j]_{e_j}$ maps $Z_1\times\ldots\times Z_m$ into $Z$ compactly for each $j=1,\ldots,m$.

Fix $j\in\{1,\ldots,m\}$. The boundedness assumption of part \ref{eq: main thm 1} of Theorem~\ref{thm:main_result_extr} for the operator $[T,b_j]_{e_j}$ is clearly just the conclusion of Theorem~\ref{thm:bound_mul_CZO}.
    
Pick now a constant $d>t$. Then, Theorem~\ref{thm:comp_muly_CZO} yields that the compactness assumption of part \ref{eq: main thm 1} of Theorem~\ref{thm:main_result_extr} is satisfied for the operator $[T,b_j]_{e_j}$ with the particular choices
\begin{align*}
    &X_j = L^{p_{1,j}(\cdot)}(w_{1,j}),\quad j=1,\ldots,m,\\
    &X = L^{q_1(\cdot)}(\nu_{\vec{w}_1}),
\end{align*}
where
\begin{equation*}
    \vec{p}_1(\cdot):=(d,\ldots,d),\quad p_1(\cdot):=\frac{d}{m},\quad\text{and}\quad \vec{w}_1:=(1,\ldots,1).
\end{equation*}

Therefore, applying part \ref{eq: main thm 1} of Theorem~\ref{thm:main_result_extr} we conclude the proof.
\end{proof}

We conclude this subsection by noticing that there are various examples of multilinear $\omega$-Calder\'on--Zygmund operators (see \cite[Subsection 5.1]{CaoOlivoYabuta2022}) that satisfy the assumptions and the conclusions of Theorem \ref{thm:main_result_CZO}.

\subsection{Multilinear fractional Calder\'on--Zygmund operators}\label{subsec:fractional}

Let $m\geq2$ and $\alpha\in(0,mn)$. Consider a function $K_{\alpha}(x,y_1,\ldots,y_m)$ defined on $(\R^n)^{m+1}\setminus\{x=y_1=\ldots=y_m\}$. We say that $K_{\alpha}$ is an \emph{$m$-linear $\alpha$-fractional Calder\'{o}n--Zygmund kernel} if there exist constants $A,B>0$, such that the following \emph{size} and \emph{smoothness} estimates are satisfied:
\begin{itemize}
  \item 
\begin{equation*}
  |K_{\alpha}(x,y_1,\ldots,y_m)| \leq \frac{A}{(|x-y_1|+\cdots+|x-y_m|)^{mn-\alpha}},
\end{equation*}
  \item
\begin{equation*}
    |K_{\alpha}(x,y_1,\ldots,y_m)-K_{\alpha}(x',y_1,\ldots,y_m)| \leq \frac{B|x-x'|}{(|x-y_1|+\cdots+|x-y_m|)^{mn-\alpha+1}},
\end{equation*}
whenever $|x-x'|\leq\frac{1}{2}\max_{j=1,\ldots,m}|x-y_j|$,
\item For each $j=1,\ldots,m$, one has
\begin{equation*}
    | K_{\alpha}(x,y_1,\ldots,y_j,\ldots,y_m)-K_{\alpha}(x,y_1,\ldots,y_j',\ldots,y_m)| \leq \frac{B|y_j-y_j'|}{(|x-y_1|+\cdots+|x-y_m|)^{mn-\alpha+1}},
\end{equation*}
whenever $|y_j-y_j'|\leq\frac{1}{2}\max_{k=1,\ldots,m}|x-y_k|$.
\end{itemize}

As noticed in \cite{BDMT2015} in the special case $m=2$, given such a kernel $K_{\alpha}$, the $m$-linear operator $T_{\alpha}$ defined by
\begin{equation*}
     T_{\alpha}(\vec{f})(x) := \int_{(\R^n)^m}K_{\alpha}(x,y_1,\ldots,y_m)f_1(y_1)\cdots f_m(y_m)\,\dd y_1\cdots\dd y_m,\quad x\in\R^n
\end{equation*}
is well defined whenever $f_j$, $j=1,\ldots,m$ are for example bounded functions with compact support on $\R^n$. Combining results from \cite{KenigStein1999, BDMT2015}, one can see that $T_{\alpha}$ extends to a bounded $m$-linear operator $L^{p_1}(\R^n)\times\cdots\times L^{p_m}(\R^n)\to L^{q}(\R^n)$ whenever $1<p_1,\ldots,p_m<\infty$, $0<q<\infty$ and $\frac{\alpha}{n}=\frac{1}{p}-\frac{1}{q}$, where $\frac{1}{p}:=\sum_{j=1}^{m}\frac{1}{p_j}$.

The compactness of commutators of $T_{\alpha}$ on unweighted and constant exponent weighted Lebesgue spaces had already been studied in \cite{BDMT2015, CaoOlivoYabuta2022}. Another related result appeared in \cite[Theorem 7.9]{extr_comp_var_bil}, where the compactness of commutators of bilinear fractional Calder\'on--Zygmund operators on weighted variable Lebesgue spaces was established under the assumption that the target space is Banach. Theorem~\ref{thm:main_result_extr} extends the latter result to the multilinear setting without imposing any restriction on the target space.

The first ingredient for doing so is the general boundedness result for the commutators established by the authors in \cite[Theorem 7.8]{extr_comp_var_bil}, which had adapted itself ideas from \cite[Theorem 2.4]{CHSW2025}.

\begin{theorem}[\cite{extr_comp_var_bil}, Theorem 7.8]
\label{thm:bound_mult_frac}
Let $b\in\mathrm{BMO}(\R^n)$. Set $\gamma:=\frac{\alpha}{n}$. Let $t$ be a fixed constant with $1<t<\frac{m}{\gamma}$. Let $j\in\{1,\ldots,m\}$. Then, for all proper $m$-admissible quadruples $(\vec{p}(\cdot), q(\cdot),\vec{1},\infty)$ with
\begin{equation*}
    t<\min\{(p_k)_{-}:~k=1,\ldots,m\}
\end{equation*}
and
\begin{equation*}
    \frac{1}{p(\cdot)}-\frac{1}{q(\cdot)}=\gamma,
\end{equation*}
and for all $m$-tuples of weights $\vec{w}$ with $\vec{w}\in\mathcal{A}_{\vec{p}(\cdot),q(\cdot)}$ and $\vec{w}^{t}\in\mathcal{A}_{\frac{\vec{p}(\cdot)}{t},\frac{q(\cdot)}{t}}$, it holds that $[T_{\alpha},b]_{e_j}$ maps $L^{p_1(\cdot)}(w_1)\times\cdots\times L^{p_m(\cdot)}(w_m)$ into $L^{q(\cdot)}(\nu_{\vec{w}})$ boundedly.
\end{theorem}

The second ingredient concerns compactness on unweighted constant exponent Lebesgue spaces. Such a result is the unweighted version of \cite[Theorem 5.8]{CaoOlivoYabuta2022}. While \cite[Theorem 5.8]{CaoOlivoYabuta2022} is formally stated only for the Riesz potential, the authors of \cite{CaoOlivoYabuta2022} observe that their statement and proof hold for arbitrary $T_{\alpha}$.

\begin{theorem}[\cite{BDMT2015, CaoOlivoYabuta2022}]\label{thm:comp_mult_frac}
Let $b\in\mathrm{CMO}(\R^n)$. Let $1<p_1,\ldots,p_m<\infty$ and let $0<q<\infty$ such that $\frac{1}{p}-\frac{1}{q}=\frac{\alpha}{n}$, where $\frac{1}{p}:=\sum_{j=1}^{m}\frac{1}{p_j}$. Then, $[T_{\alpha},b]_{e_j}$ maps $L^{p_1}(\R^n)\times\cdots\times L^{p_m}(\R^n)$ into $L^{q}(\R^n)$ compactly for every $j=1,\ldots,m$.
\end{theorem}

Now, we can state and prove the main result of this subsection.

\begin{theorem}
\label{thm:compact_frac_cz}
Let $\vec{b}=(b_1,\ldots,b_m)$ be a $m$-tuple of functions in $\mathrm{CMO}(\R^n)$. Set $\gamma:=\frac{\alpha}{n}$. Let $t$ be a fixed constant with $1<t<\frac{m}{\gamma}$. Then, for all proper $m$-admissible quadruples $(\vec{p}(\cdot),q(\cdot),\vec{1},\infty)$ with
\begin{equation*}
    t<\min\{(p_j)_{-}:~j=1,\ldots,m\}
\end{equation*}
and
\begin{equation*}
    \frac{1}{p(\cdot)}-\frac{1}{q(\cdot)}=\gamma
\end{equation*}
and for all $m$-tuples of weights $\vec{w}$ with $\vec{w}\in\mathcal{A}_{\vec{p}(\cdot),q(\cdot)}$ and $\vec{w}^{t}\in\mathcal{A}_{\frac{\vec{p}(\cdot)}{t},\frac{q(\cdot)}{t}}$, it holds that $[T_{\alpha},b_j]_{e_j}$ for each $j=1,\ldots,m$ as well as $[T_{\alpha},\vec{b}]_{\Sigma}$ map $L^{p_1(\cdot)}(w_1)\times\ldots\times L^{p_m(\cdot)}(w_m)$ into $L^{q(\cdot)}(\nu_{\vec{w}})$ compactly.
\end{theorem}

\begin{proof}
We apply part \ref{eq: main thm 1} of Theorem~\ref{thm:main_result_extr}. The class $\Theta$ we will be considering consists of all pairs $(\vec{Y},Y)$ of the form
\begin{align*}
    &Y_j = L^{p_j(\cdot)}(w_j),\quad j=1,\ldots,m,\\
    &Y = L^{q(\cdot)}(\nu_{\vec{w}}),
\end{align*}
where $(\vec{p}(\cdot),q(\cdot),\vec{1},\infty)$ is a proper $m$-admissible quadruple with
\begin{equation*}
    t<\min\{(p_{j})_{-}:~j=1,\ldots,m\}
\end{equation*}
and
\begin{equation*}
    \frac{1}{p(\cdot)}-\frac{1}{q(\cdot)}=\gamma = \frac{\alpha}{n},
\end{equation*}
and $\vec{w} = (w_1,\ldots,w_m)$ is a $m$-tuple of weights such that $\vec{w}\in\mathcal{A}_{\vec{p}(\cdot),q(\cdot)}$ and $\vec{w}^{t}\in\mathcal{A}_{\frac{\vec{p}(\cdot)}{t},\frac{q(\cdot)}{t}}$.

Fix $(\vec{Z},Z)\in\Theta$. Similarly to the proof of Theorem \ref{thm:main_result_CZO}, we only have to show that for all $j=1,\ldots,m$, $[T_{\mathfrak{m}},b_j]_{e_j}$ maps $Z_1\times\ldots\times Z_m$ into $Z$.

Fix $j\in\{1,\ldots,m\}$. The boundedness assumption of part \ref{eq: main thm 1} of Theorem~\ref{thm:main_result_extr} for the operator $[T_{\alpha},b_j]_{e_j}$ is clearly just the conclusion of Theorem~\ref{thm:bound_mult_frac}.
    
Pick now a constant
\begin{equation*}
    d \in \left(t,\frac{m}{\gamma}\right).
\end{equation*}
Such a choice is possible, because $\gamma<\frac{m}{t}$ by assumption. Then, Theorem~\ref{thm:comp_mult_frac} yields that the compactness assumption of part \ref{eq: main thm 1} of Theorem~\ref{thm:main_result_extr} is satisfied for the operator $[T_{\alpha},b_j]_{e_j}$ with the particular choices
\begin{align*}
    &X_j = L^{p_{1,j}(\cdot)}(w_{1,j}),\quad j=1,\ldots,m,\\
    &X = L^{q_1(\cdot)}(\nu_{\vec{w}_1}),
\end{align*}
where
\begin{equation*}
    \vec{p}_1(\cdot):=(d,\ldots,d),\quad q_1(\cdot):=\frac{1}{\frac{m}{d}-\gamma},\quad\text{and}\quad \vec{w}_1:=(1,\ldots,1).
\end{equation*}

Therefore, applying part \ref{eq: main thm 1} of Theorem~\ref{thm:main_result_extr} we conclude the proof.
\end{proof}

\subsection{Multilinear Fourier multipliers}\label{subsec:fourier}

Let $s\in\N$ and $m\geq 2$. Let $\mathfrak{m}\in\mathrm{C}^{s}(\R^{nm}\setminus\{0\})$ be a bounded function and let $\Phi$ be a Schwarz function on $\R^{nm}$ satisfying the conditions
\begin{equation*}
    \mathrm{supp}(\Phi) \subset \left\{(\xi_1,\ldots,\xi_{m})\in(\R^{n})^{m}:~\frac{1}{2}\leq\sum_{j=1}^{m}|\xi_j|\leq 2\right\}
\end{equation*}
and
\begin{equation*}
    \sum_{k\in\Z}\Phi(2^{-j}\xi)=1,\quad\forall\; \xi = (\xi_1,\ldots,\xi_m)\in\R^{nm}\setminus\{0\}.
\end{equation*}
We consider the usual Sobolev space $W^{s}(\R^{nm})$ equipped with the norm
\begin{equation*}
    \Vert f\Vert_{W^{s}(\R^{nm})} := \left(\int_{\R^{nm}}(1+|\xi|^2)^{s}|\widehat{f}(\xi)|^2\,\dd \xi\right)^{1/2},
\end{equation*}
where $\widehat{f}$ denotes the Fourier transform of $f$ on all the variables in $\R^{nm}$. Define
\begin{equation*}
    \mathfrak{m}_{k}(\xi) := \Phi(\xi)\mathfrak{m}(2^{k}\xi),\quad \xi\in\R^{nm}
\end{equation*}
for each $k\in\Z$. We say that $\mathfrak{m}\in\mathcal{W}^{s}(\R^{nm})$ if
\begin{equation*}
    \Vert\mathfrak{m}\Vert_{\mathcal{W}^{s}(\R^{nm})} := \sup_{k\in\Z}\Vert\mathfrak{m}_{k}\Vert_{W^{s}(\R^{nm})}<\infty.
\end{equation*}
Now, let $T_{\mathfrak{m}}$ denote the $m$-linear Fourier multiplier associated to the symbol $\mathfrak{m}$, defined by
\begin{equation*}
    T_{\mathfrak{m}}(f_1,\ldots,f_m)(x) := \int_{(\R^n)^m}\mathfrak{m}(\xi)e^{2\pi i x\cdot(\xi_1+\cdots+\xi_m)}\widehat{f}_1(\xi_1)\cdots\widehat{f}_m(\xi_m)\,\dd \xi,\quad x\in\R^n
\end{equation*}
for Schwarz functions $f_1,\ldots,f_m$ on $\R^n$.

The compactness properties of the commutators $[T_{\mathfrak{m}},b]_{e_j}$, $j=1,\ldots,m$ on constant exponent weighted Lebesgue spaces were studied in \cite{Hu2017, CaoOlivoYabuta2022}. 
Moreover, the compactness of commutators of bilinear Fourier multipliers on weighted variable Lebesgue spaces was obtained in \cite[Theorem 7.15]{extr_comp_var_bil}, provided that the target space is Banach. Theorem~\ref{thm:main_result_extr} generalizes this to the multilinear setting without requiring the target space to be Banach.

To do this, we first need the general boundedness result for the commutators established by the authors in \cite[Theorem 7.13]{extr_comp_var_bil}, which had adapted in turn some ideas from \cite{Li_Sun_2013}.

\begin{theorem}[\cite{extr_comp_var_bil}, Theorem 7.13]
\label{thm:bound_FourMult_com}
Let $b\in\mathrm{BMO}(\R^n)$. Assume $\frac{mn}{2}<s\leq mn$ and $\mathfrak{m}\in\mathcal{W}^{s}(\R^{nm})$. Then, there exists a constant $p_0\in\left(\frac{mn}{s},\infty\right)$ depending only on $m,n$ and $s$, such that the following holds. For all $r_0\in(p_0,\infty)$, for all proper $m$-admissible quadruples $(\vec{p}(\cdot), q(\cdot),\vec{r},\infty)$ with $q(\cdot)=p(\cdot)$ and $r_j := r_0$, $j = 1, \ldots, m$, and for all $\vec{w}\in\mathcal{A}_{\vec{p}(\cdot),(\vec{r},\infty)}$, we have that $[T_{\mathfrak{m}},b]_{e_j}$ maps $L^{p_1(\cdot)}(w_1)\times\cdots\times L^{p_m(\cdot)}(w_m)$ into $L^{p(\cdot)}(\nu_{\vec{w}})$ boundedly, for all $j=1,\ldots,m$.
\end{theorem}

Second, we need some information about the compactness on unweighted constant exponent Lebesgue spaces. Such a result is the unweighted version of \cite[Theorem 5.9]{CaoOlivoYabuta2022}.

\begin{theorem}[{\cite[Theorem 5.9]{CaoOlivoYabuta2022}}]
\label{thm:comp_FourMult_com_unw}
Assume $\frac{mn}{2}<s\leq mn$ and $\mathfrak{m}\in\mathcal{W}^{s}(\R^{nm})$. Let $1\leq t_1,\ldots, t_m <2$ with
\begin{equation*}
    \frac{s}{n} = \sum_{j=1}^{m}\frac{1}{t_j}.
\end{equation*}
Let $b\in\mathrm{CMO}(\R^{n})$. Then, for all $1<p_1,\ldots,p_m<\infty$ with $t_j<p_j$ for each $j=1,\ldots,m$ and
\begin{equation*}
    \frac{1}{p} := \sum_{j=1}^{m}\frac{1}{p_j},
\end{equation*}
we have that $[T_{\mathfrak{m}},b]_{e_j}$ maps $L^{p_1}(\R^n)\times\cdots\times L^{p_m}(\R^n)$ into $L^{p}(\R^n)$ compactly for each $j=1,\ldots,m$.
\end{theorem}

Now, we can state and prove the main result of this subsection.

\begin{theorem}
\label{thm:compact_FourMult_com}
Let $\vec{b}=(b_1,\ldots,b_m)$ be a $m$-tuple of function in $\mathrm{CMO}(\R^n)$. Assume $\frac{mn}{2}<s\leq mn$ and $\mathfrak{m}\in\mathcal{W}^{s}(\R^{mn})$. Consider the constant $p_0\in\left(\frac{mn}{s},\infty\right)$ depending only on $m$, $n$ and $s$ from Theorems~\ref{thm:bound_FourMult_com}. Then, the following holds. For all $r_0\in(p_0,\infty)$, for all proper $m$-admissible quadruples $(\vec{p}(\cdot), q(\cdot),\vec{r},\infty)$ with $q(\cdot)=p(\cdot)$, $r_j := r_0$, $j = 1, \ldots,m$, and for all $\vec{w}\in\mathcal{A}_{\vec{p}(\cdot),(\vec{r},\infty)}$, we have that $[T_{\mathfrak{m}},b_j]_{e_j}$ maps $L^{p_1(\cdot)}(w_1)\times\ldots\times L^{p_m(\cdot)}(w_m)$ into $L^{p(\cdot)}(\nu_{\vec{w}})$ compactly, for all $j=1,\ldots,m$. In particular, $[T_{\mathfrak{m}},\vec{b}]_{\Sigma}$ maps $L^{p_1(\cdot)}(w_1)\times\ldots\times L^{p_m(\cdot)}(w_m)$ into $L^{p(\cdot)}(\nu_{\vec{w}})$ compactly.
\end{theorem}

\begin{proof}
We apply part \ref{eq: main thm 2} of Theorem~\ref{thm:main_result_extr}. Fix $r_0\in(p_0,\infty)$. The class $\Theta$ we will be considering consists of all pairs $(\vec{Y},Y)$ of the form
\begin{align*}
    &Y_j = L^{p_j(\cdot)}(w_j),\quad j=1,\ldots,m,\\
    &Y = L^{q(\cdot)}(\nu_{\vec{w}}),
\end{align*}
where $(\vec{p}(\cdot),q(\cdot),\vec{r},\infty)$ is a proper $m$-admissible quadruple with $q(\cdot)=p(\cdot)$,
\begin{equation*}
    r_j := r_0 > p_0 > \frac{mn}{s} \geq 1,\quad j=1,\ldots,m,
\end{equation*}
and $\vec{w} = (w_1,\ldots,w_m)$ is a $m$-tuple of weights such that $\vec{w}\in\mathcal{A}_{(\vec{p}(\cdot),q(\cdot)),(\vec{r},\infty)}=\mathcal{A}_{\vec{p}(\cdot),(\vec{r},\infty)}$. Observe that in the notation of part \ref{eq: main thm 2} of Theorem~\ref{thm:main_result_extr}, for this particular class we have $\gamma=0$.

Fix $(\vec{Z},Z)\in\Theta$. Similarly to the proofs of Theorems~\ref{thm:main_result_CZO} and~\ref{thm:compact_frac_cz}, we only have to show that for all $j=1,\ldots,m$, $[T_{\mathfrak{m}},b_j]_{e_j}$ maps $Z_1\times\ldots\times Z_m$ into $Z$.

Fix $j\in\{1,\ldots,m\}$. The boundedness assumption of part \ref{eq: main thm 2} of Theorem~\ref{thm:main_result_extr} for $[T_{\mathfrak{m}},b]_{e_j}$ is clearly just the content of Theorem~\ref{thm:bound_FourMult_com}.
    
We check the compactness assumption of part \ref{eq: main thm 2} of Theorem~\ref{thm:main_result_extr}. Set
\begin{equation*}
    t_j := \frac{mn}{s},\quad j=1,\ldots,m.
\end{equation*}
Then, we have $1\leq t_j<2$, $j=1,\ldots,m$ and
\begin{equation*}
    \sum_{j=1}^{m}\frac{1}{t_j} = \frac{s}{n}.
\end{equation*}
Pick now $d > r_0$.
Then, Theorem~\ref{thm:comp_FourMult_com_unw} yields that the compactness assumption of Theorem~\ref{thm:main_result_extr} holds for the operator $[ T_{m},b_j]_{e_j}$ with the particular choices
\begin{align*}
    &X_j = L^{p_{1,j}(\cdot)}(w_{1,j}),\quad j=1,\ldots,m,\\
    &X = L^{q_1(\cdot)}(\nu_{\vec{w}_1}),
\end{align*}
where
\begin{equation*}
    \vec{p}_1(\cdot):=(d,\ldots,d),\quad q_1(\cdot) := \frac{d}{m},\quad\text{and}\quad \vec{w}_1:=(1,\ldots,1).
\end{equation*}

Therefore, applying part \ref{eq: main thm 2} of Theorem~\ref{thm:main_result_extr} we conclude the proof.
\end{proof}

\section*{Declarations}

\subsection*{Author Contributions} This has been a joint collaboration in all senses.

\subsection*{Funding}
S. L. was partly supported by the grant 26-21107S of the Czech Science Foundation and by the ERC Synergy grant HALF 101224275.

\subsection*{Data Availability}
Not applicable.

\subsection*{Competing Interests} The authors declare no competing interests.

\subsection*{Ethical Approval}
Not applicable.

\subsection*{AI disclosure statement} No AI tools were used.

\printbibliography

\end{document}